\documentclass [11pt,reqno]{amsart}
\usepackage {amsmath,amssymb,verbatim,geometry,color,esint}
\usepackage[all]{xy}
\usepackage[pdftex,hyperindex]{hyperref}
\usepackage[pdftex]{graphicx}
\usepackage{tikz,tikz-cd}
\usetikzlibrary{math}

\def\XXint#1#2#3{{\setbox0=\hbox{$#1{#2#3}{\int}$ }
\vcenter{\hbox{$#2#3$ }}\kern-.6\wd0}}

\newcommand{\A}{\mathbb{A}}

\newcommand{\C}{\mathbb{C}}
\newcommand{\DD}{\mathbb{D}}

\newcommand{\Gm}{\mathbb{G}_{\mathrm{m}}}
\newcommand{\N}{\mathbb{N}}
\renewcommand{\P}{\mathbb{P}}
 \newcommand{\Q}{\mathbb{Q}}
 \newcommand{\R}{\mathbb{R}}
 \newcommand{\Z}{\mathbb{Z}}

\newcommand{\fa}{\mathfrak{a}}

\newcommand{\cC}{\mathcal{C}}

\newcommand{\cE}{\mathcal{E}}

\newcommand{\cH}{\mathcal{H}}

\newcommand{\cL}{\mathcal{L}}
\newcommand{\cM}{\mathcal{M}}
\newcommand{\cN}{\mathcal{N}}
\newcommand{\cO}{\mathcal{O}}

\newcommand{\cR}{\mathcal{R}}
\newcommand{\cT}{\mathcal{T}}

\newcommand{\cX}{\mathcal{X}}

\newcommand{\tcL}{\widetilde{\mathcal{L}}}

\newcommand{\tcX}{\widetilde{\mathcal{X}}}

\newcommand{\Xan}{X^{\mathrm{an}}}

\newcommand{\Xdiv}{X^{\mathrm{div}}}

\newcommand{\Xlin}{X^{\mathrm{lin}}}

\newcommand{\Xval}{X^{\mathrm{val}}}

\newcommand{\Yan}{Y^{\mathrm{an}}}

\newcommand{\pian}{\pi^{\mathrm{an}}}

\renewcommand{\a}{\alpha}

\renewcommand{\d}{\delta}
\newcommand{\e}{\varepsilon}
\newcommand{\f}{\varphi}

\newcommand{\la}{\lambda}

\newcommand{\Ga}{\Gamma}
\newcommand{\La}{\Lambda}
\newcommand{\p}{\psi}

\newcommand{\cf}{{\rm cf.\ }} 
\newcommand{\eg}{{\rm e.g.\ }} 
\newcommand{\ie}{{\rm i.e.\ }} 
\newcommand{\loccit}{\textit{loc.\,cit.}}

\renewcommand{\hom}{\mathrm{hom}}

\newcommand{\inte}{\mathrm{int}}

\newcommand{\ba}{\textbf{e}}

\renewcommand{\DH}{\mathrm{DH}}
\newcommand{\an}{\mathrm{an}}

\DeclareMathOperator{\en}{E}
\DeclareMathOperator{\ten}{\widetilde{E}}
\DeclareMathOperator{\tee}{T}
\DeclareMathOperator{\ess}{S}

\DeclareMathOperator{\ii}{I}
\DeclareMathOperator{\jj}{J}
\DeclareMathOperator{\pp}{P}

\DeclareMathOperator{\qq}{Q}

\DeclareMathOperator{\Cz}{C^0}

\DeclareMathOperator{\charac}{char}

\DeclareMathOperator{\CPSH}{CPSH}
\DeclareMathOperator{\PL}{PL}

\DeclareMathOperator{\FS}{FS}
\DeclareMathOperator{\FSstar}{FS^{\star}}
\DeclareMathOperator{\IN}{IN}

\DeclareMathOperator{\Hom}{Hom}
\DeclareMathOperator{\Hnot}{H^0}

\DeclareMathOperator{\MA}{MA}

\DeclareMathOperator{\Spec}{Spec}
\DeclareMathOperator{\supp}{supp}
\DeclareMathOperator{\vol}{vol}

\DeclareMathOperator{\id}{id}

\DeclareMathOperator{\ord}{ord}

\DeclareMathOperator{\Psef}{Psef}
\DeclareMathOperator{\Proj}{Proj}
\DeclareMathOperator{\PSH}{PSH}

\DeclareMathOperator{\gr}{gr}

\DeclareMathOperator{\dd}{{d}}
\DeclareMathOperator{\quotd}{\underline{d}}
\DeclareMathOperator{\bii}{\overline{I}}

\newcommand{\n}{\chi}

\renewcommand{\div}{\mathrm{div}}

\newcommand{\triv}{\mathrm{triv}}
\newcommand{\val}{\mathrm{val}}
\newcommand{\lin}{\mathrm{lin}}

\newcommand{\cont}{\mathrm{cont}}

\newcommand{\Filt}{{F}}

\newcommand{\D}{\Delta}

\newcommand{\cro}[1]{[\![#1]\!]}

\newcommand{\hto}{\hookrightarrow}
\newcommand{\simto}{\overset\sim\to}

\numberwithin{equation}{section}       

\newtheorem{prop} {Proposition} [section]
\newtheorem{thm}[prop] {Theorem} 
\newtheorem{defi}[prop] {Definition}
\newtheorem{lem}[prop] {Lemma}
\newtheorem{cor}[prop]{Corollary}
\newtheorem{prop-def}[prop]{Proposition-Definition}

\newtheorem*{thmA}{Theorem A} 
 
\newtheorem*{thmB}{Theorem B} 
\newtheorem*{thmC}{Theorem C} 
 
\newtheorem*{thmD}{Theorem D}

\newtheorem*{corE}{Corollary E}

\newtheorem{exam}[prop]{Example}
\newtheorem{rmk}[prop]{Remark}

\theoremstyle{remark}
\newtheorem*{ackn}{Acknowledgment}

\title[A non-Archimedean approach to K-stability I]{A non-Archimedean approach to K-stability, I:\\ Metric geometry of spaces of test configurations and valuations}
\date{\today}

\author{S{\'e}bastien Boucksom \and Mattias Jonsson}

\address{Sorbonne Universit\'e and Universit\'e Paris Cit\'e\\
  CNRS, IMJ-PRG\\
  F-75005 Paris\\
  France}
\email{sebastien.boucksom@imj-prg.fr}
\address{Dept of Mathematics\\
  University of Michigan\\
  Ann Arbor, MI 48109-1043\\
  USA}
\email{mattiasj@umich.edu}

\begin{document}

\begin{abstract} For any polarized variety $(X,L)$, we show that test configurations and, more generally, $\R$-test configurations (defined as finitely generated filtrations of the section ring) can be analyzed in terms of Fubini--Study functions on the Berkovich analytification of $X$ with respect to the trivial absolute value on the ground field. Building on non-Archimedean pluripotential theory, we describe the (Hausdorff) completion of the space of test configurations, with respect to two natural pseudo-metrics, in terms of plurisubharmonic functions and measures of finite energy on the Berkovich space. We also describe the Hausdorff quotient of the space of all filtrations, and establish a 1--1 correspondence between divisorial norms and divisorial measures, both being determined in terms of finitely many divisorial valuations.   
\end{abstract}

\maketitle

\setcounter{tocdepth}{1}
\tableofcontents
%
%
%
%
\section*{Introduction}\label{sec:intro}
The notion of K-stability was introduced in complex differential geometry as a conjectural---and now partially confirmed---algebro-geometric criterion for the existence of special K\"ahler metrics. Lately, it has also become a subject in its own respect.
In a series of two (largely independent) papers of which this is the first, we show how global pluripotential theory over a trivially valued field, as developed in~\cite{trivval}, can be used to study K-stability.

Let $X$ be a projective variety (reduced and irreducible) of dimension $n\ge1$ over an algebraically closed field $k$   (assumed to be of characteristic $0$ in this introduction)  , and $L$ an ample $\Q$-line bundle on $X$. The definition of K-stability of the polarized variety $(X,L)$, as given by Donaldson~\cite{Dontoric}, involves the sign of an invariant attached to (ample) test configurations for $(X,L)$. As shown in~\cite{BHJ1,trivval}, test configurations can be alternatively understood in terms of \emph{(rational) Fubini--Study functions} on the Berkovich analytification $X^\an$, and uniform K-stability becomes a linear growth condition for the \emph{non-Archimedean K-energy} on the set of such functions. 

Filtrations of the section ring of $(X,L)$ provide another, widely used description of test configurations; more precisely, the latter correspond to $\Z$-filtrations of finite type~\cite{WN12,Sze}. Recent works related to the Hamilton--Tian conjecture~\cite{CSW,DS,HL,BLXZ} have emphasized the importance of considering more general \emph{$\R$-test configurations}, defined as $\R$-filtrations of finite type, and one first objective of this paper is to show that these can again be understood as (real) Fubini--Study functions on $X^\an$. 

On the other hand, Chi Li's recent breakthrough on the Yau--Tian--Donaldson conjecture for cscK metrics~\cite{Li22b} (based in part on the first version of the present paper) involves a stronger form of uniform K-stability, formulated as a linear growth condition for the K-energy on the space of \emph{functions of finite energy} on $X^\an$. The latter are obtained as limits of Fubini--Study functions, and are the central topic of pluripotential theory on $X^\an$ \cite{trivval}. Building on the latter technology, the second objective of this paper is to show how functions and measures of finite energy can be used to describe the completion of the space of test configurations with respect to a natural metric, leading to a picture that is quite similar to the well-developed complex analytic case~\cite{Dar15,DLR20}. 

%
%
\subsection*{From test configurations to Fubini--Study functions}\label{sec:tcFS}
Denote by $\cN_\R$ the space of (decreasing, left-continuous, separated, exhaustive) \emph{filtrations} of the section algebra $R^{(d)}=R(X,dL)$ for $d$ sufficiently divisible. It is convenient to view these as \emph{norms} $\chi\colon R^{(d)}\to\R\cup\{+\infty\}$, for which we use `additive' terminology, see~\S\ref{sec:norms}. A norm $\chi\in\cN_\R$ is of \emph{finite type} if the associated graded algebra is finitely generated. For any subgroup $\La\subset\R$, let $\cN_\La\subset\cN_\R$ be the set of norms with values in $\La\cup\{+\infty\}$, and denote by
$$
\cT_\R\subset\cN_\R\quad\text{and}\quad\cT_\La:=\cT_\R\cap\cN_\La
$$
the subsets of norms of finite type. 

As is by now well-known (see~\cite{WN12,Sze,BHJ1}), the Rees construction provides a 1--1 correspondence between $\cT_\Z$ and the set of (ample) test configurations for $(X,L)$. In line with~\cite{DS}, we view $\cT_\R$ as the space of \emph{$\R$-test configurations}. Any $\n\in\cT_\R$ lies in $\cT_\La$ for some finitely generated subgroup $\La\simeq\Z^r$,    and $\chi$ can be geometrically realized as a $\Gm^r$-equivariant degeneration of $(X,L)$ to a polarized scheme (see~\S\ref{sec:Inoue}), which is further reduced iff $\n$ is \emph{homogeneous}, in the sense that $\n(s^d)=d\n(s)$ for all $d\in\N$.   

The space $\cN_\R$ comes equipped with a non-decreasing family $(\dd_p)_{1\le p\le\infty}$ of natural pseudo-metrics. By~\cite{BE}, the space $\cN_\R(V)$ of norms on any finite dimensional vector space $V$ is indeed endowed  with a metric $\dd_p$ for any $p\in [1,\infty]$, the distance between two norms being the $\ell^p$-length of their \emph{relative spectrum}, defined by joint diagonalization in some basis. For $p=2$, this is the classical Tits metric of $\cN_\R(V)$ as a Euclidean building, whose relevance to K-stability was already emphasized in~\cite{Oda15,Cod19}. 

Any $\chi\in\cN_\R$ restricts to a norm on $R_m:=\Hnot(X,mL)$ for all $m$ sufficiently divisible, and we define the pseudometric $\dd_p$ on $\cN_\R$ by setting
\begin{equation*}
  \dd_p(\chi,\chi'):=\limsup_m m^{-1} \dd_p(\chi|_{R_m},\chi'|_{R_m}),
 \end{equation*}
where the limsup is actually a limit for $p<\infty$, by~\cite{BC,CM15,BHJ1}. The $L^p$-norm of a test configuration in $\cT_\Z$, as in~\cite{Dontoric,Sze,BHJ1}, can be computed using $\dd_p$. 

\medskip

The pseudo-metric $\dd_p$ is not a metric, even after restriction to $\cT_\La$, and our first main result describes the Hausdorff quotient of $(\cT_\La,\dd_p)$ as a natural space of functions on the \emph{Berkovich analytification} of $X$ (with respect to the trivial absolute value on $k$). Recall that the latter is a compact Hausdorff topological space $X^\an$, whose elements are \emph{semivaluations} $v$ on $X$, \ie $\R$-valued valuations on the function field of some subvariety of $X$, trivial on $k$. The space $X^\an$ contains as a dense subset the space $X^\div$ of \emph{divisorial valuations} on $X$, induced (up to scaling) by a prime divisor on a birational model of $X$. 

For any $v\in\Xan$ and any section $s$ of a line bundle on $X$, we can define $v(s)\in [0,+\infty]$ by trivializing the line bundle at the center of $v$, and setting $|s|(v):=e^{-v(s)}$ defines a continuous function $|s|\colon\Xan\to[0,1]$. Given a subgroup $\La\subset\R$, a \emph{$\La$-Fubini--Study function} for $L$ is a function $\f\in\Cz=\Cz(X^\an)$ of the form
\begin{equation*}
  \f=\tfrac1m\max_j\{\log|s_j|+\la_j\},
\end{equation*}
where $m\ge 1$ is sufficiently divisible, $(s_j)$ is a finite set of sections of $mL$ without common zeroes, and $\la_j\in\La$. 

The set $\cH_\La\subset\Cz$ of $\La$-Fubini--Study functions is stable under max and the action of $\La$ by translation, and satisfies $\cH_{\La}=\cH_{\Q\La}$. It is related to the space $\cT_\La$ of $\La$-test configurations by the \emph{Fubini--Study operator}, a surjective map 
$$
\FS\colon\cT_\La\to\cH_\La
$$ 
that associates to $\n\in\cT_\La$ the Fubini--Study function $\FS(\chi)=m^{-1}\max_j\{\log|s_j|+\chi(s_j)\}$, where $(s_j)$ is any $\chi$-orthogonal basis of $R_m$ for $m$ sufficiently divisible. Viewed as a map from (usual) test configurations to Fubini--Study functions, $\FS\colon\cT_\Z\to\cH_\Z=\cH_\Q$ is compatible with the one constructed and studied in~\cite{BHJ1,trivval} (see Appendix~\ref{sec:tc}). 

\begin{thmA} For any polarized variety $(X,L)$, any subgroup $\La\subset\R$ and $p\in[1,\infty]$, the Fubini--Study operator identifies the Hausdorff quotient of the pseudo-metric space $(\cT_\La,\dd_p)$ with $\cH_\La$. For $p=\infty$, the induced metric $\dd_\infty$ on $\cH_\La$ further coincides with the supnorm metric. 
\end{thmA} 

It is enough to prove this for $\La=\R$. Let us first describe the case $p=\infty$. The restrictions $\chi|_{R_d}$ of any norm $\chi\in\cN_\R$ generate a sequence of \emph{canonical approximants} $\n_d\in\cT_\R$, which allows us to extend the Fubini--Study operator to a map 
$$
\FS\colon\cN_\R\to\cL^\infty
$$ 
into the space of bounded functions on $X^\an$, by setting $\FS(\chi):=\lim_d \FS(\n_d)$. On the other hand, any $\f\in\cL^\infty$ defines an \emph{infimum norm} $\IN(\f)\in\cN_\R$, the avatar of the usual supnorm $\sup_{X^\an}|s|e^{-m\f}$ on $R_m$ in our additive terminology. This defines an operator 
$$
\IN\colon\cL^\infty\to\cN_\R^\hom
$$ 
into the space of homogeneous norms. Using standard but nontrivial results in non-Archimedean geometry, we show that:
\begin{itemize}
\item the composition $\IN\circ\FS\colon\cN_\R\to\cN_\R^\hom$ coincides with the \emph{homogenization} operator    $\n\mapsto\n^\hom$, where $\n^\hom(s)=\lim_{r\to\infty} r^{-1}\n(s^r)$,    which corresponds to the spectral radius construction in the usual `multiplicative' terminology;
\item homogenization preserves the finite type condition, and hence maps $\cT_\R$ onto $\cT_\R^\hom:=\cT_\R\cap\cN_\R^\hom$; 
\item on $\cT_\R$, both the Fubini--Study operator and the pseudo-metric $\dd_\infty$ factor through homogenization.\end{itemize}
These results imply that $\FS\colon (\cT_\R,\dd_\infty)\to(\cH_\R,\dd_\infty)$ is a surjective isometry, which restricts to an isometric isomorphism $(\cT_\R^\hom,\dd_\infty)\simeq(\cH_\R,\dd_\infty)$, where $\dd_\infty$ on $\cH_\R$ is the supnorm metric; this settles Theorem~A for $p=\infty$. 

\medskip 

For any $p\in[1,\infty)$, we have $\dd_1\le\dd_p\le\dd_\infty$ as pseudo-metrics on $\cN_\R$. By the previous step, the restriction of $\dd_\infty$ to $\cT_\R$ factors through the Fubini--Study operator. Thus $\dd_p|_{\cT_\R}$ descends to a pseudo-metric on $\cH_\R$, and Theorem~A asserts that it is a metric, \ie that it separates points. It is enough to prove this for $p=1$, which is accomplished via an explicit expression for $\dd_1$ in terms of the Monge--Amp\`ere energy, analogous to the known expression for the Darvas metric in the complex analytic case~\cite{Dar17}. 

Our approach is based on the close relation of the $\dd_1$-pseudometric on $\cN_\R$ to the \emph{volume} of a norm $\chi\in\cN_\R$, defined as the limit 
$$
\vol(\chi)=\lim_m m^{-1}\vol(\chi|_{R_m})\in\R,
$$ 
where $\vol(\chi|_{R_m})$ is the barycenter of the spectrum of $\chi|_{R_m}$. Indeed, for all $\chi,\chi'\in\cN_\R$, we have
\begin{equation*}
\dd_1(\chi,\chi')=\vol(\chi)+\vol(\chi')-2\vol(\chi\wedge\chi'),
\end{equation*}
with $\chi\wedge\chi'\in\cN_\R$ the pointwise min of $\chi$ and $\chi'$. When $\chi\in\cT_\Z$ corresponds to a test configuration $(\cX,\cL)\to\A^1$, with canonical compactification $(\bar\cX,\bar\cL)\to\P^1$, it was proved in~\cite{BHJ1}, using the Riemann--Roch formula, that 
$$
\vol(\chi)=\frac{(\bar{\cL}^{n+1})}{(n+1)(L^n)},
$$
where the right-hand side is also, by definition, the \emph{Monge--Amp\`ere energy} $\en(\f)$ of $\f:=\FS(\chi)\in\cH_\Q$. Setting $\ten(\f):=\sup\{\en(\p)\mid \f\ge\p\in\cH_\Q\}$ defines the extended energy functional $\ten\colon\Cz\to\R$, and an approximation argument based on Okounkov bodies leads to the key formula $\vol(\n)=\ten(\FS(\n))$, which implies 
\begin{equation*}
\dd_1(\n,\n')=\en(\f)+\en(\f')-2\ten(\f\wedge\f')
\end{equation*}
for all $\n,\n'\in\cH_\R$, where $\f=\FS(\n)$, $\f'=\FS(\n')$. The right-hand side thus defines a pseudo-metric $\dd_1$ on $\cH_\R$, and a result of~\cite{trivval} allows us to show that it separates points, thereby finishing the proof of Theorem~A.    
This formula also characterizes the metric $\dd_1$ on $\cH_\R$ as the unique one such that $
\dd_1(\f,\f')=\inf_{\f''\le\f,\f'}\{\dd_1(\f,\f'')+\dd_1(\f'',\f')\}$ for all $\f,\f'\in\cH_\R$, and $\dd_1(\f,\f')=\en(\f)-\en(\f')$ when $\f\ge\f'$.

%
%
\subsection*{Darvas metrics on functions and measures of finite energy}\label{sec:E1}
By Theorem~A, the Hausdorff completion of $(\cT_\R,\dd_p)$ can be identified with the completion of the metric space $(\cH_\R,\dd_p)$. When $p=\infty$, this is simply the closure of $\cH_\R\subset\Cz$ in the topology of uniform convergence, which is, by definition, the space $\CPSH$ of \emph{continuous $L$-psh functions} (in line with~\cite{Zha95,GublerLocal}).    

For a norm $\n\in\cN_\R$, $\FS(\n)$ lies in $\CPSH$ as soon as it is continuous (by Dini's lemma); we show that the set $\cN_\R^\cont\subset\cN_\R$ of such norms coincides with the $\dd_\infty$-closure in $\cN_\R$ of the set $\cT_\Z$ of (ample) test configurations, and that it is a strict subset (except in the trivial case $\dim X=0$, see~\S\ref{sec:cont}).

Our next goal is to describe the completion of $(\cH_\R,\dd_1)$. The answer relies on global pluripotential theory over a trivially valued field, as developed in~\cite{trivval} (inspired in part by the discretely valued case studied in~\cite{siminag}). Let us briefly describe the salient points of this theory. 

Inspired by the complex analytic case, we define an \emph{$L$-psh function} $\f\colon\Xan\to\R\cup\{-\infty\}$ as an upper semicontinuous (usc) function that can be written as the limit of a decreasing sequence (or net) in $\cH_\R$ (or $\cH_\Z=\cH_\Q$), excluding $\f\equiv-\infty$. Such functions are uniquely determined by their restrictions to $X^\div\subset X^\an$, which are further finite valued, and we equip the space $\PSH$ of $L$-psh functions with the topology of pointwise convergence on $\Xdiv$. By Dini's Lemma, the space of continuous $L$-psh functions $\CPSH$ considered above can be described as $\CPSH=\PSH\cap\Cz$. 

The Monge--Amp\`ere energy $\en$ admits a unique usc, monotone increasing extension 
$$
\en\colon\PSH\to\R\cup\{-\infty\},
$$
given by $\en(\f)=\inf\{\en(\p)\mid \f\le\p\in\CPSH\}$, and the space of \emph{$L$-psh functions of finite energy} is defined as 
$$
\cE^1:=\{\en>-\infty\}\subset\PSH. 
$$
A function in $\cE^1$ is thus a decreasing limit of functions in $\cH_\Q$ with bounded energy. The space $\cE^1$ is endowed with the \emph{strong topology}, defined as the coarsest refinement of the subspace topology from $\PSH\supset\cE^1$ for which $\en\colon\cE^1\to\R$ is continuous. Any decreasing net in $\cE^1$ is strongly convergent, and $\cH_\Q$ is thus dense in $\cE^1$ in the strong topology. 

To each $\f\in\cE^1$ is associated a \emph{Monge--Amp\`ere measure} $\MA(\f)$, a (Radon) probability measure on $X^\an$ that integrates functions in $\cE^1$. When $\f\in\cH_\Q$, $\MA(\f)$ has finite support in $X^\div$, and can be described using intersection numbers computed on the central fiber of an associated test configuration. The Monge--Amp\`ere operator $\f\mapsto\MA(\f)$ is continuous on $\cE^1$ in the strong topology, and characterized as the derivative of $\en$, \ie
$$
\frac{d}{dt}\bigg|_{t=0}\en\left((1-t)\f+t\p\right)=\int_{X^\an}(\p-\f)\MA(\f)
$$
for all $\f,\p\in\cE^1$. It takes its values in the space $\cM^1$ of \emph{measures of finite energy}, \ie Radon probability measures $\mu$ on $\Xan$ for which the Legendre transform 
\begin{equation*}
  \en^\vee(\mu):=\sup_{\f\in\cE^1}\{\en(\f)-\int\f\,d\mu\}\in[0,+\infty] 
\end{equation*}
is finite. In analogy to the complex analytic case~\cite{BBGZ}, the variational approach of~\cite{trivval} shows that $\mu=\MA(\f)$ with $\f\in\cE^1$ iff $\f$ achieves the supremum that defines $\en^\vee(\mu)$. 

The space $\cM^1$ also comes with a strong topology, the coarsest refinement of the weak topology of measures such that $\en^\vee\colon\cM^1\to\R_{\ge 0}$ is continuous. A key result of~\cite{trivval} shows that the Monge--Amp\`ere operator induces a topological embedding with dense image 
$$
\MA\colon\cE^1/\R\hto\cM^1
$$ 
(with respect to the strong topologies), which is further onto iff \emph{the envelope property} holds for $(X,L)$. The latter important property has several equivalent formulations, including the compactness of the quotient space $\PSH/\R$ (a fundamental fact in the setting of compact complex manifolds); it is established when $X$ is smooth, using multiplier ideals, and we conjecture that it holds as long as $X$ is normal (or merely unibranch, which is in turn a necessary condition). 

   The Monge--Amp\`ere operator naturally induces a map $\MA\colon\cT_\R\to\cM^1$ by setting $\MA(\n):=\MA(\FS(\n))$; as mentioned above, when $\n\in\cT_\Z$, the measure $\MA(\n)$ has finite support in $X^\div$, and can be directly described in terms of intersection numbers on (the integral closure of) the test configuration corresponding to $\n$.   

With these preliminaries in hand, we can now state: 

\begin{thmB} For any polarized variety $(X,L)$, the following holds: 
\begin{itemize}
\item[(i)] there exists a unique metric $\dd_1$ on $\cE^1$ that defines the strong topology and extends the metric $\dd_1$ on $\cH_\R\subset\cE^1$;
\item[(ii)] there exists a unique metric $\dd_1$ on $\cM^1$ that defines the strong topology and induces the quotient metric of $\dd_1$ on $\cE^1/\R\hookrightarrow\cM^1$;
\item[(iii)] the metric space $(\cM^1,\dd_1)$ is always complete, while $(\cE^1,\dd_1)$ is complete iff the envelope property holds for $(X,L)$;   
\item[(iv)] the Monge--Amp\`ere operator $\MA\colon\cT_\R\to\cM^1$ uniquely extends to an isometry 
$$
\MA\colon(\cN_\R/\R,\quotd_1)\to(\cM^1,\dd_1),
$$
where $\quotd_1$ denotes the quotient pseudometric of $\dd_1$. 
   
\end{itemize}
\end{thmB}
   In particular, the Monge--Amp\`ere operator realizes $\cM^1$ as the Hausdorff completion of $(\cT_\R/\R,\quotd_1)$, while $\cE^1$ is the Hausdorff completion of $(\cT_\R,\dd_1)$ iff the envelope property holds, \eg when $X$ is smooth (see also~\cite{DX} for an approach based on geodesic rays, when $k=\C$).   

We call the metric $\dd_1$ on $\cE^1$ the \emph{Darvas metric}; its complex analytic analogue, introduced by T.~Darvas~\cite{Dar17}, plays a crucial role in global pluripotential theory, and in particular in the variational approach to the Yau--Tian--Donaldson conjecture~\cite{YTD,Li22a,Li22b}. The space $\cE^1$ is studied over more general non-Archimedean fields in~\cite{Reb22}, where it is shown that $(\cE^1,\dd_1)$ is a geodesic metric space (assuming the envelope property). In analogy with the complex analytic case~\cite{Dar15,DLR20}, we expect that, for any $p\in[1,\infty)$, the completion of $(\cH_\R,\dd_p)$ can be identified with the space 
$$
\cE^p:=\left\{\f\in\cE^1\mid\f\in L^p(\MA(\f))\right\},
$$
assuming the envelope property. 

\medskip
Among other things, the proof of Theorem~B is based on a precise comparison between $\dd_1$ and quasi-metrics on $\cE^1$ and $\cM^1$ studied in~\cite{trivval}, using estimates that ultimately derive from the Hodge Index Theorem.     By construction, $\MA\colon(\cT_\R/\R,\quotd_1)\to(\cM^1,\dd_1)$ is an isometry, and (iv) is thus a consequence of (iii) and the $\dd_1$-density of $\cT_\Z$ in $\cN_\R$, which we prove using Okounkov bodies (see Corollary~\ref{cor:tcdense}).

If $\n\in\cN_\R^\cont$ is a continuous norm, then $\FS(\n)\in\CPSH\subset\cE^1$, and $\MA(\n)=\MA(\FS(\n))$. If the envelope property holds for $(X,L)$, then the usc regularization $\FSstar(\n)$ lies in $\cE^1$ for any norm $\n\in\cN_\R$, and $\MA(\n)=\MA(\FSstar(\n))$. In this case, we get a surjective isometry $\FSstar\colon(\cN_\R,\dd_1)\rightarrow(\cE^\infty_\uparrow,\dd_1)$, where $\cE^\infty_\uparrow$ is the set of $L$-psh functions that are \emph{regularizable from below}, \ie limits in $\PSH$ of an increasing net in $\CPSH$. This realizes $\cE^\infty_\uparrow$ as the Hausdorff quotient of $\cN_\R$. We emphasize, however, that~(iv) is valid even without assuming the envelope property for $(X,L)$.

Finally, we show that the functional $\chi\mapsto\|\n\|:=\en^\vee(\MA(\chi))$ on $\cN_\R$ extends (up to a normalization constant) the \emph{minimum norm} of a test configuration in the sense of Dervan~\cite{Der}.

%
%

\subsection*{Divisorial norms and maximal norms}
The set $\Xval$ of valuations on the function field of $X$, trivial on $k$, is a dense subset of $\Xan$. Following~\cite{BKMS} we say that $v\in\Xval$ is \emph{of linear growth} if there exists $C>0$ such that $v(s)\le Cm$ for all nonzero sections $s\in R_m=\Hnot(X,mL)$ with $m$ sufficiently divisible. In terms of pluripotential theory, the set $\Xlin\subset\Xval$ of valuations of linear growth coincides with the set of points $v\in\Xan$ that are non-pluripolar, \ie such that every $\f\in\PSH$ is finite at $v$; in particular, it contains the set $X^\div$ of divisorial valuations. 

Any $v\in\Xlin$ defines a (homogeneous) norm $\chi_v\in\cN^\hom_\R$, simply by setting $\chi_v(s):=v(s)$. We say that a norm $\n\in\cN_\R$ is \emph{divisorial} if it is of the form $\n=\min_i\{\n_{v_i}+c_i\}$ for a finite set $(v_i)$ in $X^\div$ and $c_i\in\R$. We denote by $\cN_\R^\div$ the set of divisorial norms, and by $\cN_\Q^\div:=\cN^\div_\R\cap\cN_\Q$ the subset of \emph{rational divisorial norms}, for which the $c_i$ can be chosen in $\Q$. The latter contains the homogenization $\n^\hom$ of any ample test configuration $\n\in\cT_\Z$, and $\cN^\div_\Q$ can alternatively be described in terms of norms associated to (possibly non-ample) test configurations (see Theorem~\ref{thm:normint}). 

On the other hand, we define a \emph{divisorial measure} as a Radon probability measure $\mu$ on $X^\an$ with support a finite subset of $X^\div$, \ie $\mu=\sum_i m_i \d_{v_i}$ for a finite subset $(v_i)$ of $X^\div$ and $m_i\in\R_{>0}$ such that $\sum_i m_i=0$. The set $\cM^\div\subset\cM^1$ of divisorial measures is thus the convex hull of the image of the canonical embedding $X^\div\hto\cM^1$  $v\mapsto\d_v$. For any test configuration $\n\in\cT_\Z$, the norm $\n^\hom$ and the measure $\MA(\n)=\MA(\n^\hom)$ are both divisorial. More generally, we show: 

\begin{thmC} The Monge--Amp\`ere operator induces an isometric isomorphism 
$$
\MA\colon(\cN_\R^\div/\R,\quotd_1)\simto(\cM^\div,\dd_1).
$$
\end{thmC} 
We emphasize that the envelope property is not assumed here. In the companion paper~\cite{nakstab2}, divisorial measures are used to define the notion of \emph{divisorial stability}, which implies (and is conjecturally equivalent to) uniform K-stability.  Theorem~C enables us to view divisorial stability as a condition on divisorial norms, and leads to the equivalence between divisorial stability and uniform K-stability with respect to norms/filtrations. 

The proof of Theorem~C is based on the variational approach to (non-Archimedean) Monge--Amp\`ere equations developed in~\cite{trivval}, recast in terms of norms.

\medskip
Recall that the space of norms $\cN_\R$ is equipped with pseudometrics $(\dd_p)_{p\in [1,\infty]}$, such that $\dd_1\le\dd_p\le\dd_\infty$. For $\n,\n'\in\cN_\R$, the condition $\dd_p(\n,\n')=0$ is independent of $p<\infty$; we say that $\n$ and $\n'$ are \emph{asymptotically equivalent} when this holds. While $\dd_\infty$ becomes a metric upon restriction to the space $\cN_\R^\hom$ of homogeneous norms, this is still not the case for $\dd_p$ with $p<\infty$, and our next goal is to introduce a canonical maximal subspace on which $\dd_p$ does become a metric. 

To this end, we introduce the class $\cN_\R^{\max}\subset\cN_\R^\hom$ of \emph{maximal norms}, of the form $\n=\inf_{v\in X^\div}\{\n_v+c_v\}$ for a bounded family of constants $(c_v)_{v\in X^\div}$. Any divisorial norm is maximal, and maximal norms can alternatively be characterized as decreasing limits of divisorial norms. We further show that any norm $\n_v$ with $v\in X^\lin$ is maximal. 

The following result accounts for the chosen terminology. 

\begin{thmD} Any norm $\n\in\cN_\R$ is asymptotically equivalent to a unique maximal norm $\n^{\max}\in\cN_\R^{\max}$, characterized as the largest norm in the asymptotic equivalence class of $\n$. In particular, for any $p\in [1,\infty)$, the restriction of the pseudometric $\dd_p$ to $\cN_\R^{\max}$ is a metric, and $\cN_\R^{\max}$ is maximal in $\cN_\R$ for this property. 
\end{thmD} 
To prove this result, we first construct a projection $\n\mapsto\n^{\max}$ onto $\cN_\R^{\max}$, by setting $\n^{\max}:=\inf_{v\in X^\div}\{\n_v+\FS(\n)(v)\}$, and show that $\n^{\max}=\n'^{\max}$ iff $\FS(\n)=\FS(\n')$ on $X^\div$. Using Monge--Amp\`ere estimates from~\cite{trivval}, we show that this holds if $\n\sim\n'$. Conversely, we need to show $\n\sim\n^{\max}$. Since $\FS(\n)=\sup_d\FS(\n_d)$ is an envelope of $L$-psh functions, it follows from~\cite{trivval,trivadd} that $\FS(\n)=\FSstar(\n)$ on $X^\div$, and 
$$
\vol(\n)=\ten(\FS(\n))=\ten(\FSstar(\n))\ge\vol(\n^{\max}).
$$
This yields the result, since $\n\le\n^{\max}$ implies $\dd_1(\n,\n^{\max})=\vol(\n^{\max})-\vol(\n)$. 

As before, Theorem~D does not assume the envelope property, but the proof exploits it through the use of resolution of singularities, see~\cite[Theorem~5.20]{trivval}. Note that closely related results were independently obtained in~\cite{BLQ} in a more general local setting.

%
%
  
\subsection*{Valuations of linear growth}
Finally we use the results above to study the structure of the space $X^\lin$ of valuations of linear growth, which we can endow with several metrics.

First, from the embedding $\Xlin\hto\cN_\R$ given by $v\mapsto\n_v$ we get a family of (pseudo)metrics $\dd_p$, $1\le p\le\infty$.
Denoting by $v_\triv\in X^\div$ the trivial valuation, we have in particular 
$$
\dd_\infty(v,v_\triv)=\tee(v),\quad\dd_1(v,v_\triv)=\ess(v) 
$$
where $\ess(v):=\vol(\chi_v)$ is the \emph{expected vanishing order} of $L$ along $v$, widely used in relation to the stability threshold/$\delta$-invariant~\cite{Fujval,LiEquivariant,BlJ}.  The invariant $\dd_p(v,v_\triv)$ with $v\in X^\div$ also appears (under a slightly different guise) in~\cite{Zha22}. 

Second, a valuation is of linear growth iff the Dirac mass $\d_v$ is a measure of finite energy, and in fact we have
\[
  \en^\vee(\d_v)=\ess(v)
\]
for any $v\in X^\lin$, see Theorem~\ref{thm:thmC}. In particular, we have an embedding $\Xlin\hto\cM^1$. Denote by $\dd_{\cM^1}$ the pullback of the metric $\dd_1$ on $\cM^1$ to $X^\lin$.
\begin{corE} The pseudo-metric $\dd_p$ on $X^\lin$ is an actual metric for $1\le p\le\infty$. Further, the metrics $\dd_{\cM^1}$ and $\dd_p$, $1\le p\le\infty$, on $\Xlin$ are equivalent and complete, and they are independent of $L$ up to bi-Lipschitz equivalence.
\end{corE}
Completeness with respect to $\dd_\infty$, as well as independence of $L$, was already observed in~\cite{trivval}, and the key point is thus to show $\dd_\infty\le C\dd_{\cM^1}$, which is done by invoking inequalities involving Monge--Amp\`ere integrals, as in the proof of Theorem A (see~\S\ref{sec:lingrowth} for details).

\smallskip

In~\cite{nakstab2} we use the space $\cM^1$ and its subspace $\cM^\div$ to analyze K-stability. When $X$ is a Fano variety, restricting to Dirac masses $\d_v\in\cM^1$, with $v$ in $\Xlin$ or $\Xdiv$, recovers the valuative criterion of K-stability of Fano varieties due to Fujita and Li~\cite{Fujval,LiEquivariant}. 

An interesting type of valuations $v\in\Xlin$ are those for which the associated filtration $\chi_v$ is of finite type.
If $v\in\Xdiv$, this means $v$ is `dreamy' in the sense of K.~Fujita~\cite{Fujval}, associated to a test configuration with irreducible and reduced central fiber.
While valuations $v\in\Xlin$ with $\chi_v$ of finite type play a crucial role in recent work on K-stability of Fano varieties~\cite{BLX22,BLZ,BLXZ,BX20,LX20,HL}, their role in the general polarized case is less clear (although see~\cite{DerLeg,LiuYa}). The condition of $\chi_v$ being of finite type is quite subtle and in particular depends on the ample $\Q$-line bundle $L$. For this reason we believe that it is useful to study K-stability using functionals on spaces such as $X^\div$, $\Xlin$, $\cM^\div$ or $\cM^1$,  without any finite type assumption.

%
%
\subsection*{Organization}
  
After giving some background in~\S\ref{sec:background}, we study homogenization and the related Fubini--Study and infimum norm operators in~\S\ref{sec:homogFS}, proving part of Theorem~A. In~\S\ref{sec:spectral} we make a spectral analysis of norms on the section ring of $(X,L)$, building upon~\cite{CM15,BE}. After that we give additional background on non-Archimedean pluripotential theory from~\cite{trivval}; in particular we revisit the spaces used in Theorem~B. In~\S\ref{sec:Darvas} we construct and study the Darvas metrics on $\cE^1$ and $\cM^1$, and prove the remaining part of Theorem~A as well as parts~(i)--(iii) of Theorem~B. The classes of divisorial and maximal norms are studied in~\S\ref{sec:divmaxnorm}, where we prove Theorem~D and also consider the regularized Fubini--Study operator. In~\S\ref{sec:MAnorm} we define the Monge--Amp\`ere operator on general norms, and prove Theorem~C as well as Theorem~B~(iv) and Corollary~E.
Finally, Appendix~\ref{sec:tc} revisits the relation between test configurations and Fubini--Study functions, and Appendix~\ref{sec:toric} provides some remarks on the toric case.

%
%
\subsection*{Notation and conventions}
\begin{itemize}
\item We work over an algebraically closed field $k$, of arbitrary characteristic unless otherwise specified.
\item For $x,y\in\R_+$, $x\lesssim y$ means $x\le C_n  y$ for a constant $C_n>0$ only depending on $n$, and $x\approx y$ if $x\lesssim y$ and $y\lesssim x$. Here $n$ will be the dimension of a fixed variety $X$ over $k$.
\item A \emph{pseudo-metric} on a set $Z$ is a function $d\colon Z\times Z\to\R_+$ that is symmetric, vanishes on the diagonal, and satisfies the triangle inequality. It is a metric if it further separates points. 

\item The \emph{Hausdorff quotient} of a pseudo-metric space $(Z,d)$ is the metric space $(Z_H,d_H)$ where $Z_H$ is the quotient of $Z$ by the equivalence relation $x\sim y\Leftrightarrow d(x,y)=0$, and $d_H$ is the induced metric. The map $(Z,d)\to (Z_H,d_H)$ is the unique isometric map of $(Z,d)$ onto a metric space, up to unique isomorphism. 

\item The \emph{Hausdorff completion} of a pseudo-metric space $(Z,d)$ is the complete metric space $(\hat Z,\hat d)$ defined as the completion of the Hausdorff quotient $(Z_H,d_H)$. It comes with an isometric map $(Z,d)\to (\hat Z,\hat d)$ with dense image, which is universal with respect to maps into complete metric spaces. 

\item A \emph{quasi-metric} on $Z$ is function $d:Z\times Z\to\R_+$ that is symmetric, vanishes precisely on the diagonal, and satisfies the \emph{quasi-triangle inequality}
$$
\e d(x,y)\le d(x,z)+d(z,y)
$$
for some constant $\e>0$. A quasi-metric space $(Z,d)$ comes with a Hausdorff topology, and even a uniform structure. In particular, Cauchy sequences and completeness make sense for $(Z,d)$. Such uniform structures have a countable basis of entourages, and are thus metrizable, by general theory. 
\item  We use the standard abbreviations \emph{usc} for `upper semicontinuous', \emph{lsc} for `lower semicontinuous', \emph{wlog} for `without loss of generality'', and \emph{iff} for `if and only if'. 
\item A \emph{net} is a family indexed by a directed set. On many occasions we shall consider nets $(x_d)$ indexed by $d_0\Z_{\ge1}$ for some $d_0\ge1$, and ordered by divisibility.    Note that the \emph{sequence} $(x_{m!})_{m\ge d_0}$ is cofinal in this net.   
\end{itemize}
%
%
%
%
\begin{ackn}
  We thank E.~Bedford, R.~Berman, H.~Blum, G.~Codogni, 
  T.~Darvas, R.~Dervan, A.~Ducros, C.~Favre, T.~Hisamoto,  C.~Li,  J.~Poineau 
  and M.~Stevenson for fruitful discussions and useful comments.
  The first author was partially supported by the ANR grant GRACK\@.
  The second author was partially supported by NSF grants
  DMS-1600011, DMS-1900025, DMS-2154380, and the United States---Israel Binational Science Foundation.
\end{ackn}
%
%
%
%
\section{Background}\label{sec:background} 
In the entire paper, $(X,L)$ denotes a projective variety (reduced and irreducible) endowed with an ample $\Q$-line bundle. We review a number of basic facts about norms/filtrations and Berkovich analytification, referring for instance to~\cite{BE,trivval} for more details. 
%
%
\subsection{Norms on a vector space}\label{sec:norms}
As in~\cite{BT72} we will use `additive' terminology, so by a \emph{norm} on a $k$-vector space $V$ we mean a function $\n\colon V\to\R\cup\{+\infty\}$ such that
\begin{itemize}
  \item $\n(v)=+\infty$ iff $v=0$;
\item $\n(av)=\n(v)$ for $a\in k^\times$ and $v\in V$; and
\item $\n(v+w)\ge\min\{\n(v),\n(w)\}$ for all $v,w\in V$. 
\end{itemize}
Note that $\|\cdot\|_\chi:=e^{-\chi(\cdot)}$ is then a non-Archimedean norm on $V$ with respect to the trivial absolute value on $k$ in the usual (`multiplicative') sense~\cite{BGR}. Setting 
$$
\Filt^\la V:=\{v\in V\mid \chi(v)\ge\la\},\quad\n(v):=\max\{\la\in\R\mid v\in\Filt^\la V\}
$$
for $\la\in\R$ yields a 1--1 correspondence between norms on $V$ and (non-increasing, left-continuous, exhaustive and separated) \emph{filtrations} of $V$. We also write $\Filt^{>\la}V:=\bigcup_{\la'>\la}\Filt^{\la'}V=\{\n>\la\}$, and define the \emph{associated graded space} as the $\R$-graded vector space
$$
\gr_\chi V:=\bigoplus_{\la\in\R}\Filt^\la V/\Filt^{>\la} V.
$$
Each norm $\chi$ on $V$ turns it into a (Hausdorff) topological vector space, in which $(\Filt^{m\e} V)_{m\in\N}$ forms a countable basis of (open and closed) neighborhood of $0$, for any $\e>0$. The normed space $(V,\chi)$ admits a \emph{completion} $\hat V$, a complete topological vector space containing $V$ as a dense subspace, whose topology is defined by a (unique) norm on $\hat V$ extending $\chi$. The inclusion $V\hto\hat V$ induces an isomorphism
\begin{equation}\label{equ:grcomp}
\gr_\chi V\simto\gr_\chi \hat V.
\end{equation}
We denote by $\cN_\R(V)$ the set of norms on $V$. It has a distinguished element $\n_\triv$, the \emph{trivial norm}, such that $\n_\triv(v)=0$ for all $v\ne0$, and it admits a scaling action by $\R_{>0}$ and a partial ordering defined by $\chi\le\chi'$ iff $\chi(v)\le\chi'(v)$ for all $v$. Any two elements $\chi,\chi'\in\cN_\R(V)$ admit an infimum $\chi\wedge\chi'\in\cN_\R(V)$, defined pointwise by 
$$
(\chi\wedge\chi')(v):=\min\{\chi(v),\chi'(v)\}.
$$
For any subgroup $\La\subset\R$, we denote by $\cN_\La(V)$ the set of norms with values in $\La\cup\{+\infty\}$. Thus
$$
\{\n_\triv\}=\cN_{\{0\}}(V)\subset\cN_\La(V)\subset\cN_\R(V).
$$
A norm $\chi\in\cN_\R(V)$ lies in $\cN_\La(V)$ iff the $\R$-grading of $\gr_\chi V$ reduces to a $\La$-grading. 

\medskip

Assume now that $V$ is finite dimensional. Any norm $\chi$ on $V$ admits an \emph{orthogonal basis} $(e_i)$, \ie a basis of $V$ such that 
$$
\chi(\sum_i a_i e_i)=\min_{a_i\ne 0} \chi(e_i)
$$
for all $a_i\in k$. Up to reordering, an orthogonal basis is simply a compatible basis for the flag of linear subspaces underlying the filtration defined by $\chi$, and elementary linear algebra thus implies that any two norms $\chi,\chi'$ on $V$ admit a joint orthogonal basis. 
  
In particular, all norms on $V$ are equivalent, which means, in our additive terminology, that $\chi-\chi'$ is a bounded function on $V\setminus\{0\}$ for all $\chi,\chi'\in\cN_\R(V)$. The classical \emph{Goldman--Iwahori metric} on $\cN_\R(V)$ is defined by 
\begin{equation}\label{equ:GI}
\dd_\infty(\n,\n')=\sup_{v\in V\smallsetminus\{0\}}\left|\n(v)-\n'(v)\right|,
\end{equation}
where the supremum is achieved among the elements of any joint orthogonal basis for $\chi$ and $\chi'$. For later use, note that
\begin{equation}\label{equ:GIorder}
\n\le\n'\le\n''\Longrightarrow\dd_\infty(\n,\n')\le\dd_\infty(\n,\n'').
\end{equation}
The metric space $(\cN_\R(V),\dd_\infty)$ is complete, but not locally compact as soon as $\dim V\ge 2$.    Note also that $\cN_\Z(V)$ is a closed, discrete subset of $\cN_\R(V)$, while $\cN_\Q(V)$ is dense.    For any $\n\in\cN_\R(V)$, we set
\begin{equation}\label{equ:laminmax}
\la_{\min}(\n):=\min_{v\in V\smallsetminus\{0\}}\n(v),\quad\la_{\max}(\n):=\max_{v\in V\smallsetminus\{0\}}\n(v).
\end{equation}
Thus 
$$
\dd_\infty(\n,\n_\triv)=\max\{\la_{\max}(\n),-\la_{\min}(\n)\}.
$$
Any norm $\chi$ on $V$ induces a norm on the dual space and on all tensor powers, in such a way that the bases canonically induced by any given orthogonal basis of $V$ remain orthogonal. If $\pi\colon V\to V'$ is a surjective linear map, then $\chi$ also induces a quotient norm $\chi'$ on $V'$, such that $\chi'(v')=\max\{\chi(v)\mid \pi(v)=v'\}$ for all $v'\in V'$. 
%
\subsection{Norms on a graded algebra}\label{sec:normsalg}

Let now $R=\bigoplus_{m\in\N} R_m$ be a graded $k$-algebra. It comes with an action of $k^\times$ for which $a\cdot s=a^{m}s$ for $a\in k^\times$ and $s\in R_m$. We write $\cN_\R(R)$ for the set of vector space norms $\chi\colon R\to\R$ that are
\begin{itemize} 
\item \emph{superadditive}, \ie $\chi(f g)\ge\chi(f)+\chi(g)$ for $f,g\in R$; 
\item \emph{$k^\times$-invariant}, \ie $\chi(a\cdot f)=\chi(f)$ for $a\in k^\times$ and $f\in R$; this is equivalent to $\chi$ being compatible with the grading of $R$, that is, $\chi(\sum_ms_m)=\min_m\chi(s_m)$ where $s_m\in R_m$; 
\item \emph{linearly bounded}, \ie there exists $C>0$ such that $|\chi|\le Cm$ on $R_{m}\setminus\{0\}$ for all $m\ge1$. \end{itemize}
Norms in $\cN_\R(R)$ are in 1--1 correspondence with graded, linearly bounded filtrations of $R$ as in~\cite{BC,WN12}. Each $\chi\in\cN_\R(R)$ defines a graded algebra
$$
\gr_\chi R=\bigoplus_{m\in\N}\gr_\chi R_m=\bigoplus_{(m,\la)\in\N\times\R}\Filt^\la R_m/\Filt^{>\la} R_m.
$$
   
\begin{lem}\label{lem:normval} A norm $\n\in\cN_\R(R)$ is a valuation on $R$, \ie it satisfies $\chi(fg)=\chi(f)+\chi(g)$ for all $f,g\in R$, iff $\gr_\chi R$ is an integral domain. 
\end{lem}
   
The set $\cN_\R(R)\hto\prod_m\cN_\R(R_m)$ is stable under the scaling action of $\R_{>0}$ and infima; it further admits an additive action of $\R$, denoted by $(c,\chi)\mapsto \chi+c$, such that 
\begin{equation}\label{equ:gradedtrans}
(\chi+c)(s):=\chi(s)+cm\ \text{for $s\in R_m$}.
\end{equation}
For any subgroup $\Lambda\subset\R$, denote by $\cN_\Lambda(R)\subset\cN_\R(R)$ the set of norms with values in $\Lambda\cup\{+\infty\}$. Norms in $\cN_\Z(R)$ and $\cN_\Q(R)$ will be called \emph{integral} and \emph{rational}, respectively. Integral norms are in 1--1 correspondence with $\Z$-filtrations, as considered in~\cite{Sze}. 

For any norm $\chi\in\cN_\R(R)$, the \emph{round-down}  $\lfloor\chi\rfloor\in\cN_\Z(R)$, defined by 
\begin{equation}\label{equ:rounddown}
\lfloor\chi\rfloor(s):=\lfloor\chi(s)\rfloor,\quad s\in R_m\setminus\{0\}, 
\end{equation}
is an integral norm.

\begin{exam}\label{exam:pdisc}
  Consider the algebra $k[z]=k[z_1,\dots,z_n]$ of polynomials in $n$ variables, with the usual grading. For each    $\xi\in\R^n$, the monomial valuation  
  \begin{equation}\label{equ:pdisc}
    \n_\xi(\sum_{\a\in\N^n} c_\a z^\a)=\min_{c_\a\ne 0}\langle\a,\xi\rangle=\min_\a\{v_\triv(c_\a)+\langle\a,\xi\rangle\}
 \end{equation}
  defines a norm on the graded algebra $k[z]$. The completion of $(k[z],\n_\xi)$ is the algebra $k\{z;\xi\}$ of formal power series $\sum_\a c_\a z^\a\in k\cro{z}$ such that $\lim_\a (v_\triv(c_\a)+\langle\a,\xi\rangle)=+\infty$, whose norm is still defined by~\eqref{equ:pdisc}. In multiplicative notation, $k\{z;\xi\}$ is the polydisc algebra $k\{r^{-1}z\}$, with $r_j=e^{-\xi_j}$, a building block of Berkovich spaces~\cite{BerkBook,BerkIHES}.   
\end{exam}

From now on, we assume that $R$ is finitely generated, so that each graded piece $R_m$ is finite dimensional. 
\begin{defi} We say that a norm $\chi\in\cN_\R(R)$ is \emph{generated in degree $1$} if $R$ is generated in degree $1$ and, for any $m\ge1$, the restriction $\chi|_{R_m}$ is the quotient norm of $S^m(\chi|_{R_1})$ under the canonical surjective map $S^mR_1\to R_{m}$.
\end{defi}
Concretely, $\chi$ is generated in degree $1$ iff, given a $\chi$-orthogonal basis $(s_i)$ of $R_1$, any $s\in R_m$ can be written as $s=\sum_{|\a|=m} c_\a \prod_i s_i^{\a_i}$ with $c_\a\in k$ and $\chi(s)=\min_{c_\a\ne 0}\sum_i\a_i\chi(s_i)$. 

\begin{lem}\label{lem:nft} For any subgroup $\La\subset\R$ and $\chi\in\cN_\La(R)$, the following conditions are equivalent: 
\begin{itemize}
\item[(i)] $\chi$ is generated in degree $1$; 
\item[(ii)] $\gr_\chi R=\bigoplus_{m\in\N}\gr_\chi R_m$ is generated in degree $1$; 
\item[(iii)] there exists    $\xi\in\La^N$ and a surjective map of graded $k$-algebras $\pi\colon k[z_1,\dots,z_N]\to R$ with respect to which $\chi$ is the quotient norm of $\n_\xi$ as in Example~\ref{exam:pdisc}.   
\end{itemize}
When this holds, we further have $\chi\in\cN_{\La'}(R)$ for some finitely generated subgroup $\La'\subset\La$. 
\end{lem}
\begin{proof} Assume (i). Choose a $\chi$-orthogonal basis $(s_i)_{1\le i\le N}$ of $R_1$,    and set $\xi_i:=\n(s_i)$. As noted above, any $s\in R_m\smallsetminus\{0\}$ can be written as $s=\sum_{|\a|=m} c_\a \prod_i s_i^{\a_i}$ with $\chi(s)=\min_{c_\a\ne 0}\sum_i\a_i\xi_i$. This already yields the final assertion, with $\La':=\sum_i\Z\xi_i$. 

Define $A$ as the set of $\a$ achieving $\min_{c_\a\ne 0}\sum_i\a_i\xi_i=\n(s)$    and set $s':=\sum_{\a\in A} c_\a \prod_i s_i^{\a_i}$. Then $s-s'\in\Filt^{>\chi(s)} R_m$, so $s=s'$ in $\gr_\chi R_m$. This shows that $S^m\gr_\chi R_1\to\gr_\chi R_m$ is surjective, and hence (i)$\Rightarrow$(ii). If we define $\pi\colon k[z]\to R$ by $\pi(z_i)=s_i$, then it is clear that $\chi$ is the quotient norm of $\n_\xi$ with $\xi=(\xi_i)$, hence (i)$\Rightarrow$(iii). 

Conversely, any quotient of a norm generated in degree $1$ is plainly generated in degree $1$ as well; hence (iii)$\Rightarrow$(i). Assume now (ii), and pick again a $\chi$-orthogonal basis $(s_i)$ of $R_1$. Each $s\in R_m\smallsetminus\{0\}$ can then be written as $s=\sum_{|\a|=m} a_\a \prod_i s_i^{\a_i}+s'$ where $a_\a\in k^\times$, $\chi(s)=\sum_i\a_i\chi(s_i)$ for all $\a$ and $s'\in\Filt^{>\chi(s)} R_m$.
Repeating the procedure with $s'$ in place of $s$ and using the fact that $\la\mapsto\Filt^\la R_m$ jumps only finitely many times (by finite-dimensionality of $R_m$), we end up with a decomposition $s=\sum_{|\a|=m} c_\a\prod_i s_i^{\a_i}$ such that $\chi(s)=\min_{c_\a\ne 0}\sum_i\a_i\chi(s_i)$. This proves that $\chi$ is generated in degree $1$, thus (ii)$\Rightarrow$(i). 
\end{proof}

%
\subsection{Norms on section rings}\label{sec:normsec} 
Recall that $L$ is an ample $\Q$-line bundle on a projective variety $X$. For any $d\in\N$ such that $dL$ is an actual line bundle, we write $R_d:=\Hnot(X,dL)$, and denote by 
$$
R^{(d)}=R(X,dL)=\bigoplus_{m\in\N} R_{md} 
$$
the \emph{$d$-th Veronese algebra}, \ie the section ring of $dL$; it is generated in degree $1$ for all $d$  sufficiently divisible, since $L$ is ample. When $d$  divides $d'$, we have a restriction map $\cN_\R(R^{(d)})\to\cN_\R(R^{(d')})$, and we set
\begin{equation}\label{equ:infvero}
\cN_\R=\cN_\R(X,L):=\varinjlim_d\cN_\R(R^{(d)}). 
\end{equation}
The set $\cN_\R$ inherits a partial order with finite infima, and commuting actions of $\R_{>0}$ (by scaling) and $\R$ (by translation).

An element $\chi\in\cN_\R$ is represented by a norm on some $R^{(d)}$, two such norms being identified if they coincide on some further Veronese subalgebra; for convenience, we simply refer to $\chi$ as a norm. For all $m$ sufficiently divisible, we denote by $\chi|_{R_m}\in\cN_\R(R_m)$ the restriction of $\chi$ to $R_m$. 

\begin{rmk}\label{rmk:depnot} To define $\chi|_{R_m}$, one needs to choose a representative of $\chi$ as a norm on some $R^{(d)}$. But any other choice leads to the same norms $\chi|_{R_m}\in\cN_\R(R_m)$ for $m$ sufficiently divisible, and the choice of representative can thus safely be ignored. 
\end{rmk}

   For any subgroup $\La\subset\R$, we similarly introduce 
$$
\cN_\La:=\varinjlim_d\cN_\La(R^{(d)}). 
$$
It can be identified with the set of $\n\in\cN_\R$ such that $\chi(R_m\setminus\{0\})\subset\Lambda$ for $m$ sufficiently divisible. Note that $\cN_\La$ is invariant under the scaling action of $\{t\in\R_{>0}\mid t\La\subset\La\}$ and the translation action of the divisible group $\Q\La\subset\R$, by~\eqref{equ:gradedtrans}.

  \begin{exam}\label{exam:tc}
    Any (not necessarily ample) test configuration $(\cX,\cL)/\A^1$ defines a norm $\n_\cL\in\cN_\Z$ (see \S\ref{sec:tc1}). In this case, the 
    translation action by $c\in\Q$ corresponds to twisting $\cL$ by $c\cX_0$, while the scaling action by $d\in\Z_{>0}$ corresponds to the base change $\A^1\to\A^1$ given by $z\mapsto z^d$. 
\end{exam}

The Goldman--Iwahori metric~\eqref{equ:GI} induces a pseudo-metric $\dd_\infty$ on $\cN_\R$ by setting 
\begin{equation}\label{equ:GIas} 
\dd_\infty(\chi,\chi'):=\limsup_m m^{-1}\dd_\infty(\chi|_{R_m},\chi'|_{R_m})\in\R_{\ge 0}
\end{equation}
The limsup is taken with respect to the partial ordering on $\Z_{>0}$ by divisibility, and it is finite, by linear boundedness of $\chi,\chi'$. This pseudo-metric is not a metric (see however Proposition~\ref{prop:homcomp}):

\begin{exam}\label{exam:dinftyround} Pick any norm $\n\in\cN_\R$, with round-down $\lfloor\n\rfloor\in\cN_\Z$, see~\eqref{equ:rounddown}. For $m$ sufficiently divisible, we then have $\dd_\infty(\n|_{R_m},\lfloor\n\rfloor|_{R_m})\le 1$, and hence $\dd_\infty(\n,\lfloor\n\rfloor)=0$. In particular, $\cN_\Z$ is dense in $\cN_\R$ in the $\dd_\infty$-topology. 
\end{exam}

We also introduce
\begin{equation}\label{equ:lamax}
\la_{\max}(\n):=\lim_m m^{-1}\la_{\max}(\n|_{R_m}),
\end{equation}
where $\la_{\max}(\n|_{R_m})$ is defined by~\eqref{equ:laminmax} and the limit exists and is finite because $m^{-1}\la_{\max}(\n|_{R_m})$ is increasing with respect to divisibility, and bounded by linear boundedness of $\n$. Note that
\begin{equation}\label{equ:dtriv}
\n\ge\n_\triv\Longrightarrow\dd_\infty(\n,\n_\triv)=\la_{\max}(\n).
\end{equation}

%
\subsection{$\R$-test configurations}\label{sec:normft}

\begin{defi} We say that a norm $\chi\in\cN_\R$ is \emph{of finite type} if it is represented by a norm on some $R^{(d)}$ whose associated graded algebra $\gr_\chi R^{(d)}$ is of finite type. 
\end{defi} 
  
Equivalently, a norm $\chi\in\cN_\R$ is of finite type iff it is represented by a norm on some $R^{(d)}$ that is generated in degree $1$, by Lemma~\ref{lem:nft}. We denote by 
$$
\cT_\R\subset\cN_\R
$$
the set of such norms. In line with~\cite{DS,HL}, we interpret the elements of $\cT_\R$ as \emph{$\R$-test configurations}. This is justified by the Rees construction, which sets up a 1--1 correspondence between the subset 
$$
\cT_\Z:=\cN_\Z\cap\cT_\R
$$
of $\Z$-valued norms in $\cT_\R$ and the set of (usual) ample test configurations for $(X,L)$ (see Appendix~\ref{sec:tc}).    For any $\n\in\cN_\Z$, note further that
\begin{equation}\label{equ:Zft}
\n\in\cT_\Z\Longleftrightarrow\bigoplus_{\la\in\Z}\Filt^\la R^{(d)}\text{ of finite type over }k\text{ for }r\text{ sufficiently divisible}.
\end{equation}  
   More generally, for any subgroup $\La\subset\R$ we set
$$
\cT_\La:=\cN_\La\cap\cT_\R. 
$$
   As above, $\cT_\La$ is invariant under the scaling action of $\{t\in\R_{>0}\mid t\La\subset\La\}$ and the translation action of $\Q\La$. In particular, $\cT_\Z$  is invariant under translation by $\Q$. It is also easy to see that 
\begin{equation}\label{equ:lamaxft}
\n\in\cT_\La\Longrightarrow\la_{\max}(\n)\in\Q\La.
\end{equation}

By Lemma~\ref{lem:nft}, we have
$$
\cT_\R=\bigcup_{\La\subset\R\text{ finitely generated}} \cT_{\La}. 
$$
  
The \emph{central fiber} of an $\R$-test configuration $\n\in\cT_\R$ is defined as the polarized scheme 
\begin{equation}\label{equ:centralfiber}
(\cX_0,\cL_0):=\left(\Proj\left(\gr_\n R^{(d)}\right),d^{-1}\cO(1)\right),
\end{equation}
for $d\ge 1$ sufficiently divisible. If $\n\in\cT_\La$ with $\La\simeq\Z^r$ finitely generated, the $\La$-grading of $\gr_\n R^{(d)}$ provides a $\Gm^r$-action on $(\cX_0,\cL_0)$. 

The smallest value of $r$ is called the \emph{rank} of $\n$; it is equal to $1$ iff $\n$ is a usual test configuration, up to scaling. 

\begin{rmk}\label{rmk:rank} For an $\R$-test configuration $\n\in\cT_\R$, there does not generally exist a smallest subgroup $\La\subset\R$ such that $\n\in\cT_\La$, because the subgroup $\La_m\subset\R$ generated by the values of $\n|_{R_m}$ need not stabilize for $m$ sufficiently divisible. However, the associated $\Q$-vector space $\Q\La_m$ does stabilize, its dimension being the rank of $\n$. 
\end{rmk}

\begin{exam}\label{exam:torusaction} Extending Example~\ref{exam:pdisc}, suppose that $(X,L)$ is acted upon by a torus $T=\Gm^r$. Then each $\xi\in\R^r$ defines a norm $\n=\n_\xi\in\cN_\R$, given by 
$$
\n(s):=\min\left\{\langle\a,\xi\rangle\mid\a\in\Z^r,\,s_\a\ne 0\right\}
$$
for $s\in R_m$ with $m$ sufficiently divisible, where $s=\sum_{\a\in\Z^r} s_\a$ is the weight decomposition. This norm satisfies 
$$
\gr_\n R^{(d)}\simeq \bigoplus_{\la\in\R}\left(\bigoplus_{\a\in M,\,\langle\a,\xi\rangle=\la} R^{(d)}_\a\right)=R^{(d)},
$$
which shows that $\chi\in\cT_\R$ is of finite type, with central fiber isomorphic to $(X,L)$.    Further, $\n$ lies in $\cT_\La$ for the finitely generated subgroup $\La=\sum_i\Z\xi_i$.   
\end{exam}

\begin{exam}\label{exam:HL} Pick an embedding $X\hto\P^N$ in a projective space such that $\cO(1)|_X=dL$ for some $d\ge 1$, and suppose we are given an action of a torus $T=\Gm^r$ on $(\P^N,\cO(1))$. By Example~\ref{exam:torusaction}, each $\xi\in\R^r$ defines a norm on $R(\P^N,\cO(1))$, generated in degree $1$, which restricts to a norm in $\cT_\R$. By Lemma~\ref{lem:nft}, every element of $\cT_\R$ conversely arises in this way (compare~\cite[Lemma~2.10]{HL}). 
\end{exam}
 
Following~\cite{HL,Ino}, one can use Example~\ref{exam:HL} to provide a geometric realization of $\R$-test configurations as equivariant polarized families over a toric base (see~\S\ref{sec:Inoue} for a brief discussion).

\begin{defi} We define the \emph{canonical approximants} of a norm $\chi\in\cN_\R$ as the sequence $\n_d\in\cT_\R$ defined for $d\in\Z_{\ge 1}$ sufficiently divisible by letting $\n_d$ be the (class of the) norm on $R^{(d)}$ generated in degree $1$ by $\n_d$. 
\end{defi}
If $d$  divides $d'$ then $\n_d\le\chi_{d'}\le\chi$. As in Remark~\ref{rmk:depnot}, this construction is not entirely canonical, as it depends on the choice of a representative of $\chi$, but this can be ignored as any other choice leads to the same approximants $\n_d$ for $d$ sufficiently divisible. 

A norm $\chi\in\cN_\R$ is of finite type iff $\chi=\n_d$ for all sufficiently divisible $d$. Note also that 
$$
\chi\in\cN_\La\Longrightarrow\n_d\in\cT_\La
$$
for any subgroup $\La\subset\R$.

%
%
\subsection{The Berkovich analytification}\label{sec:berk}
By a \emph{valuation} on $X$ we mean a real-valued valuation $v\colon k(X)^\times\to\R$, trivial on $k$. We denote by $\Xval$ the space of valuations, endowed with the topology of pointwise convergence on $k(X)^\times$. The \emph{trivial valuation} $v_\triv\in \Xval$ is defined by $v_\triv(f)=0$ for all $f\in k(X)^\times$.

By~\cite{BerkBook}, the space $\Xval$ admits a natural compactification $\Xan$, which as a set equals $\Xan=\coprod Y^\val$ with $Y$ ranging over all (closed) subvarieties of $X$. We somewhat imprecisely refer to the points on $\Xan$ as \emph{semivaluations} on $X$. The \emph{support} of a semivaluation in $Y^\val\subset\Xan$ is the subvariety $Y$.

By the valuative criterion of properness, each valuation $v\in \Xval$ admits a \emph{center} $c_X(v)\in X$, characterized as the unique (scheme) point $\xi\in X$ such that $v\ge 0$ on the local ring $\cO_{X,\xi}$ and $v>0$ on its maximal ideal. This applies to semivaluations as well, replacing $X$ with a subvariety, and thus defines a map $c_X\colon \Xan\to X$ (which turns out to be anticontinuous, \ie the preimage of an open subset is closed). 

The space $\Xan$ comes with a natural action of $\R_{>0}$ by scaling $(t,v)\mapsto tv$. This induces an action $(t,\f)\mapsto t\cdot\f$ on functions $\f$ on $X^\an$ by setting
\begin{equation}\label{equ:actionfunc}
(t\cdot\f)(v):=t\f(t^{-1}v), 
\end{equation}
whose fixed points are functions that are \emph{homogeneous}, \ie $\f(tv)=t\f(v)$ for all $t\in\R_{>0}$ and $v\in X^\an$. 

The set $X^\an$ is also endowed with a partial order relation, for which $v\ge v'$ iff $c_X(v)$ is a specialization of $c_X(v')$ and $v\ge v'$ pointwise on the local ring at $c_X(v)$. The trivial valuation satisfies $v\ge v_{\triv}$ for all $v\in \Xan$. 

A (rational) \emph{divisorial valuation} $v$ on $X$ is a valuation of the form $v=t\ord_E$, where $E$ is a prime divisor on a normal, projective birational model $X'\to X$ and $t\in\Q_{>0}$. The center $c_X(v)$ is then the generic point of the image of $E$ in $X$.
For convenience, we also count the trivial valuation $v_\triv$ as divisorial, \ie we allow $t=0$ above. The set $\Xdiv$ of divisorial valuations  is dense in $\Xan$ (see for instance~\cite[Theorem~2.14]{trivval}).

%
%
\subsection{Semivaluations and line bundles}\label{sec:vallb}
A semivaluation $v\in \Xan$ can be naturally evaluated on a section $s\in \Hnot(X,M)$ of any line bundle $M$ on $X$, by defining $v(s)$ as the value of $v$ on the local function corresponding to $s$ in any local trivialization of $M$ at $c_X(v)$. Thus $v(s)\in[0,+\infty]$, $v(s)>0$ iff $s$ vanishes at $c_X(v)$, and $v(s)=\infty$ iff $s$ vanishes along the support of $v$. Further, $v\in \Xval$ iff $v(s)<+\infty$ for all $s\in \Hnot(X,M)\smallsetminus\{0\}$ and all line bundles $M$. We define a continuous function $|s|\colon \Xan\to[0,1]$ by setting 
\begin{equation}\label{equ:trivmetric}
|s|(v):=\exp(-v(s)). 
\end{equation}

Now suppose $L$ is an (ample) line bundle. The $\Z$-grading of $R=R(X,L)$ defines an action of $\Gm$ on the affine cone $Y:=\Spec R$, which comes with a natural surjective $\Gm$-invariant morphism $\pi\colon Y\smallsetminus\{o\}\to X$, where the vertex $o$ of $Y$ is the point defined by the maximal ideal $\bigoplus_{m>0}R_m$.
For any $\xi\in X$, the fiber $\pi^{-1}(\xi)$ contains a unique $\Gm$-invariant point defined by the homogeneous prime ideal generated by all sections $s\in R_m$, $m\ge 1$ that vanish at $\xi$.

By general properties of the analytification functor in~\cite{BerkBook}, the $\Gm$-action on $Y$ induces an action of $\Gm(k)=k^\times$ on $\Yan$, and $\pi$ induces a surjective $k^\times$-invariant map $\pian\colon\Yan\smallsetminus\{w_o\}\to\Xan$, where $w_o\in\Yan$ is the trivial semivaluation with support $o$, which satisfies $w_o=+\infty$ on $\bigoplus_{m>0}R_m$.
A semivaluation $w\in\Yan$ is $k^\times$-invariant iff $w(\sum_ms_m)=\min_mw(s_m)$, where $s_m\in R_m$.

It is easy to see~\cite[\S4.2]{LiEquivariant} that if $v\in\Xan$, then the set of $k^\times$-invariant points in $(\pian)^{-1}(v)$ is of the form $\{w_{v,c}\}_{c\in\R}$, where $w_{v,c}$ is defined by 
\begin{equation}\label{equ:valext}
  w_{v,c}(s)=\min_m\{v(s_m)+cm\}\quad\text{for any $s=\sum_ms_m\in R$},
\end{equation}
and where the value $v(s_m)$ is defined at the top of of this section.
Note that $w_{v,c}$ is centered at the vertex $o$ iff $\la>0$. 

%
%
\subsection{Valuations of linear growth and dreamy valuations}\label{sec:lingr}
Following~\cite{BKMS}, we define the \emph{maximal vanishing order} of (multisections of) $L$ at $v\in \Xan$ as 
\begin{equation}\label{equ:T}
  \tee(v):=\tee_L(v)=\sup m^{-1}v(s)\in[0,+\infty], 
\end{equation}
where the supremum is over $m$ sufficiently divisible and $s\in R_m\setminus\{0\}$.
We say that $v$ has \emph{linear growth} if $\tee(v)<+\infty$; this notion is independent of the ample $\Q$-line bundle $L$. The set $\Xlin\subset \Xan$ of valuations of linear growth satisfies 
$$
\Xdiv\subset\Xlin\subset\Xval.
$$
Further, setting 
\begin{equation}\label{equ:dinftylin}
\dd_\infty(v,w):=\sup m^{-1}|v(s)-w(s)|, 
\end{equation}
where the supremum is again over $m$ sufficiently divisible and $s\in R_m\setminus\{0\}$, defines a metric on $X^\lin$ such that $(X^\lin,\dd_\infty)$ is complete (see~\cite[\S 11.3]{trivval}). We refer to the $\dd_\infty$-topology of $X^\lin$ as the \emph{strong topology}. 

\begin{exam}\label{exam:Izumi} If $\pi\colon X'\to X$ is a proper birational morphism, with $X'$ normal, and $E\subset X'$ is a prime divisor which is $\Q$-Cartier, then $\tee_L(E)$ coincides with the pseudoeffective threshold $\sup\{t\ge 0\mid\pi^*L-tE\in\Psef(X)\}$ (see~\cite[Theorem 2.24]{BKMS}). 
\end{exam}

    Any $v\in\Xlin$ defines a norm $\chi_v\in\cN_\R$, given by $\chi_v(s):=v(s)$ for $s\in R_m$ with $m$ sufficiently divisible. It satisfies $\la_{\max}(\n_v)=\tee(v)$ (see~\eqref{equ:lamax}). Further, the map 
$$
\Xlin\to\cN_\R,\quad v\mapsto\n_v
$$
is injective, because the function field of $X$ coincides with the homogeneous fraction field of $R^{(d)}$ for any $d$ sufficiently divisible. 

For any $v\in X^\lin$ and $c\in\R$, the norm $\chi_v+c$ can be viewed as a valuation on the affine cone $\Spec R^{(d)}$ for $d$ sufficiently divisible; it coincides with $w_{v,c}$ in the notation of~\eqref{equ:valext}. By Lemma~\ref{lem:normval}, such norms are characterized as follows. 

\begin{lem}\label{lem:charac} A norm $\n\in\cN_\R$ is of the form $\n=\n_v+c$ with $v\in X^\lin$ and $c\in\R$ iff $\gr_\chi R^{(d)}$ is an integral domain for some (or any) sufficiently divisible $d$. 
\end{lem}
When $\n\in\cT_\R$ is of finite type, the latter condition means that the corresponding central fiber $\cX_0$ is reduced and irreducible, see~\eqref{equ:centralfiber}. 

\begin{exam}\label{exam:torusval} Suppose that a torus $T$ acts on $(X,L)$. By Example~\ref{exam:torusaction}, each $\xi\in N_\R$ defines a norm $\chi_\xi\in\cT_\R$ whose associated central fiber $\cX_0\simeq X$ is integral. By Lemma~\ref{lem:charac}, $\chi_\xi$ thus determines a valuation $v_\xi\in\Xlin$, which only depends on the $T$-action on $X$, and can be obtained by the `action' of $\xi\in N_\R\subset T^\an$ on $v_\triv\in X^\an$ in the sense of `peaked points' (see~\cite[\S5.2]{BerkBook}). 
\end{exam} 

In the terminology of~\cite{Fujval}, a divisorial valuation $v\in X^\div$ such that $\n_v$ is of finite type is called \emph{dreamy} (with respect to $L$). 

\begin{exam} Assume $X$ is normal and $E\subset X$ is a $\Q$-Cartier prime divisor. If $v:=\ord_E$ is dreamy with respect to  $L$, then the pseudoeffective threshold
$$
\sup\{t\ge 0\mid L-tE\in\Psef(X)\}=\tee_L(v)=\la_{\max}(\n_v)
$$
is necessarily rational (\cf Example~\ref{exam:Izumi} and~\eqref{equ:lamaxft}). Examples with an irrational threshold are well-known (\eg when $X$ is an abelian surface of Picard number at least $2$), and therefore provide simple examples of non-dreamy valuations. 
\end{exam}

The next result generates examples of divisorial valuations that are not dreamy for any polarization of $X$. 

\begin{lem}\label{lem:dreamyloc} Pick a dreamy valuation $v\in X^\div$ (with respect to a given ample $\Q$-line bundle $L$), and assume that $v$ is centered at a closed point $p\in X(k)$, with valuation ideals 
$$
\fa_m:=\{f\in\cO_{X,p}\mid v(f)\ge m\}.
$$
Then the Rees algebra $\bigoplus_{m\in\N}\fa_m$ is of finite type over $\cO_{X,p}$. 
\end{lem}
In particular, the (local) \emph{volume} of $v$
$$
\vol(v)=\lim_{m\to\infty}\frac{n!}{m^m}\dim(\cO_{X,p}/\fa_m)
$$ 
must be rational (see~\cite{ELS}). 

\begin{proof} After replacing $v$ with a multiple, we may assume that $v$ is $\Z$-valued, and hence that $\n_v$ is a $\Z$-filtration. By~\eqref{equ:Zft}, the bigraded $k$-algebra $\bigoplus_{(\la,m)\in\Z\times\N}\Filt^\la R_{dm}$ is finitely generated over $k$ for $d$ sufficiently divisible, and hence so is the graded subalgebra $\bigoplus_{m\in\N}\Filt^m R_{dm}$.

On the other hand, by~\cite[Lemma~2.17]{BKMS}, we can find $d\ge 1$ sufficiently divisible such that $\cO_X(mdL)\otimes\fa_m$ is globally generated for all $m\in\N$. Since $\Hnot(X,\cO_X(mdL)\otimes\fa_m)=\Filt^m R_{dm}$, we infer that $\bigoplus_{m\in\N}\fa_m$ is of finite type over $\cO_{X,p}$. 
\end{proof}

\begin{exam} Assume $k=\C$, $\dim X\ge 4$, and pick a smooth point $p\in X(k)$. By~\cite{Kur03}, we can find a divisorial valuation $v\in X^\div$ centered at $p$ such that $\vol(v)$ is irrational. By Lemma~\ref{lem:dreamyloc}, $v$ is not dreamy with respect to any ample $\Q$-line bundle $L$ on $X$. 
\end{exam}

%
%
 \subsection{Fubini--Study functions}\label{sec:FSfunc}
A \emph{Fubini--Study function} (for $L$) is a function $\f\in\Cz=\Cz(X^\an)$ of the form
\begin{equation}\label{equ:FSalt}
  \f=\tfrac1m\max_j\{\log|s_j|+\la_j\}, 
\end{equation}
with $m\ge 1$ such that $mL$ is a (globally generated) line bundle, $(s_j)$ a finite set of $R_m$ without common zeros, and $\la_j\in\R$. Recall that~\eqref{equ:FSalt} means $\f(v)=\tfrac1m\max_j\{-v(s_j)+\la_j\}$ for all $v\in X^\an$, see~\eqref{equ:trivmetric}. 

\begin{rmk} The function $\f$ defines a continuous metric $|\cdot|e^{-m\f}$ on the Berkovich analytification of $mL$. This metric is the pullback of a standard (non-Archimedean) Fubini--Study (or Weil) metric on $\cO(1)$ under the morphism $X\to\P^N$ defined by $(s_j)_{0\le j\le N}$, which explains the chosen terminology.
\end{rmk}

If the $\la_j$ in~\eqref{equ:FSalt} can be chosen in a subgroup $\La\subset\R$, we say that $\f$ is a \emph{$\La$-Fubini--Study function}, and write $\cH_\La=\cH_\La(L)\subset\Cz$ for the set of such functions. Thus
$$
\{0\}=\cH_{\{0\}}\subset\cH_\La\subset\cH_\R.
$$
Note that 
\begin{equation}\label{equ:HQ}
\cH_\La=\cH_{\Q\La}
\end{equation}
and $\cH_\La(dL)=d\cH_\La(L)$ for any $d\in\Q_{>0}$.  The set $\cH_\La$ is stable under finite max and under the action of $\Q\La$ by translation. 

Recall the action~\eqref{equ:actionfunc} of $\R_{>0}$ on functions on $X^\an$. If $\f$ is given by~\eqref{equ:FSalt} and $t\in\R_{>0}$, then 
$$
t\cdot\f=\tfrac 1m\max_j\{\log|s_j|+t\la_j\}.
$$
Thus $\cH_\R$ is stable under the action of $\R_{>0}$, while $\cH_\La$ is stable under the action of the stabilizer $\{t\in\R_{>0}\mid t\La\subset\La\}$. In particular, $\cH_\Q$ is stable under the action of $\Q_{>0}$. 

%
%
%
%
\section{Homogenization and the Fubini--Study operator}\label{sec:homogFS}
In this section we study the homogenization of a norm, and the related Fubini--Study and infimum norm operators. We show that homogenization preserves norms of finite type, establish a 1--1 correspondence between homogeneous norms of finite type and Fubini--Study functions, and we prove Theorem~A in the case $p=\infty$. 

In what follows, $\cL^\infty$ denotes the space of bounded functions $
\f\colon X^\an\to\R$, endowed with its usual supnorm metric $\dd_\infty(\f,\f'):=\sup_{X^\an}|\f-\f'|$. 
%
\subsection{Homogenization}
In this section, $R=\bigoplus_{m\in\N} R_m$ denotes any reduced graded $k$-algebra.  

\begin{defi} We say that a norm $\chi\in\cN_\R(R)$ is \emph{homogeneous} if $\chi(f^d)=d\n(f)$ for all $f\in R$ and $d\in\N$. 
\end{defi}
In multiplicative terminology, this means that $\|\cdot\|_\chi=e^{-\chi}$ is power-multiplicative, see~\cite{BGR}. It is easy to see that a norm $\chi\in\cN_\R(R)$ is homogeneous iff the associated graded algebra $\gr_\chi R$ is reduced. We denote by 
$$
\cN_\R^\hom(R)\subset\cN_\R(R)
$$ 
the set of homogeneous norms on $R$. For any Veronese subalgebra $R^{(d)}=\bigoplus_{m\in\N} R_{dm}$, $d\ge 1$, the restriction map $\cN_\R(R)\to\cN_\R(R^{(d)})$ induces a bijection
\begin{equation}\label{equ:Verohom}
\cN_\R^\hom(R)\simto\cN_\R^\hom(R^{(d)}).
\end{equation}
Any norm $\chi\in\cN_\R(R)$ is dominated by a minimal homogeneous norm, namely its \emph{homogenization} $\n^\hom$, defined by
\begin{equation}\label{equ:homog}
\n^\hom(f):=\sup_{d\ge 1} \tfrac 1d\chi(f^d)=\lim_{d\to\infty} \tfrac 1d\chi(f^d),
\end{equation}
where the second equality holds by superadditivity of $d\mapsto\chi(f^d)$ and Fekete's Lemma. It is indeed easy to check that~\eqref{equ:homog} defines a vector space norm on $R$ that is superadditive, $k^\times$-invariant, linearly bounded and homogeneous, \ie an element $\n^\hom\in\cN_\R^\hom(R)$. 

\begin{rmk}\label{rmk:specrad} Note that $\|\cdot\|_{\n^\hom}=e^{-\n^\hom(\cdot)}$ is the \emph{spectral radius (semi)norm} of $\|\cdot\|_\chi$ in the (multiplicative) terminology of~\cite{BerkBook}. 
\end{rmk}

Using standard but nontrivial results on $k$-affinoid algebras, we prove: 

\begin{thm}\label{thm:homogft}
Let $\chi\in\cN_\R(R)$ be a norm generated in degree $1$, with homogenization $\n^\hom$. Then: 
\begin{itemize}
\item[(i)] there exists $C>0$ such that $\chi(f)\le\n^\hom(f)\le\n(f)+C$ for all $f\in R$; 
\item[(ii)] the $k$-algebra $\gr_{\n^\hom} R$ is finitely generated; 
\item[(iii)] if $\chi\in\cN_\La(R)$ for a subgroup $\La\subset\R$, then $\n^\hom\in\cN_{\Q\La}(R)$. 
\end{itemize}
\end{thm}

\begin{proof} Pick a surjective map of graded algebras $\pi\colon k[z]=k[z_1,\dots,z_N]\to R$ and $\xi\in\R^N$ such that $\chi$ is the quotient norm of $\n_\xi$  (see Lemma~\ref{lem:nft}). As in Example~\ref{exam:pdisc}, the completion of $k[z]$ with respect to $\n_\xi$ is the polydisc algebra $k\{z;\xi\}$, and $\pi$ induces a surjection $k\{z;\xi\}\to\hat R$ onto the completion of $R$, whose norm is the quotient of the norm $\n_\xi$ of $k\{z;\xi\}$. 

As a consequence, $\hat R$ is a $k$-affinoid algebra in the sense of~\cite{BerkBook}, corresponding geometrically to the affinoid domain $Y^\an\cap\DD(r)$ of the Berkovich analytification $Y^\an\hto\A^{N,\an}$ of the affine cone $Y:=\Spec R\hookrightarrow\A^N=\Spec k[z]$, where $\DD(r)\subset\A^{N,\an}$ is the closed polydisc of polyradius $r=(e^{-\xi_1},\dots,e^{-\xi_N})$. 

Since $R$ is assumed to be reduced, it follows from the non-Archimedean GAGA principle that $\hat R$ is reduced as well (see~\cite[Th\'eor\`eme 3.3]{Duc09}), and (i) is now a consequence of~\cite[Proposition~2.1.4~(ii)]{BerkBook}, which states (in multiplicative terminology) that the spectral radius (semi)norm of any reduced $k$-affinoid algebra is equivalent to the given norm. 

Next, note that $\gr_{\n^\hom} \hat R$ coincides, by definition, with the graded reduction of $\hat R$ in the sense of Temkin~\cite[\S3]{TemkinLocalII}. By~\cite[Proposition~3.1]{TemkinLocalII}, the surjection $k\{z;\xi\}\to\hat R$ therefore induces a finite morphism $\gr_{\n_\xi} k\{z;\xi\}\to\gr_{\n^\hom}\hat R$. Now we have $\gr_{\n_\xi} k\{z;\xi\}\simeq\gr_{\n_\xi} k[z]$ and $\gr_{\n^\hom} \hat R\simeq\gr_{\n^\hom} R$, see~\eqref{equ:grcomp}. Thus $\gr_{\n^\hom} R$ is finite over $\gr_{\n_\xi} k[z]\simeq k[T]$, which yields (ii). 

Finally suppose that $\chi\in\cN_\Lambda(R)$ for a subgroup $\La\subset\R$.  In this case, we can choose $\xi\in\Lambda^N$, so $k\{T;\xi\}$ and $\hat{R}$ are both $\Lambda$-strict $k$-affinoid algebras. By~\cite[3.1.2.1~(iv)]{TemkinSurvey}, we thus have $\n^\hom(\hat{R}\smallsetminus\{0\})\subset\Q\Lambda$, which proves (iii). 
\end{proof}
%
\subsection{Homogenization of norms on section rings} 

Returning to the setting of a polarized variety $(X,L)$ and its space of norms $\cN_\R$ (see \S\ref{sec:normsec}), we introduce:
\begin{defi} Consider a norm $\chi\in\cN_\R$. Then: 
\begin{itemize}
\item[(i)] we say that $\chi$ is \emph{homogeneous} if it admits a homogeneous representative on $R^{(d)}=R(X,dL)$ for some $d$ ; 
\item[(ii)] we define the \emph{homogenization} of $\chi$ as the norm $\n^\hom\in\cN_\R$ induced by the homogenization of any representative of $\chi$. 
\end{itemize}
\end{defi}
By~\eqref{equ:homog}, $\n^\hom$ is well-defined; it satisfies $\n^\hom\ge\n$,  and is characterized as the smallest homogeneous norm with this property.    
\begin{exam}\label{exam:homfloor} For any $\n\in\cN_\R$ we have $\n^\hom=(\lfloor\n\rfloor)^\hom$. This is indeed a direct consequence of~\eqref{equ:homog}.
\end{exam}
   
We denote by 
$$
\cN_\R^\hom\subset\cN_\R
$$
the subset of homogeneous norms. By~\eqref{equ:Verohom}, we have 
\begin{equation}\label{equ:homRr}
\cN_\R^\hom\simeq\cN_\R^\hom(R^{(d)})
\end{equation}
for any $d$ such that $dL$ is a line bundle. In other words, a homogeneous norm $\n\in\cN_\R$ is well-defined on $R_m=\Hnot(X,mL)$ for any $m$ such that $mL$ is a line bundle. 
  
\begin{rmk}\label{rmk:homnet} An element of $\n\in\cN_\R$ is a norm on $R^{(d)}$ for some sufficiently divisible $d$ that depends on $\n$, so in general it does not make sense to talk about pointwise convergence of sequences or nets in $\cN_\R$. By~\eqref{equ:homRr}, it does however make sense when the norms are homogeneous, as they are then defined on $R^{(d)}$ any fixed $d$ such that $dL$ is an honest line bundle. 
\end{rmk}
  
The subset $\cN_\R^\hom$ is stable under minima, and under the scaling action of $\R_{>0}$ and the additive action of $\R$. For any subgroup $\La\subset\R$, we set
$$
\cN_\La^\hom:=\cN_\La\cap\cN_\R^\hom.
$$
   
Recall that $\cN_\R$ is equipped with a pseudo-metric $\dd_\infty$, see~\eqref{equ:GIas}. Using~\eqref{equ:homog}, it is straightforward to check:

\begin{lem}\label{lem:homcomp} Homogenization $\n\mapsto\n^\hom$ defines a projection $\cN_\R\twoheadrightarrow\cN_\R^{\hom}$ which is equivariant for the actions of $\R_{>0}$ and $\R$, commutes with minima, and satisfies
$$
\dd_\infty(\n^\hom,\n'^{\hom})\le\dd_\infty(\n,\n'),\quad\la_{\max}(\n^\hom)=\la_{\max}(\n)
$$
for all $\n,\n'\in\cN_\R$. 
\end{lem} 

The restriction of $\dd_\infty$ to $\cN^\hom_\R$ is further well-behaved: 

\begin{prop}\label{prop:homcomp} The restriction of $\dd_\infty$ to $\cN_\R^\hom$ is a metric. Furthermore, the metric space $(\cN_\R^\hom,\dd_\infty)$ is complete, and contains $\cN_\Z^\hom$ as a closed subset. 
\end{prop}
Note that $\cN_\Z^\hom$ is always a strict subset of $\cN_\R^\hom$, thanks to the scaling action of $\R_{>0}$. In contrast, recall that $\cN_\Z$ is $\dd_\infty$-dense in $\cN_\R$ (see Example~\ref{exam:dinftyround}). 

\begin{lem}\label{lem:hominfty} Pick $d\ge 1$ such that $dL$ is an honest line bundle, and view $\dd_\infty$ as a pseudo-metric on $\cN^\hom(R^{(d)})$ via~\eqref{equ:homRr}. For all $\n,\n'\in\cN^\hom(R^{(d)})$ we then have
$$
\dd_\infty(\n,\n')=\sup_{m\ge 1}\tfrac 1{md}\dd_\infty(\n|_{R_{md}},\n'|_{R_{md}}). 
$$
\end{lem}
\begin{proof} By homogeneity of $\n,\n'$, we have for all $m,l\ge 1$ 
\begin{align*}
\dd_\infty(\chi|_{R_{md}},\chi'|_{R_{md}}) & = \sup_{s\in R_{md}\smallsetminus\{0\}}|\chi(s)-\chi'(s)| =  l^{-1}\sup_{s\in R_{lmd}\smallsetminus\{0\}}|\chi(s^l)-\chi'(s^l)|\\
& \le l^{-1}\sup_{t\in R_{lmd}\smallsetminus\{0\}}|\chi(t)-\chi'(t)|= l^{-1}\dd_\infty(\n|_{R_{lmd}},\n'|_{R_{lmd}}). 
\end{align*}
Thus $m\mapsto \tfrac 1{md}\dd_\infty(\chi|_{R_{md}},\chi'|_{R_{md}})$ is increasing with respect to divisibility, and~\eqref{equ:GIas} yields the result (recall that the limsup in the latter formula is understood with respect to the divisibility order). 
\end{proof}

\begin{proof}[Proof of Proposition~\ref{prop:homcomp}] Pick $d$ as in Lemma~\ref{lem:hominfty}. For each $m\ge 1$, $(\cN_\R(R_{md}),\dd_\infty)$ is a complete metric space, in which $\cN_\Z(R_{md})$ sits as a closed subspace. This implies that $\cN_\R(R^{(d)})\hto\prod_{m\ge 1}\cN_\R(R_{md})$ is complete with respect to the metric 
$$
\widetilde\dd_\infty(\chi,\chi'):=\sup_{m\ge 1}\tfrac 1{md}\dd_\infty(\chi|_{R_{md}},\chi'|_{R_{md}}),
$$
and that $\cN_\Z(R^{(d)})\hto\prod_{m\ge 1}\cN_\Z(R_{md})$ is closed. It is also clear that $\cN_\R^\hom(R^{(d)})$ is closed in $\cN_\R(R^{(d)})$ with respect to $\widetilde\dd_\infty$, so the result now follows from Lemma~\ref{lem:hominfty}. 
\end{proof}

Note that Lemma~\ref{lem:hominfty} ensures compatibility of the $\dd_\infty$-metrics on $\cN_\R^\hom$ and $X^\lin$ (see~\eqref{equ:dinftylin}): 

\begin{cor}\label{cor:dinftylin} The map $v\mapsto\n_v$ defines an isometric embedding $(X^\lin,\dd_\infty)\hto(\cN_\R^\hom,\dd_\infty)$, \ie $\dd_\infty(v,w)=\dd_\infty(\n_v,\n_w)$ for all $v,w\in X^\lin$. 
\end{cor}

For any subgroup $\La\subset\R$, we denote by 
$$
\cT_\La^\hom:=\cT_\La\cap\cN_\R^\hom=\cT_\R\cap\cN_\La^\hom. 
$$
the set of homogeneous $\La$-valued norms of finite type. As a straightforward consequence of Theorem~\ref{thm:homogft}, we get: 

\begin{lem}\label{lem:homogft}    For any $\n\in\cT_\R$, the following holds:
\begin{itemize}
\item[(i)] $\n^\hom\in\cT_\R^\hom$; 
\item[(ii)] for $m$ sufficiently divisible, $\dd_\infty(\n|_{R_m},\n^\hom|_{R_m})$ is bounded, and hence $\dd_\infty(\n,\n^\hom)=0$;
\item[(iii)] for any subgroup $\La\subset\R$, $\n\in\cT_\La\Longrightarrow\n^\hom\in\cT_{\Q\La}^\hom$.
\end{itemize}
\end{lem}
As we shall see, homogenization in fact maps $\cT_\La$ onto $\cT^\hom_{\Q\La}$ (\cf Corollary~\ref{cor:FSbij}). For $\La=\Z$, the homogenization map $\cT_\Z\twoheadrightarrow\cT_\Q^\hom$ is closely related to integral closure (see Appendix~\ref{sec:tc} for a detailed discussion). 

%
%
\subsection{The Fubini--Study operator}\label{sec:FS}
Assume first that $L$ is a globally generated line bundle. To any norm $\chi$ on $R_1=\Hnot(X,L)$, we associate a function on $X^\an$ by setting
\begin{equation}\label{equ:FSL1}
\FS_L(\chi):=\sup_{s\in R_1\smallsetminus\{0\}}\{\log|s|+\chi(s)\}, 
\end{equation}
\ie $\FS_L(\chi)(v)=\sup_{s\in R_1\smallsetminus\{0\}}\{-v(s)+\chi(s)\}$ for $v\in X^\an$. Given a $\chi$-orthogonal basis $(s_i)$ of $R_1$, one easily checks that
\begin{equation}\label{equ:FSL}
\FS_L(\chi)=\max_i\{\log|s_i|+\chi(s_i)\}\in\cH_\R
\end{equation}
see~\cite[Lemma 7.17]{BE}. This implies 
\begin{equation}\label{equ:lamaxFS}
\la_{\max}(\n)=\sup_{X^\an}\FS_L(\n)=\FS_L(\n)(v_\triv),
\end{equation}
as well as
\begin{equation}\label{equ:FSFSL}
\chi\in\cN_\La(R_1)\Longrightarrow\FS_L(\chi)\in\cH_\La. 
\end{equation}
for any subgroup $\La\subset\R$. 

\begin{lem}\label{lem:FSft} Assume that $L$ is a line bundle, and let $\chi$ be a norm on $R=R(X,L)$. For each $m\ge 1$ we then have 
$$
\FS_{mL}(\chi|_{R_m})\ge m\FS_L(\chi|_{R_1}),
$$
and equality holds if $\chi$ is generated in degree $1$. 
\end{lem}
\begin{proof} For each $s\in R_1\smallsetminus\{0\}$ we have $\chi(s^m)\ge m\chi(s)$, and the inequality follows, by~\eqref{equ:FSL1}. Assume that $\chi$ is generated in degree $1$. To get equality, we need to show
$$
v(s)\ge\chi(s)-m\FS_L(\chi_1)(v)
$$ 
for all $s\in R_m\smallsetminus\{0\}$ and $v\in X^\an$. To see this, pick an orthogonal basis $(s_i)$ of $R_1$, and write $s=\sum_{|\a|=m} c_\a \prod_i s_i^{\a_i}$ with $\chi(s)=\min_\a\sum_i\a_i\chi(s_i)$ for some $\a$ with $c_\a\ne 0$. Then $\FS_L(\chi)(v)=\max_i\{\chi(s_i)-v(s_i)\}$, and hence  
$$
v(s)\ge\min_{c_\a\ne 0}\sum_i\a_i v(s_i)\ge\min_{c_\a\ne 0}\sum_i\a_i\left(\chi(s_i)-\FS_L(\chi_1)(v)\right)=\chi(s)-m\FS_L(\chi_1)(v),
$$
which concludes the proof. 
\end{proof}

Returning to the general case of a $\Q$-line bundle, pick $\chi\in\cN_\R$, and set 
$$
\FS_m(\chi):=m^{-1}\FS_{mL}(\chi|_{R_m})\in\cH_\R
$$ 
for $m$ sufficiently divisible.  By Lemma~\ref{lem:FSft}, $\FS_m(\n)$ is an increasing function of $m$ with respect to divisibility, and is further uniformly bounded, by linear boundedness of $\chi$. We may thus introduce: 

\begin{defi}\label{defi:FSop} The \emph{Fubini--Study operator} $\FS\colon\cN_\R\to\cL^\infty$ takes a norm $\chi\in\cN_\R$ to the bounded function $\FS(\chi)\colon X^\an\to\R$ defined as the pointwise limit
$$
\FS(\chi):=\lim_m \FS_m(\chi)=\sup_m \FS_m(\chi). 
$$
\end{defi}
Recall that $\cL^\infty$ denotes the space of bounded functions on $X^\an$. The bounded function $\FS(\chi)$ is lsc (lower semicontinuous), being a supremum of continuous functions; it is however not continuous in general (see Theorem~\ref{thm:cont} below). By~\eqref{equ:lamaxFS} we have  
\begin{equation}\label{equ:lamaxFS2}
\la_{\max}(\n)=\FS(\n)(v_\triv)
\end{equation}
Note also that the canonical approximants $\n_d\in\cT_\R$ of any norm $\chi\in\cN_\R$ satisfy 
\begin{equation}\label{equ:FScanapp}
\FS(\n_d)=\FS_d(\chi)\in\cH_\R
\end{equation}
for all $d$ sufficiently divisible, by Lemma~\ref{lem:FSft}. In particular, if $\n\in\cT_\R$ is of finite type then the limit in Definition~\ref{defi:FSop} is stationary. 

The Fubini--Study operator $\FS\colon\cN_\R\to\cL^\infty$ is increasing with respect to the partial orderings on $\cN_\R$ and $\cL^\infty$, and equivariant under the actions of $\R$ and $\R_{>0}$. It is also easily seen to be $1$-Lipschitz with respect to the $\dd_\infty$-(pseudo)metrics, \ie 
\begin{equation}\label{equ:FScontr}
\dd_\infty(\FS(\chi),\FS(\chi'))\le\dd_\infty(\chi,\chi')\text{ for all }\chi,\chi'\in\cN_\R. 
\end{equation}
As we see below, equality holds when $\n,\n'$ are homogeneous (see Corollary~\ref{cor:FSisom}). 

The next result shows that $\FS$ is invariant under homogenization.

\begin{prop}\label{prop:FSH}
  For any $\chi\in\cN_\R$ we have $\FS(\n)=\FS(\n^\hom)$.
\end{prop}
\begin{proof} That $\FS(\chi)\le\FS(\n^\hom)$ is clear, since $\chi\le\n^\hom$. Now let $v\in\Xan$ and $\e>0$. Successively pick $m\ge1$ sufficiently divisible such that $\FS(\n^\hom)(v)\le\FS_m(\n^\hom)(v)+\e$, then $s\in R_m\smallsetminus\{0\}$ such that $m\FS_m(\n^\hom)(v)\le\n^\hom(s)-v(s)+m\e$, and finally $d\ge 1$ such that $\n^\hom(s)\le\frac1d\chi(s^d)+m\e$, see~\eqref{equ:homog}. Then
  \begin{equation*}
    \FS(\n^\hom)(v)
    \le\tfrac1{md}(\chi(s^d)-v(s^d))+3\e
    \le\FS_{md}(\n)(v)+3\e
    \le\FS(\n)(v)+3\e,
  \end{equation*}
  completing the proof.
\end{proof}
The Fubini--Study operator relates norms of finite type and Fubini--Study functions, as follows: 

\begin{prop}\label{prop:FSft} For any subgroup $\La\subset\R$, we have 
$$
\FS(\cT_\La)=\FS(\cT_{\Q\La})=\FS(\cT_{\Q\La}^{\hom})=\cH_\La=\cH_{\Q\La}. 
$$
\end{prop}
As we shall see, the map $\FS\colon\cT_{\Q\La}^\hom\to\cH_\La$ is further 1--1 (see Corollary~\ref{cor:FSbij}). 

\begin{proof} If $\n\in\cT_\La$ then $\n=\n_m$ for $m$ sufficiently divisible, and hence $\FS(\n)=\FS(\n_m)=\FS_m(\n)\in\cH_\La$ by~\eqref{equ:FSFSL}. Conversely, pick $\f\in\cH_\La$, and write $\f=d^{-1}\max_i\{\log|s_i|+\la_i\}$ with $d\ge 1$, a finite family $(s_i)_{1\le i\le N}$ in $R_d$ without common zeros, and $\la_i\in\La$. After enlarging the family $(s_i)$ and choosing the corresponding $\la_i\ll 0$ in $\La$, we may further assume that $(s_i)$ spans $R_d$. Consider the surjective map $k^N\to R_d$ that takes the canonical basis $(e_i)$ to $(s_i)$. Denote by $\chi_0$ the norm on $k^N$ that is diagonal in $(e_i)$, with $\chi_0(e_i)=\la_i$, and let $\n_d\in\cN_\R(R_d)$ be the quotient norm. It is then easy to check (see~\cite[Lemma 7.17]{BE}) that 
$$
\f=d^{-1}\max_j\{\log|s_j|+\la_j\}=d^{-1}\FS_{dL}(\n_d)
$$
By Lemma~\ref{lem:FSft}, the norm $\chi\in\cN_\La(R^{(d)})$ generated in degree $1$ by $\n_d$ satisfies $\FS(\chi)=\f$, which proves $\FS(\cT_\La)=\cH_\La$. By Proposition~\ref{prop:FSH} and Lemma~\ref{lem:homogft}, we infer
$$
\cH_\La=\FS(\cT_\La)\subset\FS(\cT_{\Q\La}^\hom)\subset\cH_{\Q\La}, 
$$
which concludes the proof since $\cH_{\Q\La}=\cH_\La$, see~\eqref{equ:HQ}. 
\end{proof}

%
%
\subsection{The infimum norm and homogenization}\label{sec:infnorm}
Next we define an operator
\begin{equation*}
  \IN\colon\cL^\infty\to\cN_\R^\hom
\end{equation*}
that to a bounded function $\f$ on $\Xan$ attaches a homogeneous norm, the \emph{infimum norm} $\chi=\IN(\f)$. For any $m\in\N$ such that $mL$ is a line bundle, it is defined on $R_m$ by
\begin{equation}\label{equ:GN}
  \chi(s):=\inf_{v\in\Xan}\{v(s)+m\f(v)\}=\inf_{X^\an}\{m\f-\log|s|\}. 
\end{equation}
In `multiplicative' notation, this is simply the usual supremum norm 
$$
\|s\|_\f=\sup_{\Xan}|s|e^{-m\f},
$$ 
compare for instance~\cite[\S6]{BE}. The operator $\IN=\IN_L$ is increasing, and equivariant for the actions of $\R$ and $\R_{>0}$, \ie
\begin{equation}\label{equ:INequiv}
\IN(\f+c)=\IN(\f)+c,\quad\IN(t\cdot\f)=t\IN(\f)
\end{equation}
for $\f\in\cL^\infty$, $c\in\R$ and $t\in\R_{>0}$. For any $\f,\f'\in\cL^\infty$, it is also easy to see that
\begin{equation}\label{equ:GNmin}
  \IN(\f\wedge\f')=\IN(\f)\wedge \IN(\f'), 
\end{equation}
\begin{equation}\label{equ:INcontr}
\dd_\infty\left(\IN(\f),\IN(\f')\right)\le\dd_\infty(\f,\f')
\end{equation}    
\begin{equation}\label{equ:INlsc}
\IN(\f)=\IN(\f_\star), 
\end{equation}
where $\f_\star\in\cL^\infty$ denotes the \emph{lsc regularization} of $\f$, \ie the largest lsc function such that $\f_\star\le\f$.    The following key result relates homogenization and infimum norms: 

\begin{thm}\label{thm:INFS}
  For any $\chi\in\cN_\R$, we have $\IN(\FS(\chi))=\n^\hom$. 
\end{thm}

\begin{cor}\label{cor:FSisom} For all $\n,\n'\in\cN_\R$ we have $\dd_\infty(\FS(\n),\FS(\n'))=\dd_\infty(\n^\hom,\n'^{\hom})$. 
In particular, the Fubini--Study operator defines an isometric embedding of complete metric spaces 
$$
\FS\colon (\cN_\R^\hom,\dd_\infty)\hto(\cL^\infty,\dd_\infty).
$$  
\end{cor}
Recall that $\dd_\infty$ also denotes the supnorm metric on the space $\cL^\infty$ of bounded functions on $X^\an$. 

The following result settles the case $p=\infty$ of Theorem~A in the introduction: 

\begin{cor}\label{cor:FSbij} For any subgroup $\La\subset\R$, the Fubini--Study operator defines a surjective isometry $\FS\colon(\cT_\La,\dd_\infty)\twoheadrightarrow(\cH_\La,\dd_\infty)$, which factors as 
\begin{itemize} 
\item a surjective isometry $(\cT_\La,\dd_\infty)\twoheadrightarrow(\cT_{\Q\La}^{\hom},\dd_\infty)$ defined by homogenization;
\item an isomorphism $\FS\colon(\cT_{\Q\La}^\hom,\dd_\infty)\simto(\cH_\La,\dd_\infty)$, with inverse $\IN\colon(\cH_\La,\dd_\infty)\simto(\cT_{\Q\La}^\hom,\dd_\infty)$. 
\end{itemize}
For any homogeneous norm $\n\in\cN_\R^\hom$, we further have $\FS(\n)\in\cH_\La\Longleftrightarrow\n\in\cT_{\Q\La}$. 
\end{cor}
The last equivalence fails in general for non-homogeneous norms, see Example~\ref{exam:FSft}

The proof of Theorem~\ref{thm:INFS} uses the Berkovich maximum modulus principle as well as the remarks in~\S\ref{sec:vallb}. 

\begin{proof}[Proof of Theorem~\ref{thm:INFS}] After passing to a multiple, we may assume that $L$ is a line bundle and $\chi$ is a norm on $R=R(X,L)$. Let $M$ be the Berkovich spectrum of the normed ring $(R,\chi)$, \ie the set of semivaluations $w\colon R\to\R\cup\{+\infty\}$ such that $w\ge\chi$. Geometrically, $M$ sits as a compact subset of the analytification $Y^\an$ of the affine cone $Y=\Spec R$, and is obtained as the image of the unit disc bundle in the total space of $L^\vee$, \ie the blowup of $o\in Y$ (compare~\cite{Fan19}).

Since homogenization corresponds to the spectral radius seminorm construction (see Remark~\ref{rmk:specrad}), the Berkovich maximum modulus principle~\cite[Theorem~1.3.1]{BerkBook} (applied to the completion of $(R,\chi)$) yields $\n^\hom(f)=\min_{w\in M}w(f)$ for any $f\in R$, where the infimum is attained by compactness of $M$. In particular, for any $s\in R_m\smallsetminus\{0\}$, $m\ge 1$, we have $\n^\hom(s)=\min_{w\in M}w(s)$. 

  Let $M^{\mathrm{inv}}$ be the set of $k^\times$-invariant semivaluations in $M$. We have a projection $p\colon M\to M^{\mathrm{inv}}$ defined by
$$
p(w)(\sum_ms_m)=\min_mw(s_m)
$$
where $s_m\in R_m$. Thus $p(w)\le w$, so in the formula for $\n^\hom(s)$, it suffices to take the infimum over $w\in M^{\mathrm{inv}}$.  As in~\S\ref{sec:vallb} consider the projection $\pian\colon\Yan\smallsetminus\{w_o\}\to\Xan$, where $w_o$ is the trivial semivaluation at the vertex of the cone; this satisfies $w_o(f)=+\infty$ for $f\in\bigoplus_{m>0}R_m$ and $w_o(f)=0$ for $f\in R\setminus\bigoplus_{m>0}R_m$. For any $v\in\Xan$, the set of $k^\times$-invariant points in $(\pian)^{-1}(v)$ is of the form $\{w_{v,c}\}_{c\in\R}$, where $w_{v,c}$ is the unique $k^\times$-invariant point such that $w_{v,c}(s)=v(s)+c m$ for $s\in R_m$, $m\ge1$, see~\eqref{equ:valext}. 

It follows that $M^{\mathrm{inv}}\subset M\subset Y^\an$ is the set of semivaluations $w_{v,c}$, where $v\in\Xan$ and $v(s)+m c\ge\chi(s)$ for all $s\in R_m$, $m\ge 1$. Note that this condition on $c$ means precisely that $c\ge\FS(\chi)(v)$. Altogether, this means that if $s\in R_m$, $m\ge 1$, then 
  \begin{align*}
    \n^\hom(s)
    &=\inf\{v(s)+m c\mid v\in\Xan,\,c\ge\FS(\chi)(v)\}\\
    &=\inf\{v(s)+m\FS(\chi)(v)\mid v\in\Xan\}\\
    &=\IN(\FS(\chi))(s),
  \end{align*}
  which completes the proof.
\end{proof}

\begin{proof}[Proof of Corollary~\ref{cor:FSisom}] By Proposition~\ref{prop:FSH}, we may assume that $\n,\n'$ are homogeneous.  By Theorem~\ref{thm:INFS}, \eqref{equ:INcontr} and~\eqref{equ:FScontr}, we then have 
$$
\dd_\infty(\n,\n')=\dd_\infty(\IN(\FS(\n)),\IN(\FS(\n')))\le\dd_\infty(\FS(\n),\FS(\n'))\le\dd_\infty(\n,\n').
$$
Thus equality holds everywhere, and the result follows. 
\end{proof}

\begin{proof}[Proof of Corollary~\ref{cor:FSbij}] By Corollary~\ref{cor:FSisom} and Lemma~\ref{lem:homogft}, the Fubini--Study operator defines a surjective isometry $\FS\colon(\cT_\La,\dd_\infty)\twoheadrightarrow(\cH_\La,\dd_\infty)$, which factors as homogenization followed by $\FS\colon(\cT_{\Q\La}^\hom,\dd_\infty)\simto(\cH_\La,\dd_\infty)$, whose inverse is necessarily given by $\IN$, by Theorem~\ref{thm:INFS}. This implies that homogenization defines a surjective isometry $(\cT_\La,\dd_\infty)\twoheadrightarrow(\cT_{\Q\La}^\hom,\dd_\infty)$. 

The final assertion follows, by injectivity of $\FS$ on $\cN_\R^\hom$. 

\end{proof}

%
\subsection{Continuous norms}\label{sec:cont} 

Building on the previous results, we are now in a position to characterize the $\dd_\infty$-closure of the set $\cT_\Z$ of test configurations, as follows. 

\begin{thm}\label{thm:cont} For any norm $\n\in\cN_\R$, the following are equivalent:
\begin{itemize}
\item[(i)] $\n$ lies in the $\dd_\infty$-closure of $\cT_\Z$;
\item[(ii)] $\n$ lies in the $\dd_\infty$-closure of $\cT_\R$;
\item[(iii)] the canonical approximants $(\n_d)$ satisfy $\dd_\infty(\n_d,\n)\to 0$; 
\item[(iv)] $\FS(\n)$ is continuous;
\item[(v)] $\FS(\n_d)\to\FS(\n)$ uniformly on $X^\an$.
\end{itemize}
\end{thm}

\begin{defi} We say that $\n\in\cN_\R$ is \emph{continuous} when the equivalent properties of Theorem~\ref{thm:cont} holds. 
\end{defi}
The set $\cN_\R^{\cont}$ of continuous norms is thus the $\dd_\infty$-closure of $\cT_\Z$ (or $\cT_\R$); it is a strict subset of $\cN_\R$ as soon as $\dim X\ge 1$ (see Example~\ref{exam:subvar} below). 

\begin{proof}[Proof of Theorem~\ref{thm:cont}] We trivially have (i)$\Rightarrow$(ii). Assume (ii), and pick $\e>0$. In view of~\eqref{equ:GIas}, we can find $\n'\in\cT_\R$ and $d\ge 1$ such that 
\begin{equation}\label{equ:dinftymr}
\dd_\infty(\n|_{R_{md}},\n'|_{R_{md}})\le mr\e\text{ for all }m\ge 1.
\end{equation}
Replacing $d$ with a multiple, we can further assume $\n'_d=\n'$. For $m=1$, \eqref{equ:dinftymr} yields $\n'\le\n+r\e$ on $R_d$, and hence $\n'=\n'_d\le\n_d+mr\e$ on $R_{md}$ for all $m\ge 1$. On the other hand, \eqref{equ:dinftymr} yields $\n\le\n'+mr\e$ on $R_{md}$, which proves $\n_d\le\n\le\n_d+2mr\e$ on $R_{md}$ for all $m\ge 1$. This implies $\dd_\infty(\n_d,\n)\le\e$. This proves (ii)$\Rightarrow$(iii), the converse being obvious since $\n_d\in\cT_\R$. 

Assume (ii). To prove (i), we may replace $\n$ with its round-down and assume $\n\in\cN_\Z$ (see Example~\ref{exam:dinftyround}. Its canonical approximants then satisfy $\n_d\in\cT_\Z$, and hence (ii)$\Rightarrow$(i), thanks to (iii). 

Since $\FS(\n)$ is the pointwise limit of the increasing net of continuous functions $(\FS(\n_d))$, Dini's lemma yields (iv)$\Leftrightarrow$(v). Finally, we claim that 
\begin{equation}\label{equ:FSapprox}
\dd_\infty(\n_d,\n)=\dd_\infty(\FS(\n_d),\FS(\n))
\end{equation}
for $d$ sufficiently divisible, which will show (iii)$\Leftrightarrow$(v). Note that 
$$
\n_d\le\n\le\n^\hom\Longrightarrow\dd_\infty(\n_d,\n)\le\dd_\infty(\n_d,\n^\hom),
$$
see~\eqref{equ:GIorder}. Now $\dd_\infty(\n_d,(\n_d)^\hom))=0$ (see Lemma~\ref{lem:homogft}), and hence 
$$
\dd_\infty(\n_d,\n^\hom)=\dd_\infty((\n_d)^\hom,\n^\hom)=\dd_\infty(\FS(\n_d),\FS(\n)),
$$
by~Corollary~\ref{cor:FSisom}. This shows $\dd_\infty(\n_d,\n)\le\dd_\infty(\FS(\n_d),\FS(\n))$, while the converse holds by~\eqref{equ:FScontr}. This shows~\eqref{equ:FSapprox}, and concludes the proof. 
\end{proof}

\begin{exam}\label{exam:subvar} To each subvariety $Z\varsubsetneq X$ we associate a norm $\n=\n_Z\in\cN_\Z^\hom$ by setting, for each nonzero section $s\in R_m$ with $m$ sufficiently divisible
$$
\n(s)=\left\{
  \begin{array}{ll}
    m & \text{if } s|_Z\equiv 0 \\
    0 & \text{otherwise}. 
  \end{array}\right.
$$
We claim that $\n$ is not continuous. Indeed, using the description of $\n_d$ as a quotient norm, it is easy to check that any $s\in R_{dm}$ locally lies in $I_Z^p$ with $p:=\n_d(s)\in\N$. Choosing $s\in R_{dm}=\Hnot(X,rmL)$ that locally belongs to $I_Z$ but not $I_Z^2$ (which is possible for any $m$ large enough, since $L$ is ample), we get $\n_d(s)=1$, while $\n(s)=rm$. This shows $\dd_\infty(\n|_{R_{dm}},\n_d|_{R_{dm}})\ge rm-1$, and hence $\dd_\infty(\n,\n_d)\ge 1$, which proves the claim. 

Alternatively, one can show that $\FS(\n)$ is identically $1$ on $X^\an\setminus Z^\an$, and $0$ on $Z^\an$, and hence is not continuous. 
\end{exam}

\begin{exam}\label{exam:FSft} As a variant of Example~\ref{exam:subvar}, consider the norm $\n\in\cN_\R$ defined for $s\in R_m\setminus\{0\}$ by 
$$
\n(s)=\left\{
  \begin{array}{ll}
    m & \text{if } s|_Z\equiv 0 \\
    m-\sqrt m & \text{otherwise},
  \end{array}\right.
$$
which is indeed a norm, by subadditivity of $m\mapsto\sqrt m$. Then $\n^\hom=\n_\triv+1$, hence $\FS(\n)=\FS(\n^\hom)\equiv 1$. In particular, $\FS(\n)\in\cH_\R$; however, $\n$ is not of finite type. Indeed, 
$$
\dd_\infty(\n|_{R_m},\n^\hom|_{R_m})=\sqrt m
$$
is not bounded (see Lemma~\ref{lem:homogft}). 
\end{exam}

\begin{rmk} It follows from Example~\ref{exam:subvar} that the set $\cT_\Z$ of test configurations is never $\dd_\infty$-dense in $\cN_\R$ when $\dim X\ge 1$. In contrast, $\cT_\Z$ is dense with respect to any of the weaker pseudometrics $\dd_p$, $p\in[1,\infty)$ to be introduced in Section~\ref{sec:spectral} (see Corollary~\ref{cor:tcdense}). 
\end{rmk}

We next analyze the behavior of homogenization on continuous norms. 

\begin{prop}\label{prop:homcont} For each $\n\in\cN_\R$, we have 
$$
\n\in\cN_\R^\cont\Longleftrightarrow\n^\hom\in\cN_\R^\cont\Longrightarrow\dd_\infty(\n,\n^\hom)=0. 
$$
Further, homogenization induces a surjective isometry $(\cN_\R^\cont,\dd_\infty)\twoheadrightarrow(\cN_\R^{\cont,\hom},\dd_\infty)$. 
\end{prop}
Here $\cN_\R^{\cont,\hom}:=\cN_\R^\cont\cap\cN_\R^\hom$ denotes the set of continuous homogeneous norms. 

\begin{proof} The first equivalence follows from Proposition~\ref{prop:FSH} and Theorem~\ref{thm:cont}~(iv). By Lemma~\ref{lem:homogft}, $\dd_\infty(\n,\n^\hom)=0$ holds on $\cT_\R$. By $\dd_\infty$-continuity of homogenization (see Lemma~\ref{lem:homcomp}), this extends to the $\dd_\infty$-closure $\cN_\R^\hom$. This proves the second implication, which in turn yields the last point, by the triangle inequality.
\end{proof}

When $\n\in\cN_\R$ is not continuous, the property $\dd_\infty(\n,\n^\hom)=0$ fails in general; in other words, homogenization is not a $\dd_\infty$-isometry on the whole space $\cN_\R$: 

\begin{exam}\label{exam:don} Pick a subvariety $Z\varsubsetneq X$, and set for $s\in R_m\setminus\{0\}$ 
$$
\n(s)=\left\{
  \begin{array}{ll}
    m & \text{if } s\in I_Z^2 \\
    m/2 & \text{if } s\in I_Z\setminus I_Z^2\\
    0 & \text{if } s\notin I_Z
  \end{array}\right.
$$
As one easily checks, this defines a norm $\n\in\cN_\Q$, such that $\n^\hom=\n_Z$ is the norm in Example~\ref{exam:subvar}. Further, $\dd_\infty(\n|_{R_m},\n^\hom|_{R_m})=m/2$ for $m$ sufficiently divisible, and hence $\dd_\infty(\n,\n^\hom)=1/2$. 
\end{exam}

Finally, recall from~\cite{trivval} that the space $\CPSH$ of \emph{continuous (bounded) $L$-psh functions} on $X^\an$ can be described as the closure of $\cH_\R$ (or, equivalently, $\cH_\Q=\cH_\Z$) with respect to uniform convergence (see also~\S\ref{sec:Lpsh} below). We show:  

\begin{thm}\label{thm:contisom} The Fubini--Study and infimum norm operators induce inverse isomorphic isometries
$\FS\colon(\cN_\R^{\cont,\hom},\dd_\infty)\simto(\CPSH,\dd_\infty)$, $\IN\colon(\CPSH,\dd_\infty)\simto(\cN_\R^{\cont,\hom},\dd_\infty)$. 
\end{thm}

\begin{proof} By Corollary~\ref{cor:FSisom}, the Fubini--Study operator defines an isometric embedding of complete metric spaces $\FS\colon(\cN_\R^\hom,\dd_\infty)\hto(\cL^\infty,\dd_\infty)$, which thus maps the closure of any subset onto the closure of its image. Now $\FS(\cT_\R^\hom)=\cH_\R$ (see Proposition~\ref{prop:FSft}), where the closure of $\cT_\R^\hom$ in $(\cN_\R^\hom,\dd_\infty)$ is $\cN_\R^{\cont,\hom}$ (by Proposition~\ref{prop:homcont}) and the closure of $\cH_\R$ in $(\cL^\infty,\dd_\infty)$ is $\CPSH$. It follows that $\FS\colon(\cN_\R^{\cont,\hom},\dd_\infty)\simto(\CPSH,\dd_\infty)$ is an isometric isomorphism, whose inverse is necessarily given by $\IN$, by Theorem~\ref{thm:INFS}. 
\end{proof}

%
%
\subsection{The Fubini--Study envelope}

As in~\cite[\S7.5]{BE} and~\cite[\S5.3]{trivval}, we define the \emph{Fubini--Study envelope} of a bounded function $\f\in\cL^\infty$ as the pointwise supremum
\begin{equation}\label{equ:FSenv}
\qq(\f)=\qq_L(\f):=\sup\{\p\in\cH_\R\mid \p\le\f\}. 
\end{equation} 
Since any $\p\in\cH_\R$ is a uniform limit of functions in $\cH_\Q$, one can replace $\cH_\R$ with $\cH_\Q$ in this definition. We note that   
\begin{equation}\label{equ:FSequiv}
\qq_{dL}(t\f)=d\qq_L(\f),\quad\qq(t\cdot\f)=\qq(\f),\quad\qq(\f+c)=\qq(\f)+c
\end{equation}
for all $d\in\Q_{>0}$, $t\in\R_{>0}$, $c\in\R$,    and refer to~\S\ref{sec:env} for more information. We view the next result as a `dual' version of Proposition~\ref{prop:FSH}.
\begin{prop}\label{prop:GNQ}
  For any $\f\in\cL^\infty$ we have $\IN(\f)=\IN(\qq(\f))$.
  \end{prop}
This is in fact a special case of~\cite[Lemma~7.23]{BE} but we repeat the simple argument for the convenience of the reader. 
\begin{lem}\label{lem:GNQ}
  If $\f\in\cL^\infty$ and $s\in R_m$ with $m$ sufficiently divisible, then $\log|s|\le m\f$ iff $\log|s|\le m\qq(\f)$.
\end{lem}

\begin{proof} We may assume $m=1$. Since $\qq(\f)\le\f$, we only need to prove the direct implication. For $t\in\R$, set $\p_t=\max\{\log|s|,-t\}$. Then $\p_t\in\cH_\R$, and $\p_t\le\f$ for $t\gg0$ since $\f$ is bounded. Thus  $\p_t\le\qq(\f)$ by the definition of $\qq$, so $\log|s|\le\p_t\le\qq(\f)$.
\end{proof}
\begin{proof}[Proof of Proposition~\ref{prop:GNQ}]
  Pick $s\in R_m$ with $m$ sufficiently divisible. We must prove that $\la:=\inf_{\Xan}(m\f-\log|s|)$ equals $\la':=\inf_{X^\an}(m\qq(\f)-\log|s|)$. Since $\qq(\f)\le\f$ we have $\la'\le\la$. The reverse inequality follows from Lemma~\ref{lem:GNQ} applied to the bounded function $\f-m^{-1}\la$, together with~\eqref{equ:FSequiv}. 
\end{proof}

We similarly have a dual version of Theorem~\ref{thm:INFS}: 

\begin{prop}\label{prop:FSIN}
  For any $\f\in\cL^\infty$, we have $\FS(\IN(\f))=\qq(\f)$. 
\end{prop}

\begin{proof} After passing to a multiple, we may assume that $L$ is a line bundle, so that $\chi:=\IN(\f)$ is a norm on $R=R(X,L)$. For all $m\ge 1$ and $s\in R_m\smallsetminus\{0\}$, we have $\log|s|\le m\f-\chi(s)$ on $X^\an$, by definition of the infimum norm. By Lemma~\ref{lem:GNQ} and~\eqref{equ:FSequiv}, this yields $\log|s|\le m\qq(\f)-\chi(s)$, and hence 
$$
\FS(\chi)=\sup\{m^{-1}(\log|s|+\chi(s))\mid m\ge 1,\,s\in R_m\smallsetminus\{0\}\}\le\qq(\f).
$$
For the reverse inequality, pick any $\p\in\cH_\R$ with $\p\le\f$. Since $\FS(\cT_\R)=\cH_\R$ (see Proposition~\ref{prop:FSft}), there exists $m\ge1$ and a norm $\chi'$ on $R_m$ such that 
$$
m\p=\FS_{mL}(\chi')=\sup_{s\in R_m\setminus\{0\}}\{\log|s|+\chi'(s)\}. 
$$
Since $\p\le\f$, this gives $\log|s|+\chi'(s)\le m\f$, \ie $\chi'\le\chi|_{R_m}$ on $R_m$. As a result, 
$$
\p=m^{-1}\FS_{mL}(\chi')\le m^{-1}\FS_{mL}(\chi|_{R_m})=\FS_m(\chi)\le\FS(\chi),
$$
which completes the proof.  
\end{proof}

Combining Theorem~\ref{thm:INFS} and Proposition~\ref{prop:FSIN} with Propositions~\ref{prop:FSH} and~\ref{prop:GNQ}, we also obtain
\begin{cor}\label{cor:3opsis1}
  We have $\FS\circ\IN\circ\FS=\FS$ on $\cN_\R$, and
  $\IN\circ\FS\circ\IN=\IN$ on $\cL^\infty$.
\end{cor}

%
%
%
%
\section{Spectral analysis}\label{sec:spectral}
In this section we define a volume function $\vol\colon\cN_\R\to\R$ as well as pseudo-metrics $\dd_p$, $p\in [1,\infty)$, on the space $\cN_\R$ of  norms on section rings of multiples of $L$. 
Much of the material is studied for more general non-Archimedean ground fields in~\cite{CM15,BE}, but we present the details for the convenience of the reader.
%
%
\subsection{The finite-dimensional case}
We first describe the space $\cN_\R(V)$ of non-Archimedean norms on a $k$-vector space $V$ of dimension $N<\infty$, essentially following~\cite{GIT,BE}.

Pick a norm $\chi\in\cN_\R(V)$, and a $\chi$-orthogonal basis $(e_j)_{1\le j\le N}$ of $V$. After permutation, we may assume that the sequence $\la_j(\chi):=\chi(e_j)$, $j=1,\dots,N$ satisfies
 \begin{equation*}
\la_1(\n)\ge\dots\ge\la_N(\n). 
  \end{equation*}
It is then independent of the choice of orthogonal basis and is called the \emph{spectrum} of $\chi$ (\ie the `jumping values' of the associated filtration in the terminology of~\cite{BC}). In terms of~\eqref{equ:laminmax} we have 
$$
\la_1(\n)=\la_{\max}(\n),\quad\la_N(\n)=\la_{\min}(\n). 
$$
The \emph{volume} of $\n$ is defined as the mean value\footnote{Note that a different normalization is used in~\cite{BE,BGM}, where the volume is defined as the sum of the elements of the spectrum.} of its spectrum, \ie 
$$
\vol(\n):=\frac{1}{N}\sum_j\la_j(\chi).
$$
For any basis $(e_j)$ of $V$ we have $\vol(\chi)\ge N^{-1}\sum_j\chi(e_j)$ with equality iff $(e_j)$ is $\chi$-orthogonal.

More generally, any two norms $\chi,\chi'$ admit a common orthogonal basis $(e_j)$. The \emph{relative spectrum} of $\n$ with respect to $\n'$ is the sequence
$$
\la_1(\n,\n')\ge\dots\ge\la_N(\n,\n')
$$
obtained by reordering $\left(\chi(e_j)-\chi'(e_j)\right)_{1\le j\le N}$, and the 
\emph{spectral measure} of $\chi$ with respect to $\chi'$ is the corresponding probability measure
$$
\sigma(\n,\n'):=\frac{1}{N}\sum_j\d_{\la_j(\n,\n')}.
$$
Its barycenter satisfies
\begin{equation}\label{equ:barvol}
\int_\R\la\,d\sigma(\n,\n')=\frac 1N\sum_j\la_j(\n,\n')=\vol(\n)-\vol(\n'). 
\end{equation}
When $\n'=\chi_\triv$ is the trivial norm, we simply write 
$$
\sigma(\n)=\sigma(\n,\chi_\triv)=\frac{1}{N}\sum_j\d_{\la_i(\n)},
$$
and call it the \emph{spectral measure of $\n$}.    In terms of the associated $\R$-filtration $\Filt^\la V=\{\n\ge \la\}$, we have 
\begin{equation}\label{equ:sigmafilt}
\sigma(\n)=\frac 1N\sum_{\la\in\R}\dim(\Filt^\la V/\Filt^{>\la} V)\,\d_\la.
\end{equation}

\medskip
To any basis $\ba=(e_i)$ of $V$ is associated an \emph{apartment}  $\A_\ba\subset\cN_\R(V)$, defined as the set of norms diagonalized in this basis. We then have a canonical parametrization 
$$
\iota_\ba\colon\R^N\simto\A_\ba, 
$$ 
and a \emph{Gram--Schmidt retraction}
$$
\rho_\ba\colon\cN_\R(V)\to\A_\ba. 
$$
The map $\iota_\ba$ sends $(\la_j)\in\R^N$ to the unique $\n\in\A_\ba$ such that $\n(e_i)=\la_i$, while $\rho_\ba$ sends a norm $\n$ to the unique norm $\rho_\ba(\n)\in\A_\ba$ such that
$$
\rho_\ba(\n)(e_i)=\sup_{a\in k^N}\n(e_i+\sum_{j<i} a_j e_j). 
$$
\ie the norm induced, via the basis $\ba$, from the natural subquotient norm on the graded object of the complete flag defined by $\ba$. By additivity of the volume in exact sequences, we have 
\begin{equation}\label{equ:detrho}
\vol(\rho_\ba(\n))=\vol(\n),
\end{equation}
see~\cite[Lemma 2.12]{BE}. Each $\A_\ba$ is trivially preserved by the translation action of $\R$, the scaling action by $\R_{>0}$, and by the operation $(\chi,\chi')\mapsto\chi\wedge\chi'$. Moreover,
$$
\n\le\n'\Longrightarrow\rho_\ba(\n)\le\rho_\ba(\n').
$$
%
%
\subsection{Metrics on the space of norms}

Generalizing the classical construction of the Tits metric on the Euclidean building $\cN_\R(V)$ (see for instance~\cite[\S2.2]{GIT}), it is shown in~\cite[Theorem 3.1]{BE} that each $\mathfrak{S}_N$-invariant norm $\tau$ on $\R^N$ induces a unique metric $\dd_\tau$ on $\cN_\R(V)$ such that 
$\iota_\ba:(\R^N,\tau)\hto(\cN_\R(V),\dd_\tau)$ is an isometry for any basis $\ba$. 
It has the property that $\rho_\ba\colon\cN_\R(V)\to\A_\ba$ is a contraction. All metrics on $\cN_\R(V)$ obtained this way are equivalent. They turn $\cN_\R(V)$ into a metric space that is complete, but not locally compact. 

In particular, for each $p\in[1,\infty]$ we define a metric $\dd_p$ on $\cN_\R(V)$ by setting for any two norms $\n,\n'$ with relative spectrum $(\la_i)=(\la_i(\chi,\chi'))$
\begin{equation}\label{e418}
\dd_p(\n,\n'):=(N^{-1}\sum_i |\la_i|^p)^{1/p}
\end{equation}
for $p\in[1,\infty)$, and 
$$
\dd_\infty(\n,\n'):=\max_i|\la_i|.
$$
Thus $\dd_p(\chi,\chi')$ is the $L^p$-norm of the identity with respect to the spectral measure $\sigma(\chi,\chi')$.
Note that 
\begin{equation}\label{equ:dpcomp}
\dd_1\le\dd_p\le\dd_\infty^{1-\tfrac 1p}\dd_1^{\tfrac 1p}\le\dd_\infty
\end{equation}
on $\cN_\R(V)$ for $p\in(1,\infty)$. The metric $\dd_2$ is the Tits metric mentioned above, while $\dd_\infty$ coincides with the Goldman--Iwahori metric~\eqref{equ:GI}. Our main interest lies in the metric $\dd_1$, which is closely related to the volume: 
\begin{lem}\label{lem:d1vol} For all $\n,\n'\in\cN_\R(V)$ and $p\in [1,\infty)$ we have 
\begin{equation}\label{equ:dppyth}
\dd_p(\n,\n')^p=\dd_p(\n,\n\wedge\n')^p+\dd_p(\n\wedge\n',\n')^p.
\end{equation}
For $p=1$, we further have 
\begin{equation}\label{equ:d1vol}
d_1(\n,\n')=\vol(\n)+\vol(\n')-2\vol(\n\wedge\n').
\end{equation}
\end{lem}
\begin{proof} The first assertion follows from the fact that the minimum $\n\wedge\n'$ of two norms in an apartment $\A_\ba\simeq\R^N$ is computed component-wise, and the trivial identity
$$
\sum_i|\la_i-\la'_i|^p=\sum_i|\la_i-\min\{\la_i,\la'_i\}|^p+\sum_i |\min\{\la_i,\la'_i\}-\la'_i|^p.
$$
for all $\la,\la'\in\R^N$. On the other hand, it follows from~\eqref{equ:barvol} that
$\n\ge\n'\Longrightarrow\dd_1(\n,\n')=\vol(\n)-\vol(\n')$,
and~\eqref{equ:d1vol} follows. 
\end{proof}
The volume function is trivially $1$-Lipschitz with respect to $\dd_1$, \ie
\begin{equation}\label{equ:vollip}
|\vol(\n)-\vol(\n')|\le\dd_1(\n,\n')
\end{equation}
for all $\n,\n'\in\cN_\R(V)$. This is also the case for the min operator: 

\begin{lem}\label{lem:d1min}
  Let $\chi_i$, $\chi'_i$, $i=1,2$, be norms on $V$. Then 
  \begin{equation}\label{e421}
    \dd_1(\chi_1\wedge\chi_2,\chi'_1\wedge\chi'_2)
    \le\dd_1(\chi_1,\chi'_1)+\dd_1(\chi_2,\chi'_2).
  \end{equation}
\end{lem}

\begin{proof} First assume $\chi_i\ge\chi'_i$, $i=1,2$. Pick a basis $\ba$ such that $\chi'_1,\chi'_2\in\A_\ba$. Lemma~\ref{lem:d1vol} together with~\eqref{equ:detrho} show that replacing $\chi_i$ by $\rho_\ba(\chi_i)$, $i=1,2$ does not change the right-hand side of~\eqref{e421}. As for the left-hand side, $\rho_\ba\colon\cN_\R(V)\to\A_\ba$ being order preserving implies 
 \begin{equation*}
    \rho_\ba(\chi_1)\wedge\rho_\ba(\chi_2)
    \ge\rho_\ba(\chi_1\wedge\chi_2)
    \ge\chi'_1\wedge\chi'_2,
  \end{equation*}
which shows that the left-hand side of~\eqref{e421}
can only increase upon replacing $\chi_i$ by $\rho_\ba(\chi_i)$, $i=1,2$, using again~\eqref{equ:detrho} and~\eqref{equ:d1vol}.
  
As a result, we may in fact assume that all four norms belong to $\A_\ba$. Write $\chi_i(e_j)=\la_{i,j}$ and $\chi'_i(e_j)=\la'_{i,j}$ for
  $1\le j\le N$ and $i=1,2$. Then $\la_{i,j}\ge \la'_{i,j}$ for all $i,j$, and
  we must prove that
  \begin{equation*}
    \sum_j\la_{1,j}\wedge\la_{2,j}-\sum_j\la'_{1,j}\wedge\la'_{2,j}
    \le
    \sum_j(\la_{1,j}-\la'_{1,j})+\sum_j(\la_{2,j}-\la'_{2,j});
  \end{equation*}
  this is straightforward.

   \smallskip
  Finally consider arbitrary norms. Set $\chi''_i=\chi_i\wedge\chi'_i$
  for $i=1,2$. By~\eqref{equ:dppyth} we have 
  \begin{equation*}
    \dd_1(\chi_1\wedge\chi_2,\chi'_1\wedge\chi'_2)
    =\dd_1(\chi_1\wedge\chi_2,\chi''_1\wedge\chi''_2)
    +\dd_1(\chi'_1\wedge\chi'_2,\chi''_1\wedge\chi''_2)
  \end{equation*}
  and $\chi_i,\chi'_i\ge\chi''_i$, for $i=1,2$,
  so~\eqref{e421} follows from what precedes, together with~\eqref{equ:dppyth}. 
\end{proof}
%
%
\subsection{Spectral measures and volume}
Now we return to the setting of a projective variety $X$ and an ample $\Q$-line bundle $L$ on $X$. The following equidistribution result is a special case of a result of Chen--Maclean~\cite{CM15}, which deals with general non-Archimedean fields. 

\begin{thm}\label{thm:CM}
For any two norms $\chi,\chi'\in\cN_\R$, the scaled spectral measures 
$$
(1/m)_\star \sigma(\chi|_{R_m},\chi'|_{R_m})
$$ 
have uniformly bounded support, and they admit a weak limit. 
\end{thm}
The limit is taken with respect to the partial order by divisibility. If $L$ is an actual line bundle and $\chi,\chi'$ are norms on $R(X,L)$, then the limit also exists as $m\to\infty$ in the usual total ordering.

\begin{defi}\label{defi:spectral} For any $\chi,\chi'\in\cN_\R$, the \emph{spectral measure of $\chi$ with respect to $\chi'$} is the compactly supported (Borel) probability measure on $\R$ defined as 
$$
\sigma(\chi,\chi'):=\lim_m(1/m)_\star \sigma(\chi|_{R_m},\chi'|_{R_m}). 
$$
The \emph{spectral measure} of $\chi$ is $\sigma(\chi):=\sigma(\chi,\chi_\triv)$, and the \emph{volume} of $\chi$ is the barycenter
$$
\vol(\chi)=\int_\R\la\,d\sigma(\chi).
$$
\end{defi}
By~\eqref{equ:barvol}, we have 
\begin{equation}\label{equ:vollim}
\vol(\chi)=\lim_m m^{-1}\vol(\chi|_{R_m}). 
\end{equation}

\begin{exam}\label{exam:ess} For any $v\in X^\lin$ with associated norm $\n_v\in\cN_\R^\hom$, the spectrum of $\n_v|_{R_m}$ is the \emph{vanishing sequence} of $R_m$ with respect to $v$ as defined in~\cite{BKMS}, \ie the (finite) set of values of $v$ on nonzero elements of $R_m$, counted with multiplicity, and 
\begin{equation}\label{equ:S}
 \ess(v)=\ess_L(v):=\vol(\n_v)
\end{equation}
coincides with the \emph{expected vanishing order} of~\cite{BlJ} (see also~\cite{MR15}). 
\end{exam}

The existence of the spectral measure $\sigma(\chi)$ (called the \emph{limit measure} of the corresponding filtration in~\cite[\S5.1]{BHJ1}) follows from~\cite[Theorem A]{BC}.    When $\chi\in\cT_\Z$, the limit measure coincides with the Duistermaat--Heckman measure of the corresponding test configuration, see~\cite[Proposition 3.12]{BHJ1}.
  
As we shall see, a simple trick borrowed from~\cite{CM15} reduces the proof of Theorem~\ref{thm:CM} to this special case $\chi'=\chi_\triv$.  

\begin{proof}[Proof of Theorem~\ref{thm:CM}] The uniform boundedness part is a direct consequence of the linear boundedness condition that we impose on norms in $\cN_\R(R)$. Set $N_m:=\dim R_m$, denote by $(\la_{m,j})_{1\le j\le N_m}$ the spectrum of $\chi|_{R_m}$ with respect to $\chi'|_{R_m}$, and set
$$
\sigma_m:=(1/m)_\star \sigma(\chi|_{R_m},\chi'|_{R_m})=\frac{1}{N_m}\sum_{j=1}^{N_m}\d_{m^{-1}\la_{m,j}}.
$$
As is well-known, in order to prove convergence of $\sigma_m$, it suffices to show that 
  \begin{equation*}
    \int_\R\min\{\la,c\}\,d\sigma_m
    =\frac1{mN_m}\sum_{j=1}^{N_m}\max\{\la_{m,j},mc\}
  \end{equation*}
  
  converges for all $c\in\R$, see~\cite[Proposition~5.1]{CM15}.
  But $(\min\{\la_{m,j},cm\})_j$ is the 
  spectrum of $\chi|_{R_m}\wedge(\chi'|_{R_m}+cm)$ with respect to $\n'|_{R_m}$.  
   Replacing $\chi$ with $\chi\wedge(\chi'+c)$ (where $\R$ acts by translation according to~\eqref{equ:gradedtrans}), we are reduced to proving that the barycenter
  $$
  (mN_m)^{-1}\sum_j\la_{m,j}
  $$
  of the measure $\sigma_m$ converges. Now, this barycenter is the difference of the barycenters of $(1/m)_\star \sigma(\chi|_{R_m})$ and $(1/m)_\star \sigma(\chi'|_{R_m})$, 
  each of which admits a limit by~\cite[Theorem A]{BC}, and we are done.   
 \end{proof}
 
The proof of Theorem~\ref{thm:CM} shows that 
\begin{equation}\label{e423}
  \int\min\{\la,c\}\,\sigma(\chi,\chi')(d\la)
  =\vol\left(\chi\wedge(\n'+c)\right)-\vol(\chi')
\end{equation}
for all $\chi,\chi'\in\cN_\R$ and $c\in\R$. Some further properties of the spectral measure $\sigma(\chi)$ are described by the following result.  

\begin{thm}\label{thm:volandsup} Pick $\n\in\cN_\R$, with associated filtration $\Filt^\la R_m=\{s\in R_m\mid\n(s)\ge\la\}$. Then: 
\begin{itemize}
\item[(i)] for each $\la\in\R$, $\dim\Filt^{m\la} R_m/\dim R_m$ admits a limit $\vol(\n\ge\la)\in [0,1]$ as $m\to\infty$; 
\item[(ii)] the function $\la\mapsto\vol(\n\ge\la)^{1/n}$ is positive and concave on $(-\infty,\la_{\max}(\n))$, and vanishes on $(\la_{\max}(\n),+\infty)$; 
\item[(iii)] $\sigma(\n)=-\frac{d}{d\la}\vol(\n\ge\la)$ in the sense of distributions;
\item[(iv)] $\supp\sigma(\n)=[\la_{\min}(\n),\la_{\max}(\n)]$ with 
\begin{equation}\label{equ:laminmes}
\la_{\min}(\n):=\inf\left\{\la\in\R\mid\vol(\n\ge\la)<1\right\};
\end{equation}
\item[(v)] for any $a\le\la_{\min}(\n)$, \ie such that $\vol(\n\ge a)=1$, we have 
$$
\vol(\n)=a+\int_a^{+\infty}\vol(\n\ge\la)\,d\la=a+\int_a^{\la_{\max}(\n)}\vol(\n\ge\la)\,d\la.
$$
\end{itemize}
\end{thm}
\begin{proof} Properties (i)--(iv) are direct consequence of~\cite{BC} (see~\cite[\S 3.1]{BHJ1}. To see (v), set $b:=\la_{\max}(\n)$, $f(\la):=\vol(\n\ge\la)$, and pick a cut-off function $\theta\in C^\infty_c(\R)$ such that $\theta\equiv 1$ on $[a,b]$. Since $\theta(\la)\la$ is smooth and compactly supported, (iii) and (iv) yield
$$
\vol(\n)=\int_\R\la\,d\sigma(\n)=-\int_\R\theta(\la)\la f'(\la)\,d\la=\int_\R(\theta(\la)\la)'f(\la)\,d\la=\int_\R(\theta'(\la)\la+\theta(\la))f(\la)\,d\la.
$$
Since $f(\la)=1$ for $\la\le a$, $f(\la)=0$ for $\la\ge b$ and $\theta(\la)=1$ for $\la\ge a$, this is equal to 
$$
\int_{-\infty}^a(\theta'(\la)\la+\theta(\la))\,d\la+\int_a^{+\infty} f(\la)\,d\la=a+\int_a^{+\infty} f(\la)\,d\la=a+\int_a^b f(\la)\,d\la,
$$
by integration by parts. This proves (v). 
\end{proof}

The next result shows how spectral measures behave under operations on norms. It follows from elementary computations of spectra in joint orthogonal bases; the details are left to the reader.
\begin{prop}\label{P401}
  Let $\chi, \chi'\in\cN_\R$, and pick $c\in\R$, $t\in\R_{>0}$. Then:
  \begin{itemize}
  \item[(i)]
    $\sigma(\chi',\chi)$ is the pushforward of 
    $\sigma(\chi,\chi')$ 
    under $\la\mapsto-\la$;
  \item[(ii)]
    $\sigma(\chi+c,\chi')$ is the pushforward of
    $\sigma(\chi,\chi')$ under $\la\mapsto\la+c$;
  \item[(iii)]
    $\sigma(\chi,\chi\wedge\chi')$ is the 
pushforward 
    of $\sigma(\chi,\chi')$ under $\la\mapsto\max\{\la,0\}$;
    \item[(iv)] $\sigma(t\chi,t\chi')$ is the pushforward of $\sigma(\chi,\chi')$ by $\la\mapsto t\la$. 
    \end{itemize}
\end{prop}
\begin{rmk}
  In~\cite[Theorem 3.3]{BLXZ} (which appeared after the first version of this article was posted), the authors construct a natural \emph{joint spectral measure} $\rho(\chi,\chi')$ on $\R^2$ associated to any pair $\chi,\chi'\in\cN_\R$, that encodes the asymptotic behavior of the spectra of $\chi$ and $\chi'$ in jointly orthogonal bases. The spectral measure $\sigma(\chi,\chi')$ is the pushforward of $\rho(\chi,\chi')$ under the map $\R^2\to\R$ given by $(\la,\la')\mapsto\la-\la'$.
\end{rmk}
%
%
%
%
\subsection{The $\dd_p$-pseudometrics and asymptotic equivalence}\label{sec:asympequ}
Pick $p\in[1,\infty)$ and $\chi,\chi'\in\cN_\R$. By definition, we have 
$$
\dd_p(\chi|_{R_m},\chi'|_{R_m})^p=\int_\R|\la|^p\,d\sigma(\chi|_{R_m},\chi'|_{R_m}). 
$$
Theorem~\ref{thm:CM} thus shows that the limit 
$$
\dd_p(\chi,\chi'):=\lim_m m^{-1}\dd_p(\chi|_{R_m},\chi'|_{R_m})
$$
exists in $[0,+\infty)$, and coincides with the $L^p$-norm of the identity with respect to the spectral measure $\sigma(\chi,\chi')$, \ie 
\begin{equation}\label{equ:dpLp}
\dd_p(\chi,\chi')^p=\int_\R|\la|^p\sigma(\chi,\chi').
\end{equation}
It is clear that $(\dd_p)_{1\le p<\infty}$ is a non-decreasing family of pseudo-metrics on $\cN_\R$. For $p=1$,  \eqref{equ:d1vol} further yields
 \begin{equation}\label{equ:d1volgr}
  \dd_1(\chi,\chi')
  =\vol(\chi)+\vol(\chi')-2\vol(\chi\wedge\chi').
\end{equation}
For any $p\in[1,\infty]$, we have $\dd_1\le\dd_p\le\dd_\infty$.    One also easily checks (using for instance~\eqref{equ:dpLp} and Proposition~\ref{P401})
\begin{equation}\label{equ:equivdp} 
\dd_p(t\n,t\n')=t\dd_p(\n,\n'),\quad\dd_p(\n+c,\n'+c)=\dd_p(\n,\n'),\quad\dd_p(\n,\n+c)=|c|
\end{equation}
for all $\n,\n'\in\cN_\R$, $t\in\R_{>0}$, $c\in\R$.

The pseudo-metric $\dd_1$ defines a (non-Hausdorff) topology on $\cN_\R$, which is strictly coarser than the $\dd_p$-topology for any $p>1$. However, \eqref{equ:dpcomp} remains valid on $\cN_\R$, and shows that the $\dd_p$-topologies with $p<\infty$ all agree on $\dd_\infty$-bounded subsets of $\cN_\R$. In particular, they share the same pairs of non-separated points, which gives rise to: 

\begin{defi}\label{defi:asympequiv} We say that two norms $\chi,\chi'\in\cN_\R$ are \emph{asymptotically equivalent}, and write $\chi\sim\chi'$, if the following equivalent conditions hold:
\begin{itemize}
\item[(i)] $\dd_1(\chi,\chi')=0$; 
\item[(ii)] $\dd_p(\chi,\chi')=0$ for all $p\in[1,\infty)$; 
\item[(iii)] $\sigma(\chi,\chi')=\d_0$. 
\end{itemize}
\end{defi}
The equivalence between (i)---(iii) follows from~\eqref{equ:dpLp}. Since $\dd_1\le\dd_\infty$, we trivially have
$$
\dd_\infty(\chi,\chi')=0\Longrightarrow\chi\sim\chi'. 
$$
The converse fails in general  ---and thus so does the analogue of~\eqref{equ:dpLp} for the pseudo-metric $\dd_\infty$ (see, however, Corollary~\ref{cor:FSqft2} below for the case of continuous norms). 

\begin{exam}\label{exam:subvar2} Pick any subvariety $Z\varsubsetneq X$, and consider the norm $\n=\n_Z\in\cN_\Z$ as in Example~\ref{exam:subvar}. Using~\eqref{equ:sigmafilt}, it is easy to see that 
$\sigma(\n|_{R_m})=\e_m\d_0+(1-\e_m)\d_m$ with 
$$
\e_m:=\dim\Hnot(Z,mL)/\dim\Hnot(X,mL)=O(1/m).
$$
Thus $\sigma(\n)=\lim_m (1/m)_\star\sigma(\n|_{R_m})=\d_1$, and hence $\n\sim\n_\triv+1$ (see Proposition~\ref{P401}~(ii)). On the other hand, since $\n\ne\n_\triv+1$ are both homogeneous, we have $\dd_\infty(\n,\n_\triv+1)>0$ (see Proposition~\ref{prop:homcomp}). 
 \end{exam}

 By~\eqref{equ:d1volgr}, we have:
  
\begin{lem}\label{lem:inequiv} If $\n,\n'\in\cN_\R$ satisfy $\n\ge\n'$, then 
$\n\sim\n'\Longleftrightarrow\vol(\n)=\vol(\n')$. 
\end{lem}
   As we next show, spectral measures are continuous with respect to the $\dd_1$-topology. 

\begin{thm}\label{thm:specmeas} Consider nets $(\n_i)$, $(\n'_i)$ in $\cN_\R$, converging respectively to $\n,\n'\in\cN_\R$ in the $\dd_1$-topology. Then $\sigma(\n_i,\n'_i)\to\sigma(\n,\n')$ weakly. 
\end{thm}
\begin{proof} Set $\sigma_i:=\sigma(\n_i,\n'_i)$ and $\sigma:=\sigma(\n,\n')$.
  As in the proof of Theorem~\ref{thm:CM}, it suffices to prove that 
  $$
  \int\min\{\la,c\}\,\sigma_i(d\la)=\vol(\chi_i\wedge(\chi'_i+c))-\vol(\chi'_i)
  $$ 
  converges to
  $$
  \int\min\{\la,c\}\,\sigma(d\la)=\vol(\chi\wedge(\chi'+c))-\vol(\chi')
  $$ 
  for all $c\in\R$, see~\eqref{e423}. This follows immediately from the Lipschitz property of the volume and min operators, see~\eqref{equ:vollip} and Lemma~\ref{lem:d1min}.
\end{proof}

\begin{cor}\label{cor:quantequiv} For any $\n,\n'\in\cN_\R$, the quantities 
$$
\sigma(\n,\n'),\quad\la_{\max}(\n)\quad\text{and}\quad\vol(\n)
$$
only depend on the asymptotic equivalence classes of $\n,\n'$. Further, 
$$
\n_i\sim\n'_i,\,i=1,2\Longrightarrow\n_1\wedge\n_2\sim\n'_1\wedge\n'_2.
$$
\end{cor}
\begin{proof} The first claim follows directly from Theorem~\ref{thm:specmeas}. It implies the second one, as $\la_{\max}(\n)$ can be reconstructed from $\sigma(\n)=\sigma(\n,\n_\triv)$, by Theorem~\ref{thm:volandsup}. Finally, the Lipschitz properties of the volume and the min operator (see~\eqref{equ:vollip} and Lemma~\ref{lem:d1min}) carry over to $\cN_\R$, which takes care of the last two claims.  
\end{proof}

Following~\cite{Fujplt,Zha22}, we finally show: 

\begin{lem}\label{lem:dpcompare} Suppose $\n\in\cN_\R$ satisfies $\n\ge\n_\triv$. For any $p\in [1,\infty)$, we then have:  
\begin{itemize}
\item[(i)] $\dd_p(\n,\n_\triv)^p=p\int_0^{+\infty}\la^{p-1}\vol(\n\ge\la)\,d\la=p\int_0^{\la_{\max}(\n)}\la^{p-1}\vol(\n\ge\la)\,d\la$; 
\item[(ii)] $\la_{\max}(\n)=\dd_\infty(\n,\n_\triv)\le C_{n,p}\dd_p(\n,\n_\triv)$ with $C_{n,p}:={n+p\choose n}^{1/p}$.
\end{itemize}
\end{lem}
Note that $C_{n,p}\to 1$ as $p\to\infty$, and hence $\dd_p(\n,\n_\triv)\to\dd_\infty(\n,\n_\triv)$. 
\begin{proof} As in the proof of Theorem~\ref{thm:volandsup}, set $b:=\la_{\max}(\n)$, $f(\la):=\vol(\n\ge\la)$, and pick a smooth function $g\in C^\infty_c(\R)$ such that $g(\la)=\la^p$ for $\la\in [0,b]$. Then $f(\la)=1$ for $\la\le 0$, $f(\la)=0$ for $\la>b$, and hence 
$$
\dd_p(\n,\n_\triv)^p=\int_0^b\la^p\,d\sigma(\n)=-\int_\R g(\la) f'(\la)\,d\la=\int_\R g'(\la) f(\la)\,d\la
$$
$$
=\int_{-\infty}^0 g'(\la) d\la+\int_0^{+\infty} p\la^{p-1} f(\la)\,d\la,
$$
where $\int_{-\infty}^0 g'(\la)\,d\la=g(0)=0$. This proves (i). 

 The first equality in (ii) is~\eqref{equ:dtriv}. To prove the inequality, we argue as in~\cite[\S5]{Zha22}. By Theorem~\ref{thm:volandsup}, $f(\la)^{1/n}$ is concave on $(-\infty,b]$, and hence $f(\la)^{1/n}\ge 1-\la/b$ for $\la\in [0,b]$. Using (i), this yields 
$$
\dd_p(\n,\n_\triv)^p=p\int_0^{\la_{\max}} \la^{p-1} f(\la)\,d\la
$$
$$
\ge p\int_0^{b}\la^{p-1}\left(1-\frac{\la}{b}\right)^n\,d\la=b^p \,p\int_0^1 t^{p-1}(1-t)^n\,dt=\frac{p!n!}{(n+p)!} b^p,
$$
and the result follows. 
\end{proof}

%
\subsection{The space of norms modulo translation}
For any $p\in[1,\infty]$, the additive action of $\R$ on $\cN_\R$ preserves the pseudo-metric $\dd_p$, which thus induces a quotient pseudo-metric $\quotd_p$ on the space of norms modulo translation $\cN_\R/\R$, such that
$$
\quotd_p(\n,\n')=\inf_{c\in\R}\dd_p(\n,\n'+c)
$$
for $\n,\n'\in\cN_\R$. This supremum is actually achieved: 
\begin{lem}\label{lem:distach} For any $\n,\n'\in\cN_\R$, there exists $c\in\R$ such that $\quotd_p(\n,\n')=\dd_p(\n,\n'+c)$    and $|c|\le 2\dd_p(\n,\n')$.   
\end{lem}
In particular, $\quotd_p(\n,\n')=0$ iff $\n,\n'$ are \emph{asymptotically equivalent modulo translation}, in the sense that $\n\sim\n'+c$ for some $c\in\R$ (which is then uniquely determined by $c=\vol(\n)-\vol(\n')$). When $\n'=\n_\triv$ we say that $\n$ is \emph{asymptotically constant}. 

\begin{proof}    By~\eqref{equ:equivdp} , for all $c\in\R$ we have 
$$
|c|=\dd_p(\n',\n'+c)\le\dd_p(\n',\n)+\dd_p(\n,\n'+c). 
$$
Thus $\dd_p(\n,\n'+c)\le\dd_p(\n,\n')\Longrightarrow |c|\le 2\dd_p(\n,\n')$, and hence 
$$
\quotd_p(\n,\n')=\inf\{\dd_p(\n,\n'+c)\mid c\in\R,\,|c|\le 2\dd_p(\n,\n')\},
$$
which is achieved by compactness. 
  
\end{proof}

\begin{defi}\label{defi:Lpnorm}
 For each $p\in[1,\infty)$, we define the \emph{$L^p$-norm} of $\chi\in\cN_\R$ as 
$$
\|\n\|_p:=\dd_p(\n,\n_\triv+\vol(\n)).
$$ 
\end{defi}
This definition extends the notion in~\cite{Dontoric} of the $L^p$-norm of a test configuration, see~\cite[Remark~6.10]{BHJ1}. Indeed, \eqref{equ:dpLp} yields
$$
\|\chi\|^p_p=\int|\la-\overline\la|^p\,d\sigma(\chi), 
$$
the $p$-th central moment of the spectral measure $\sigma(\n)$, where $\overline\la=\int\la\,d\sigma(\chi)=\vol(\chi)$ is its barycenter. Note also that
$$
\|\n+c\|_p=\|\n\|_p,\quad\|t\n\|_p=t\|\n\|_p
$$
for $c\in\R$, $t\in\R_{>0}$.

\begin{prop}\label{prop:asympcst} Given a norm $\chi\in\cN_\R$, the following are equivalent:
\begin{itemize}
\item[(i)] $\chi$ is asymptotically constant;
\item[(ii)] $\|\chi\|_p=0$ for some $p\in[1,+\infty)$;
\item[(iii)] $\|\chi\|_p=0$ for all $p\in[1,+\infty)$.
\end{itemize}
\end{prop}
\begin{proof} If $\chi\sim\chi_\triv+c$ with $c\in\R$, then $c=\vol(\chi)$, by~\eqref{equ:vollip}. The rest is straightforward.  
\end{proof}

%
%
\subsection{Convergence of the canonical approximants}\label{sec:Fujita}
The next result strengthens the approximation result proved in~\cite[Theorem 1.14]{BC} (see also~\cite[Th\'eor\`eme 3.15]{Bou14} and~\cite{Cod19}). A version valid for arbitrary non-Archimedean fields is given in~\cite[Theorem~4.5.4]{Reb21} (which appeared after the first version of the current paper). 

Recall that the canonical approximants $\n_d\in\cT_\R$ of a norm $\chi\in\cN_\R$, which are defined for $d$ sufficiently divisible, satisfy $\n_d\le\chi$ and form an increasing net with respect to divisibility. 

\begin{thm}\label{thm:Fuj} For any $\chi\in\cN_\R$ and $p\in[1,\infty)$, we have $\dd_p(\n_d,\n)\to 0$.
\end{thm} 
   Recall that the result holds for $p=\infty$ iff $\n$ is continuous (see Theorem~\ref{thm:cont}). 

\begin{cor}\label{cor:tcdense}
For any $p\in [1,\infty)$, the set $\cT_\Z$ of test configurations is dense in $\cN_\R$ in the $\dd_p$-topology. 
\end{cor}  
\begin{proof} By Theorem~\ref{thm:cont}, $\cT_\Z$ is dense in $\cT_\R$ for $\dd_\infty$-topology, and hence also for the $\dd_p$-topology, since $\dd_p\le\dd_\infty$. It therefore suffices to show that $\cT_\R$ is $\dd_p$-dense in $\cN_\R$, which follows from Theorem~\ref{thm:Fuj} since the canonical approximants of any norm lie in $\cT_\R$.
\end{proof}

By $\dd_1$-continuity of spectral measures (see Theorem~\ref{thm:specmeas}) we also get: 
\begin{cor}\label{cor:Fuj} For all $\n,\n'\in\cN_\R$ we have $\lim_d\sigma(\n_d,\n'_d)=\sigma(\n,\n')$. 
\end{cor}

\begin{proof}[Proof of Theorem~\ref{thm:Fuj}] Since $\n_d\le\n$ is an increasing net,  $\dd_\infty(\n_d,\n)$ is decreasing, and hence uniformly bounded. In view of~\eqref{equ:dpcomp}, it is thus enough to show the result for $p=1$, \ie
$$
\dd_1(\n_d,\n)=\vol(\n)-\vol(\n_d)
$$ 
tends to $0$. Since $(\n_d)$ is $\dd_\infty$-bounded, we can find $a<b$ such that $\sigma(\n)$ and each $\sigma(\n_d)$ is has support in $[a,b]$, and hence 
$$
\vol(\n)-\vol(\n_d)=\int_a^b\vol(\n\ge\la)\,d\la-\int_a^b\vol(\n_d\ge\la)\,d\la,
$$
see Theorem~\ref{thm:volandsup}. By dominated convergence, it will thus suffice to show $\lim_d\vol(\n_d\ge\la)=\vol(\n\ge\la)$ for each $\la\in\R$ fixed.  

To see this, we may assume, after replacing $L$ by a multiple, that $R=R(X,L)$ is generated in degree $1$ and that $\chi$ is defined on $R$. By definition (see Theorem~\ref{thm:volandsup}~(i)), we have 
$$
\vol(\n\ge\la)=\frac{\vol(S)}{\vol(R)}=V^{-1}\vol(S),
$$
where $S\subset R$ is the graded subalgebra with graded pieces $S_m:=\Filt^{m\la} R_m$ and 
$\vol(S)=\lim_{m\to\infty}\frac{n!}{m^n}\dim S_m$. Similarly, $\vol(\n_d\ge\la)=(d^n V)^{-1}\vol(T(d))$ where $T(d)\subset S^{(d)}$ is generated in degree $1$ by $T(d)_1=S^{(d)}_1=S_d$. The desired convergence now follows from Lemma~\ref{L412} below.
\end{proof}

\begin{lem}\label{L412}
  Let $S\subset R$ be a graded subalgebra, and suppose we are given, for each $d$  divisible enough, a graded subalgebra
  $T(d)\subset S^{(d)}$ such that $T(d)_1=S^{(d)}_1=S_d$. 
  Then $\lim_d d^{-n}\vol(T(d))=\vol(S)$.
\end{lem}
\begin{proof}
  We use Okounkov bodies, following~\cite{Bou14}.
  Set $K:=k(X)$, and pick a valuation $\nu\colon K^\times\to\Z^n$ of maximal rational rank, equal to $n$ (\eg associated to a flag of subvarieties as in~\cite{LM09}).
  Set $\Gamma_m:=\nu(S_m\smallsetminus\{0\})$ and 
  $\Gamma(d)_m:=\nu(T(d)_m\smallsetminus\{0\})$ for $m\ge1$. 
  Let $\D(S)$ and $\D(T(d))$ be the 
  closed convex hull inside $\R^n$ of $\bigcup_mm^{-1}\Gamma_m$
  and of $\bigcup_mm^{-1}\Gamma(d)_m$, respectively.
  Then $\vol(S)=n!\vol(\D(S))$ and 
  $\vol(T(d))=n!\vol(\D(T(d)))$,
  so it suffices to prove that 
  $\lim_d\vol(d^{-1}\D(T(d)))=\vol(\D(S))$.

  Since $T(d)_m\subset S_{dm}$, we get 
  $\Gamma(d)_m\subset\Gamma_{dm}$ for all $d,m$, and hence
  $d^{-1}\D(T(d))\subset\D(S)$ for all $d$.
  If $\vol(\D(S))=0$, we are done, so we may assume $\D(S)$
  has nonempty interior. Pick compact subsets $A$ and $B$ of $\R^n$
  with $A\Subset B\Subset\D(S)$. 
  It suffices to prove that $d^{-1}\D(T(d))\supset A$ for $d$ sufficiently divisible.
  Now $d^{-1}\Z^n\cap B=d^{-1}\Gamma_d\cap B$,
  see~\cite[Lemme~1.13]{Bou14}.
  If $\D_d$ is the convex hull of $d^{-1}\Gamma_d$, it follows that $\D_d\supset A$.
  But $T(d)_1=S_d$, so $\Gamma(d)_1=\Gamma_d$, and hence $d^{-1}\D(T(d))\supset\D_d\supset A$,  which completes the proof.
\end{proof}
%
%
%
%

\section{Non-Archimedean pluripotential theory}\label{sec:npp}
In this section we summarize results from~\cite{trivval} that are relevant to our later purposes. 
%
%
\subsection{$L$-psh functions}\label{sec:Lpsh}
An \emph{$L$-psh} function $\f\colon \Xan\to[-\infty,+\infty)$ is defined as the pointwise limit of any decreasing net in $\cH_\Q$ (or $\cH_\R$), excluding $\f\equiv-\infty$.
We denote by $\PSH=\PSH(L)$ the set of all $L$-psh functions. If $\f\in\PSH$, then $\f+c\in\PSH$ for all $c\in\R$. 
If $\f,\p\in\PSH$, then $\max\{\f,\p\}\in\PSH$. If $(\f_j)_j$ is a decreasing net in $\PSH$, and $\f$ is the pointwise limit of $(\f_j)$, then $\f\in\PSH$, or $\f\equiv-\infty$. We can thus describe $\PSH$ as the smallest class of functions which is invariant under max, translation by a constant, decreasing limits, and contains all functions of the form $m^{-1}\log|s|$ with $m$ sufficiently divisible and $s\in R_m\smallsetminus\{0\}$

By Dini's Lemma, the set 
\begin{equation*}
  \CPSH:=\PSH\cap\Cz
\end{equation*}
of (bounded) continuous $L$-psh functions is the closure of $\cH_\Q$ (or $\cH_\R$) in $\Cz$ (in line with the definition of a semipositive (continuous) metric in~\cite{Zha95,GublerLocal,CM18}). 

The set $\PSH$ is stable under convex combinations, and under the action $(t,\f)\mapsto t\cdot\f$ of $\R_{>0}$ on functions, see~\eqref{equ:actionfunc}. If $v,v'\in\Xan$ and $v\le v'$, then $\f(v)\ge\f(v')$ for all $\f\in\PSH$. In particular,
$$
\sup\f:=\sup_{\Xan}\f=\f(v_\triv)
$$ 
for all $\f\in\PSH$.

A subset $\Sigma\subset \Xan$ is \emph{pluripolar} if $\Sigma\subset\{\f=-\infty\}$ for some $L$-psh function $\f$. This condition is independent of the choice of ample $\Q$-line bundle $L$, and $\Sigma$ is pluripolar iff 
\begin{equation}\label{equ:AT}
\tee(\Sigma):=\sup_{\f\in\PSH}(\sup\f-\sup_\Sigma\f)\in[0,+\infty] 
\end{equation}
is finite. If $\Sigma=\{v\}$ with $v\in X^\an$, then $\tee(\{v\})=\tee(v)$ as defined in~\eqref{equ:T}, and the set $\Xlin\subset X^\an$ of valuations of linear growth thus coincides with the set of non-pluripolar points $v\in \Xan$, \ie such that every $\f\in\PSH$ is finite-valued on $v$. 

Since every divisorial valuation has linear growth, $L$-psh functions are finite-valued on $\Xdiv$. The restriction map $\PSH\to\R^{\Xdiv}$ is further injective~\cite[Corollary~4.23]{trivval}, and we endow $\PSH$ with the induced topology of pointwise convergence on $\Xdiv$. This is in fact equivalent to pointwise convergence on $\Xlin$~\cite[Theorem~11.4]{trivval}. 

Note that since $\cH_\R(dL)=r\cH_\R(L)$ we have $\PSH(dL)=d\PSH(L)$ for any $d\ge \in\Q_{>0}$. To study $\PSH(L)$ we may therefore in practice assume that $L$ is an ample line bundle and that $R(X,L)$ is generated in degree 1.

We refer to~\cite[Example~4.13]{trivval} for a concrete description of $L$-psh functions on curves. See also Appendix~\ref{sec:toric} below for the toric case.
%
\subsection{Monge--Amp\`ere operator and energy on $\cH_\Q$}\label{sec:MAen}
  
The \emph{mixed Monge--Amp\`ere operator} on $\cH_\Z=\cH_\Q$ associates to 
any tuple $(\f_1,\dots,\f_n)\in\cH_\Q^n$ a Radon probability measure
$\MA(\f_1,\dots,\f_n)$, defined as follows. Pick an integrally closed, semiample test configurations $(\cX,\cL_i)$ for $(X,L_i)$ (with the same $\cX$) such that $\f_i=\f_{\cL_i}$, see Appendix~\ref{sec:tc}. Denoting by $\cX_0=\sum_ib_iE_i$ the irreducible decomposition of the central fiber, we then have 
\begin{equation*}
  \MA(\f_1,\dots,\f_n)=\sum_ib_i(\cL_1|_{E_i}\cdot\ldots\cdot\cL_n|_{E_i})\d_{v_i},
\end{equation*}
where $v_i\in X^\div$ is the divisorial valuation defined by $E_i$.

Following the strategy by Chambert-Loir in~\cite{CL06}, the mixed Monge--Amp\`ere operator admits a unique continuous extension to the space $\CPSH$ of  \emph{continuous} $L$-psh functions (with respect to uniform convergence), and this extension is in turn a special case of the with the general theory developed in~\cite{CLD}.
\begin{lem}\label{lem:suppMA}
  For any $\f_1,\dots,\f_n\in\cH_\R$, the support of $\MA(\f_1,\dots,\f_n)$ is a finite subset of $X^\lin$. 
\end{lem} 

\begin{proof} Set $\mu:=\MA(\f_1,\dots,\f_n)$ and $\Sigma:=\supp\mu$. The finiteness of $\Sigma$ is proved in~\cite[Example~8.11]{BE}, as a consequence 
of~\cite[Proposition~6.9.2]{CLD} and the invariance under ground field extension of the Chambert-Loir--Ducros construction. By~\cite[Proposition~7.21]{trivval}, we further have $\int\f\,\mu>-\infty$ for any $\f\in\PSH$. Since $\f$ is bounded above, this implies that $\f$ is finite at each $v\in\Sigma$; thus $v$ is nonpluripolar, and hence $v\in X^\lin$. 
\end{proof}

We will use notation such as $\MA(\f^{\langle j\rangle},\p^{\langle n-j\rangle})$, with $j$ copies of $\f$ and $n-j$ copies of $\p$, and write $\MA(\f)=\MA(\f^{\langle n\rangle})$. We then have $\MA(0)=\d_{v_\triv}$.

The~\emph{Monge--Amp\`ere energy} $\en\colon\CPSH\to\R$ is the primitive of the Monge--Amp\`ere operator in the sense that 
\begin{equation}\label{equ:Ederiv}
\frac{d}{dt}\bigg|_{t=0}\en((1-t)\f+t\p)=\int(\p-\f)\MA(\f)
\end{equation}
for $\f,\p\in\CPSH$, normalized by $\en(0)=0$. As such, $\en$ is monotone increasing, \ie $\f\ge\p\Longrightarrow \en(\f)\ge \en(\p)$. Integration along line segments yields
\begin{equation}\label{equ:cocycle} 
\en(\f)-\en(\p)=\frac{1}{n+1}\sum_{j=0}^n\int_{\Xan}(\f-\p)\,\MA(\f^{(j)},\p^{(n-j)}),
\end{equation}
for $\f,\p\in\CPSH$, and hence 
\begin{equation*}
\en(\f)=\frac{1}{n+1}\sum_{j=0}^n\int_{\Xan}\f\,\MA\left(\f^{\langle j\rangle},0^{\langle n-j\rangle}\right).
\end{equation*}

If $\f\in\cH_\Z=\cH_\Q$ is represented by a test configuration $(\cX,\cL)$, then
\begin{equation}\label{equ:testen}
\en(\f)=\frac{(\bar\cL^{n+1})}{(n+1)(L^n)},
\end{equation}
where $(\bar\cX,\bar\cL)\to\P^1$ is the canonical compactification of $(\cX,\cL)\to\A^1$.

The functional $\en$ is concave on $\CPSH$, which amounts to 
\begin{equation}\label{equ:enconc}
\en(\f)-\en(\p)\le\int(\f-\p)\MA(\p)
\end{equation}
for all $\f,\p\in\CPSH$, by~\eqref{equ:Ederiv}. Combined with~\eqref{equ:cocycle}, this implies
\begin{equation}\label{equ:Eapprox}
\f\ge\p\Longrightarrow\en(\f)-\en(\p)\approx\int(\f-\p)\MA(\p).
\end{equation}
In addition to $\en$, we introduce the translation invariant functional
$$
\ii(\f,\p):=\int(\f-\p)\left(\MA(\p)-\MA(\f)\right)\ge 0, 
$$
which satisfies the quasi-triangle inequality
\begin{equation}\label{equ:quasiI}
\ii(\f_1,\f_2)\lesssim\ii(\f_1,\f_3)+\ii(\f_3,\f_2). 
\end{equation}
We also set
$$
\ii(\f):=\ii(\f,0):=\sup\f-\int\f\MA(\f). 
$$
The Monge--Amp\`ere operator is homogeneous with respect to the action of $\R_{>0}$ on continuous $L$-psh functions $\f$, in the sense that 
$\MA(t\cdot\f)=t_\star\MA(\f)$ for all $t>0$. Similarly, we have 
$\en(t\cdot\f)=t\en(\f)$, $\ii(t\cdot\f)=t\ii(\f)$. 
%
%
\subsection{Functions and measures of finite energy}
The Monge--Amp\`ere energy admits a unique non-decreasing, usc extension $\en\colon\PSH\to\R\cup\{-\infty\}$, given for $\f\in\PSH$ by 
\begin{equation}\label{equ:endef}
\en(\f):=\inf\left\{\en(\p)\mid \f\le\p\in\CPSH\right\}.
\end{equation}
We denote by 
$$
\cE^1:=\left\{\f\in\PSH\mid \en(\f)>-\infty\right\}
$$
the set of $L$-psh functions \emph{of finite energy}. In other words, functions in $\cE^1$ are decreasing limits of nets $\f_i\in\cH_\Q$ with energy $\en(\f_i)$ uniformly bounded below. 

The \emph{weak topology} of $\cE^1$ is its subspace topology from $\PSH$, and the \emph{strong topology} on $\cE^1$ is the coarsest refinement of the weak topology for which $\en$ becomes continuous. 

For a decreasing or increasing net $(\f_j)$ in $\cE^1$, strong and weak convergence coincide, \ie $\f_j\to\f$ strongly in $\cE^1$ iff $\f_j\to\f$ pointwise on $\Xdiv$, see Example~12.2 and Theorem~12.5 in~\cite{trivval}, respectively.

Denote by $\cM$ the space of Radon probability measures on $\Xan$, endowed with the weak topology. The main point in introducing the strong topology is that the mixed Monge--Amp\`ere operator $\MA$, a priori only defined as a map $(\cC^0\cap\PSH)^n\to\cM$, admits a (unique) extension $(\cE^1)^n\to\cM$ that is continuous in the strong topology
  on both sides.
  
Further, 
$$
(\f_0,\f_1,\dots,\f_n)\mapsto\int\f_0\,\MA(\f_1,\dots,\f_n)
$$ 
is finite-valued and (strongly) continuous on tuples in $\cE^1$. In particular, the functional $\ii$ from~\S\ref{sec:MAen} extend continuously to $\cE^1$, and it induces a quasi-metric on $\cE^1/\R$ that defines the strong topology. 

The \emph{energy} of a probability measure $\mu\in\cM$ on $\Xan$ is defined by 
\begin{equation}\label{equ:enmeas}
  \en^\vee(\mu):=\sup_{\f\in\cE^1}\left\{\en(\f)-\int\f\,d\mu\right\},
\end{equation}
where the supremum can be restricted to functions in $\cH_\Q$, by approximation. This defines a convex, lsc function $\en^\vee\colon\cM\to[0,+\infty]$. We denote by 
$$
\cM^1:=\left\{\mu\in\cM\mid\en^\vee(\mu)<+\infty\right\}
$$ 
the set of measures of finite energy. It comes with a strong topology, defined as the coarsest refinement of the weak topology of measures in which $\en^\vee$ is continuous. The topological space $\cM^1$ does not depend on $L$. 

By~\eqref{equ:enconc}, for any $\f\in\cE^1$, the measure $\mu=\MA(\f)$ has finite energy, and $\f$ achieves the supremum in~\eqref{equ:enmeas}, \ie
  
\begin{equation}\label{equ:enMA}
  \en^\vee(\MA(\f))=\en(\f)-\int\f\,\MA(\f). 
\end{equation}
  
Conversely, a measure $\mu\in\cM^1$ satisfies $\mu=\MA(\f)$ with $\f\in\cE^1$ iff $\f$ achieves the supremum in~\eqref{equ:enmeas}. By a main result of~\cite{trivval}, the Monge--Amp\`ere operator induces a topological embedding with dense image
$$
\MA\colon\cE^1/\R\hto\cM^1,
$$
with respect to the strong topology on both sides.
%
\subsection{Envelopes}\label{sec:env}
   Consider a bounded-above family $(\f_i)$ of $L$-psh functions, and set $\f:=\sup_i\f_i$. By definition, the usc regularization $\f^\star\colon X^\an\to\R\cup\{-\infty\}$ is the smallest usc function such that $\f^\star\ge\f$. 

\begin{lem}\label{lem:envdiv} The restriction of $\f^\star$ to $X^\div$ coincides with $\f$, and $\f^\star$ is the smallest usc function on $X^\an$ with this property. 
\end{lem}
\begin{proof} By~\cite[Theorem~5.6]{trivval}, points of $X^\div$ are non-negligible, which is a reformulation of the first assertion. Consider next a usc function $\p\colon X^\an\to\R\cup\{-\infty\}$ such that $\p=\f$ on $X^\div$. For each $i$, we then have $\f_i\le\p$ on $X^\div$, and hence on $X^\an$, by~\cite[Theorem~4.22]{trivval}. Taking the supremum over $i$ yields $\f\le\p$ on $X^\an$, and hence $\f^\star\le\p$, since $\p$ is usc. 
\end{proof}

We say that $(X,L)$ has the \emph{envelope property} if $\f^\star$ is $L$-psh for each bounded-above family of $L$-psh functions, using the above notation. It is proved in~\cite[Theorem~5.20]{trivval} that the envelope property holds if $X$ is smooth and $k$ has characteristic zero, or in any characteristic if $\dim X\le 2$~\cite{GJKM}. For later use, we record: 

\begin{lem}\label{lem:incrdiv} Let $(\f_i)$ be a bounded-above, increasing net in $\cE^1$. Set $\f:=\sup_i\f_i$, and assume that $\f^\star$ is $L$-psh (\eg $L$ has the envelope property). Then $\f^\star\in\cE^1$ and $\f_i\to\f^\star$ strongly in $\cE^1$.
\end{lem}

\begin{proof} The first point holds because $\f^\star\ge\f_i$, and hence $\en(\f^\star)\ge\en(\f_i)>-\infty$. As recalled above, we have $\f^\star=\f$ on $X^\div$. Thus $\f_i\to\f^\star$ pointwise on $X^\div$, \ie weakly in $\cE^1$, and hence strongly as well, since $(\f_i)$ is an increasing net. 
\end{proof}

Given a function $\f\colon\Xan\to\R\cup\{\pm\infty\}$ we define the \emph{psh envelope} pointwise as
\begin{equation*}
  \pp(\f):=\sup\{\p\in\PSH\mid \p\le\f\}.
\end{equation*}
Note that the \emph{Fubini--Study envelope} in~\eqref{equ:FSenv} can be written
\begin{equation*}
  \qq(\f)=\sup\{\p\in\cH_\Q\mid \p\le\f\}=\sup\{\p\in\CPSH\mid \p\le\f\}.
\end{equation*}  
in both cases the convention $\sup\emptyset=-\infty$ applies.
Clearly $\qq(\f)\le\pp(\f)\le\f$, and either $\inf\f=-\infty\equiv\qq(\f)$, or $\qq(\f)$ is (finite-valued and) lsc. In the latter case, \begin{equation}\label{equ:PQ}
\qq(\f)=\pp(\f_\star)
\end{equation}
where $\f_\star$ is the lsc regularization of $\f$ (see~\cite[Lemma~5.19]{trivval}). In particular, $\qq(\f)=\pp(\f)$ when $\f$ is continuous.

\smallskip
The functions $\pp(\f)$ and $\qq(\f)$ are not psh in general. For any $\f\in\Cz$, we have 
$$
\pp(\f)=\qq(\f)\Longleftrightarrow\pp(\f)\in\Cz\Longleftrightarrow\pp(\f)\in\PSH,
$$
and these properties hold if (and only if) $L$ has the envelope property. For the next result, see~\cite[Corollary 5.18]{trivval}.

\begin{lem}\label{lem:envusc} Assume that $(X,L)$ has the envelope property, and consider a usc function $\f\colon\Xan\to\R\cup\{-\infty\}$. Then: 
\begin{itemize}
\item[(i)] either $\pp(\f)\in\PSH$ or $\pp(\f)\equiv-\infty$; 
\item[(ii)] if $\f$ is the pointwise limit of a decreasing net of usc functions $\f_i\colon X^\an\to\R\cup\{-\infty\}$, then $\pp(\f_i)\searrow\pp(\f)$ pointwise on $X^\an$. 
\end{itemize}
\end{lem}

Denote by $\cE^\infty\subset\cE^1$ the space of bounded $L$-psh functions. A  function $\f\in\cE^\infty$ is \emph{regularizable from below} if there exists an increasing net $(\f_j)_j$ in $\CPSH$ that converges to $\f$ in $\PSH$ (\ie pointwise on $X^\div$). Such a net can then be chosen in $\cH_\Q$, and converges strongly to $\f$ in $\cE^1$. We write 
$$
\cE^\infty_\uparrow\subset\cE^\infty
$$ 
for the space of $L$-psh functions regularizable from below. If the envelope property holds, then a bounded function $\f\in\cE^\infty$ lies in $\cE^\infty_\uparrow$ iff its discontinuity locus $\{\f_\star<\f\}$ is pluripolar~\cite[Theorem~11.23]{trivval}. 

\begin{rmk}\label{rmk:pshex} Assuming the envelope property, the inclusion $\cE^\infty_\uparrow\subset\cE^\infty$ is strict if $n\ge 1$, whereas $\CPSH\subset\cE^\infty_\uparrow$ is strict as soon as $n\ge 2$. See Examples~13.23 and~13.25 in~\cite{trivval}. 
\end{rmk}

For any bounded function $\f\in\cL^\infty$, denote by $\qq^\star(\f):=\qq(\f)^\star$ the usc regularization of $\qq(\f)$.

\begin{lem}\label{lem:envreg} Assume that $(X,L)$ has the envelope property. Then:
\begin{itemize}
\item[(i)] $\qq^\star\colon\cL^\infty\to\cL^\infty$ is a projection operator onto $\cE^\infty_\uparrow$; 
\item[(ii)] for all $\f,\p\in\cE^\infty_\uparrow$ we have $\qq^\star(\f\wedge\p)=\pp(\f\wedge\p)$. 
\end{itemize}
 \end{lem}
\begin{proof} (i) follows from~\cite[Theorem~13.24]{trivval}. Pick $\f,\p\in\cE^\infty_\uparrow$. By~\eqref{equ:PQ} we have $\qq(\f\wedge\p)=\pp((\f\wedge\p)_\star)=\pp(\f_\star\wedge\p_\star)$. Since $\f,\p$ are regularizable from below, their discontinuity locus is pluripolar, \ie $\f_\star=\f$, $\p_\star=\p$, and hence $\f_\star\wedge\p_\star=\f\wedge\p$, outside a pluripolar set. By~\cite[Theorem~13.20]{trivval}, it follows that $\pp(\f_\star\wedge\p_\star)^\star=\pp(\f\wedge\p)^\star$, which coincides with $\pp(\f\wedge\p)$ since $\f\wedge\p$ is usc (see Lemma~\ref{lem:envusc}). This proves (ii). 
\end{proof}

%
\subsection{The extended energy}\label{sec:exten}
Recall from~\cite[\S 8]{trivval}\footnote{In \loccit, the extended energy was simply denoted by $\en(\f)$.} that the \emph{extended Monge--Amp\`ere energy} of an arbitrary function $\f\colon X^\an\to\R\cup\{\pm\infty\}$ is defined as 
\begin{equation}\label{equ:engen}
\ten(\f):=\sup\{\en(\p)\mid \p\in\PSH, \p\le\f\}\in\R\cup\{\pm\infty\}. 
\end{equation}
Note that $\ten(\f)=\ten(\pp(\f))$, since any $\p\in\PSH$ satisfies $\p\le\f\Leftrightarrow\p\le\pp(\f)$. If $\f\colon X^\an\to\R\cup\{+\infty\}$ is lsc (and hence bounded below), then $\pp(\f)=\qq(\f)$, see~\eqref{equ:PQ}, and hence 
\begin{equation}\label{equ:enextQ}
\ten(\f)=\ten(\pp(\f))=\ten(\qq(\f))=\sup\{\en(\p)\mid \p\in\cH_\R,\, \p\le\f\}. 
\end{equation}
A Dini-type argument (see~\cite[Proposition~8.3]{trivval}) further yields: 

\begin{lem}\label{lem:exten} The functional $\f\mapsto\ten(\f)$ is continuous along increasing nets of bounded-below lsc functions. 
\end{lem}

Following~\cite{trivadd}, we say that $(X,L)$ has the \emph{weak envelope property} if there exists a birational model $\pi\colon X'\to X$ and an ample $\Q$-line bundle $L'$ on $X'$ such that $\pi^\star L\le L'$ and $(X',L')$ has the envelope property. This is for instance the case whenever $\charac k=0$, or if $\dim X\le 2$. 

\begin{lem}\label{lem:envdiv2} Assume $(X,L)$ has the weak envelope property, and pick any bounded-above family $(\f_i)$ of $L$-psh functions. Set $\f:=\sup_i\f_i$. Then:
\begin{itemize}
\item[(i)] $\f^\star=\f$ on $X^\lin$;
\item[(ii)] if $\f$ is further bounded below, then $\ten(\f^\star)=\ten(\f)$. 
\end{itemize}

\end{lem}
\begin{proof} Point (i) means that each $v\in X^\lin$ is non-negligible. Use the previous notation. Since $\pi$ is birational, we have $\pi^{-1}(\{v\})=\{v'\}$ with $v'\in X'^\lin$. By~\cite[Lemma~5.4]{trivval}, it suffices to show that $v'$ is non-negligible, and this follows from~\cite[Theorem~13.17]{trivval}, which applies because $(X',L')$ has the envelope property. Finally, (ii) follows from~\cite[Theorem~B]{trivadd}, as the assumption guarantee that $\pp(\f)=\f$. 
\end{proof}

%
%
%
%

\section{Darvas metrics}\label{sec:Darvas}
In this section, we study the metrics on the spaces $\cH_\R$, $\cE^1$ and $\cM^1$ induced by the $\dd_1$-pseudometric of $\cN_\R$, and prove the main part of Theorem~B. 

%
%

\subsection{Volume vs.~energy}\label{sec:volen}

The next result will be a key tool in what follows. 

\begin{thm}\label{thm:volen}
For any $\chi\in\cN_\R$ we have  $\vol(\chi)=\ten(\FS(\chi))$. 
\end{thm}
Here $\FS(\chi)$ is bounded and lsc, but not $L$-psh in general, and $\ten(\FS(\n))$ is its extended energy (see~\S\ref{sec:exten}). 

\begin{proof}    Consider the round-down $\n':=\lfloor\chi\rfloor\in\cN_\Z$. Then $\dd_\infty(\n,\n')=0$ (see Example~\ref{exam:dinftyround}), and hence $\FS(\n)=\FS(\n')$, $\vol(\n)=\vol(\n')$. As a result, we may and do assume $\n\in\cN_\Z$.    By Theorem~\ref{thm:Fuj}, the canonical approximants $\n_d\in\cT_\Z$ satisfy $\vol(\n_d)\to\vol(\n)$. On the other hand, $\FS(\n_d)=\FS_d(\n)$ increases pointwise to $\FS(\n)$ (see~\eqref{equ:FScanapp}), and hence $\en(\FS(\n_d))=\ten(\FS(\n_d))\to\ten(\FS(\n))$, by Lemma~\ref{lem:exten}. 

We are thus reduced to the case $\n\in\cT_\Z$, which is a consequence of~\cite{BHJ1}. Indeed, $\chi$ corresponds to an ample test configuration $(\cX,\cL)$ for $(X,L)$ under the Rees correspondence (see Appendix~\ref{sec:tc}). By~\cite[Proposition 3.12]{BHJ1}, the spectral measure $\sigma(\chi)$ coincides with the Duistermaat--Heckman measure $\DH(\cX,\cL)$, and passing to the barycenters yields 
$$
\vol(\n)=\frac{(\bar\cL^{n+1})}{(n+1)(L^n)}, 
$$
by~\cite[Lemma~7.3]{BHJ1}. By~\eqref{equ:testen}, the right-hand side is also equal to $\en(\FS(\n))=\ten(\FS(\n))$, and the result follows. 
\end{proof}

\begin{cor}\label{cor:homequiv} Any norm $\n\in\cN_\R$ is asymptotically equivalent to its homogenization, \ie $\n\sim\n^\hom$. 
\end{cor} 
\begin{proof} Since $\n\le\n^\hom$, it suffices to show $\vol(\n)=\vol(\n^\hom)$ (see Lemma~\ref{lem:inequiv}). This follows from Theorem~\ref{thm:volen}, since $\FS(\n)=\FS(\n^\hom)$ by Proposition~\ref{prop:FSH}. 
\end{proof}

\begin{cor}\label{cor:volIN} For any $\f\in\cL^\infty$ we have $\vol(\IN(\f))=\ten(\qq(\f))=\ten(\f_\star)$. 
\end{cor}
\begin{proof} Since $\IN(\f_\star)=\IN(\f)$ (see~\eqref{equ:INlsc}), we may assume that $\f$ is lsc, and hence $\ten(\f)=\ten(\qq(\f))$, see~\eqref{equ:enextQ}. Now $\FS(\IN(\f))=\qq(\f)$, by Proposition~\ref{prop:FSIN}, and we conclude by Theorem~\ref{thm:volen}. 
\end{proof}
%
%
\subsection{The Darvas metric on $\cH_\R$}\label{sec:DarvasHR}
Recall from Corollary~\ref{cor:FSbij} that the operators 
$$
\FS\colon(\cT_\R,\dd_\infty)\twoheadrightarrow(\cH_\R,\dd_\infty),\quad\IN\colon(\cH_\R,\dd_\infty)\hto(\cT_\R,\dd_\infty)
$$
are isometries such that $\FS\circ\IN=\id$.  

For any $p\in[1,\infty]$, the pseudo-metric $\dd_p$ on $\cT_\R\subset\cN_\R$ satisfies $\dd_p\le\dd_\infty$; it is thus constant along the fibers of $\FS$, and hence descends to a pseudo-metric $\dd_p$ on $\cH_\R$, such that 
$$
\FS\colon(\cT_\R,\dd_p)\twoheadrightarrow(\cH_\R,\dd_p),\quad\IN\colon(\cH_\R,\dd_p)\hto(\cT_\R,\dd_p)
$$
are isometries.  Theorem~A asserts that $\dd_p$ is a metric on $\cH_\R$. Since $\dd_p\ge\dd_1$, this follows from the following more precise result. 

\begin{thm}\label{thm:DarvasH} The pseudo-metric $\dd_1$ on $\cH_\R$ is a metric, uniquely characterized by 
\begin{equation}\label{equ:d1en}
\f\ge\p\Longrightarrow\dd_1(\f,\p)=\en(\f)-\en(\p);
\end{equation} 
\begin{equation}\label{equ:d1quasi}
\dd_1(\f,\p)=\inf\{\dd_1(\f,\tau)+\dd_1(\tau,\p)\mid \tau\in\cH_\R, \tau\le\f\wedge\p\}, 
\end{equation}
for all $\f,\p\in\cH_\R$. 
\end{thm}

\begin{proof} We first prove that $\dd_1$ satisfies~\eqref{equ:d1en} and~\eqref{equ:d1quasi}, which will take care of uniqueness. Pick $\f,\p\in\cH_\R$, and set $\n:=\IN(\f)$, $\n':=\IN(\p)$. Then $\FS(\n)=\f$, $\FS(\n')=\p$, and Theorem~\ref{thm:volen} implies $\vol(\n)=\en(\f)$, $\vol(\n')=\en(\p)$. If $\f\ge\p$, then $\n\ge\n'$, and~\eqref{equ:d1volgr} yields
$$
\dd_1(\f,\p)=\dd_1(\n,\n')=\vol(\n)-\vol(\n')=\en(\f)-\en(\p), 
$$
which proves~\eqref{equ:d1en}. Next, pick $\e>0$. Applying Theorem~\ref{thm:Fuj} to $\n\wedge\n'$ yields $\n''\in\cT_\R$ such that $\n''\le\n\wedge\n'$ and $\dd_1(\chi'',\chi\wedge\chi')\le\e$. If we set $\tau:=\FS(\chi'')\in\cH_\R$, then $\tau\le\f,\p$, and
  \begin{multline*}
    \dd_1(\f,\tau)+\dd_1(\tau,\p)
    =\dd_1(\n,\n'')+\dd_1(\n'',\n')\\
    \le\dd_1(\n,\n\wedge\n')+\dd_1(\n\wedge\n',\n')+2\e\\
    =\dd_1(\n,\n')+2\e
    =\dd_1(\f,\p)+2\e,
  \end{multline*}
  where we have used~\eqref{equ:d1volgr}. This proves~\eqref{equ:d1quasi}. Assume now $\dd_1(\f,\p)=0$. By~\eqref{equ:d1en} and~\eqref{equ:d1quasi}, there exists a sequence $(\tau_i)$ in $\cH_\R$ such that $\tau_i\le\f,\p$ and $\en(\tau_i)\to \en(\f)$ and $\en(\tau_i)\to\en(\p)$. By~\cite[Proposition~12.6]{trivval}, it follows that $(\tau_i)$ converges to both $\f$ and $\p$ in $\cE^1$, and hence $\f=\p$, since the topology is separated. This proves, as desired, that $\dd_1$ is a metric on $\cH_\R$ (this conclusion alternatively follows from~\eqref{equ:biid1} below). 
\end{proof}

%
\subsection{The Darvas metric on $\cE^1$}
We next prove that the metric $\dd_1$ on $\cH_\R$ canonically extends to $\cE^1$, yielding an analogue in our context of the metric introduced by Darvas~\cite{Dar17} in the complex analytic setting. 

\begin{thm}\label{thm:Darvas} There exists a unique metric $\dd_1$ on $\cE^1$ that defines the strong topology and restricts to the previous metric $\dd_1$ on $\cH_\R$. Further: 
\begin{itemize}
\item[(i)] for all $\f,\p\in\CPSH$, we have
\begin{equation}\label{equ:Darvas}
\dd_1(\f,\p)=\en(\f)+\en(\p)-2\ten(\pp(\f\wedge\p)); 
\end{equation}
\item[(ii)] the metric space $(\cE^1,\dd_1)$ is complete iff the envelope property holds for $(X,L)$; 
\item[(iii)] if the envelope property holds, then~\eqref{equ:Darvas} remains valid for all $\f,\p\in\cE^1$, and $\pp(\f\wedge\p)\in\cE^1$. 
\end{itemize}
\end{thm}
Recall that the envelope property holds whenever $X$ is smooth and $\charac(k)=0$, and fails if $X$ is not unibranch. We refer to the metric $\dd_1$ on $\cE^1$ as the \emph{Darvas metric}. By~\cite{Reb22}, $(\cE^1,\dd_1)$ is a geodesic metric space. 

\medskip

Our strategy to extend $\dd_1$ to $\cE^1$ is to compare it to the functional
$$
\bii(\f,\p):=\ii(\f,\p)+|\sup\f-\sup\p|. 
$$
It was indeed proven in~\cite[\S 12.1]{trivval} that $\bii$ is a quasi-metric on $\cE^1$ that defines the strong topology, and further satisfies 
\begin{equation}\label{equ:biimax}
\bii(\f,\p)=\bii(\f,\f\vee\p)+\bii(\f\vee\p,\p)
\end{equation}
Set 
$$
\bii(\f):=\bii(\f,0),\quad\dd_1(\f):=\dd_1(\f,0). 
$$
By the quasi-triangle inequality, \eqref{equ:biimax} implies
\begin{equation}\label{equ:biimax2}
\bii(\f\vee\p)\lesssim\max\{\bii(\f),\bii(\p)\}. 
\end{equation}

\begin{lem}\label{lem:d1bii} The quasi-metrics $\dd_1$ and $\bii$ on $\cH_\R$ are H\"older comparable, in the sense that  
\begin{align}
  \dd_1(\f,\p)
  &\lesssim\bii(\f,\p)^{\a}\max\{\bii(\f),\bii(\p)\}^{1-\a}\label{equ:d1bii}\\
  \bii(\f,\p)
  &\lesssim\dd_1(\f,\p)^{\a}\max\{\dd_1(\f),\dd_1(\p)\}^{1-\a},\label{equ:biid1}
\end{align}
with $\a:=1/2^n$. In particular, $\dd_1(\f)\approx\bii(\f)$. 
\end{lem}
\begin{proof} Assume $\f\ge\p$. Then~\eqref{equ:d1en} and~\eqref{equ:Eapprox} show that 
  \begin{equation*}
    \dd_1(\f,\p)=\en(\f)-\en(\p)\approx\int(\f-\p)\MA(\p),
  \end{equation*}
and hence 
$$
\ii(\f,\p)=\int(\f-\p)(\MA(\p)-\MA(\f))\lesssim \dd_1(\f,\p).
$$
Now we can write
  \begin{equation}\label{equ:biidiff}
    \bii(\f,\p)
    =\int(\f-\p)\MA(\p)+\int(\f-\p)(\MA(0)-\MA(\f)).
  \end{equation}
  As a special case of~\cite[Lemma~7.30]{trivval} we also have the estimate
$$
\bigg|\int(\f-\p)(\MA(0)-\MA(\f))\bigg|\lesssim\ii(\f,\p)^{\a}\max\{\ii(\f),\ii(\p)\}^{1-\a}.
$$
Since $\ii\le\bii$, this yields on the one hand
\begin{align*}
\dd_1(\f,\p) & \lesssim\bii(\f,\p)+\ii(\f,\p)^{\a}\max\{\bii(\f),\bii(\p)\}^{1-\a}\\
& \lesssim\bii(\f,\p)^\a\max\{\bii(\f,\p),\bii(\f),\bii(\p)\}^{1-\a}\\
& \lesssim\bii(\f,\p)^\a\max\{\bii(\f),\bii(\p)\}^{1-\a}. 
\end{align*}
Since $\ii\lesssim\dd_1$, we get on the other hand
\begin{align*}
\bii(\f,\p) & \lesssim\dd_1(\f,\p)+\dd_1(\f,\p)^{\a}\max\{\dd_1(\f),\dd_1(\p)\}^{1-\a}\\
& \lesssim\dd_1(\f,\p)^\a\max\{\dd_1(\f),\dd_1(\p)\}^{1-\a}. 
\end{align*}
This proves~\eqref{equ:d1bii}, and~\eqref{equ:biid1} when $\f\ge\p$.

\smallskip
Now consider arbitrary $\f,\p\in\cH_\R$. To prove~\eqref{equ:d1bii}, set
$\sigma:=\f\vee\p\in\cH_\R$. From what precedes and~\eqref{equ:biimax2}, we have 
\begin{equation*}
  \dd_1(\f,\sigma)\lesssim\bii(\f,\sigma)^{\a}\max\{\bii(\f),\bii(\p\}^{1-\a}
  \quad\text{and}\quad
  \dd_1(\p,\sigma)\lesssim\bii(\p,\sigma)^{\a}\max\{\bii(\f),\bii(\p\}^{1-\a},
\end{equation*}  
which together with
the triangle inequality for $\dd_1$ yields~\eqref{equ:d1bii}.

The proof of~\eqref{equ:biid1} is similar. By~\eqref{equ:d1quasi} we can pick $\tau\in\cH_\R$ with $\tau\le\f,\p$ such that
$\max\{\dd_1(\tau,\f),\dd_1(\tau,\p)\}\lesssim\dd_1(\f,\p)$, and hence $\dd_1(\tau)\lesssim\max\{\dd_1(\f),\dd_1(\p\}$.
Since $\tau\le\f,\p$, the first step yields 
\begin{align*}
  \bii(\f,\tau)\lesssim\dd_1(\f,\tau)^{\a}\max\{\dd_1(\f),\dd_1(\tau)\}^{1-\a}\lesssim\dd_1(\f,\p)^\a\max\{\dd_1(\f),\dd_1(\p\}^{1-\a},\\
  \bii(\p,\tau)\lesssim\dd_1(\p,\tau)^{\a}\max\{\dd_1(\p),\dd_1(\tau)\}^{1-\a}\lesssim\dd_1(\f,\p)^\a\max\{\dd_1(\f),\dd_1(\p\}^{1-\a}, 
\end{align*}   
and the quasi-triangle inequality for $\bii$ yields~\eqref{equ:biid1}. 
\end{proof}

\begin{proof}[Proof of Theorem~\ref{thm:Darvas}] Since $\cH_\R$ is dense in $\cE^1$, uniqueness is clear. Given $\f,\p\in\cE^1$, pick sequences $(\f_i)$, $(\p_i)$ in $\cH_\R$ converging strongly to $\f$ and $\p$, respectively (for example, we can use decreasing sequences). Thus $\lim_i\bii(\f_i,\f)=\lim_i\bii(\p_i,\p)=0$. Using~\eqref{equ:d1bii} this implies that $(\dd_1(\f_i,\p_i))_i$ is a Cauchy sequence, so that
  $\dd_1(\f,\p):=\lim_j\dd_1(\f_j,\p_j)$ exists. It is easy to see that it does not depend on the choice of sequence $(\f_i)$ and $(\p_i)$, and that the extension is a pseudo-metric on $\cE^1$. Further, the estimates of Lemma~\ref{lem:d1bii} still hold for $\f,\p\in\cE^1$. In particular $\dd_1(\f,\p)=0$ iff $\bii(\f,\p)=0$ iff $\f=\p$, so $\dd_1$ is a metric on $\cE^1$.  These estimates also show that $\dd_1$ and $\bii$ share the same Cauchy sequences in $\cE^1$, so that $(\cE^1,\dd_1)$ is complete iff $(\cE^1,\bii)$ is complete. By~\cite[Theorem~12.8]{trivval}, this is also equivalent to the envelope property for $(X,L)$, which proves (ii). 
 
Next, pick $\f,\p\in\cH_\R$. By~\eqref{equ:d1en} and~\eqref{equ:d1quasi}, we have
\begin{align*}
\dd_1(\f,\p) & =\inf\left\{\dd_1(\f,\tau)+\dd_1(\tau,\p)\mid\tau\in\cH_\R,\,\tau\le\f\wedge\p\right\}\\
& =\en(\f)+\en(\p)-2\sup\{\en(\p)\mid\p\in\cH_\R,\,\p\le\f\wedge\p\}\\
& =\en(\f)+\en(\p)-2\ten(\pp(\f\wedge\p)),
\end{align*}
see~\eqref{equ:enextQ}. Since $\dd_1\le\dd_\infty$, all terms are continuous with respect to uniform convergence, and the identity therefore remains valid on $\CPSH$, which yields (i). 

Finally, assume that the envelope property holds. Pick $\f,\p\in\cE^1$, set $\rho:=\pp(\f\wedge\p)$, and choose decreasing nets $(\f_i)$, $(\p_i)$ in $\cH_\R$ converging to $\f,\p$. By Lemma~\ref{lem:envusc}, we either have $\rho\in\PSH$ or $\rho\equiv-\infty$, and $\pp(\f_i\wedge\p_i)$ decreases pointwise to $\rho$. Since~\eqref{equ:Darvas} holds for $\f_i,\p_i\in\cH_\R$, it also holds for $\f,\p$, by continuity of $\en$ along decreasing nets. In particular, $\en(\rho)$ is finite, and hence $\rho\in\cE^1$. This proves (iii). 
\end{proof}

We also note the following useful Lipschitz property:
\begin{prop}\label{prop:enlipd1} For all $\f,\p\in\cE^1$ we have 
$|\en(\f)-\en(\p)|\le\dd_1(\f,\p)$. 
\end{prop}
\begin{proof} By continuity, we may assume $\f,\p\in\cH_\R$. Pick $\e>0$ and $\tau\in\cH_\R$ such that $\tau\le\f\wedge\p$ and $\en(\f)+\en(\p)-2\en(\tau)\le\dd_1(\f,\p)+\e$, see~\eqref{equ:d1quasi}. Then 
$$
|\en(\f)-\en(\p)|\le|\en(\f)-\en(\tau)|+|\en(\tau)-\en(\p)|=\en(\f)+\en(\p)-2\en(\tau)\le\dd_1(\f,\p)+\e,
$$
and the result follows. 
\end{proof}

As in~\cite[Theorem~3.7]{DDL18}, we next provide a comparison of the Darvas metric $\dd_1$ on $\cE^1$ with the functional $\ii_1\colon\cE^1\times\cE^1\to\R_{\ge 0}$ defined by 
$$
\ii_1(\f,\p):=\int|\f-\p|(\MA(\f)+\MA(\p)).
$$
This functional is obviously symmetric, and it separates points, as a consequence of the Domination Principle (see~\cite[Corollary~10.6]{trivval}). For all $\f,\p\in\cE^1$, we further have
\begin{equation}\label{equ:I1max}
\ii_1(\f,\p)=\ii_1(\f,\f\vee\p)+\ii_1(\f\vee\p,\p). 
\end{equation}
As with $\ii$ and $\bii$, this follows from the Locality Principle (see~\cite[Theorem~7.40, Proposition~7.45]{trivval}). 

\begin{thm}\label{thm:d1andI1} For all $\f,\p\in\cE^1$ we have 
$\dd_1(\f,\p)\approx\ii_1(\f,\p)$. 
\end{thm}
The proof relies on the following analogue of~\cite[Lemma~3.8]{DDL18}.
\begin{lem}\label{lem:d1midpt}
  If $\f,\p\in\cE^1$ and $\rho:=\tfrac 12(\f+\p)$, then $\dd_1(\f,\p)\approx\dd_1(\f,\rho)+\dd_1(\rho,\p)$.
\end{lem}
\begin{proof} By approximation, we may assume $\f,\p\in\cH_\R$. 
  Pick any $\tau\in\cH_\R$ with $\tau\le\f\wedge\p$.
  Then $\tau\le\f,\p,\rho$, and~\eqref{equ:d1en} yields
   \begin{align*}
    \dd_1(\f,\rho)+\dd_1(\rho,\p)
    &\le\dd_1(\f,\tau)+\dd_1(\p,\tau)+2\dd_1(\rho,\tau)\\
    &=(\en(\f)-\en(\tau))+(\en(\p)-\en(\tau))+2(\en(\rho)-\en(\tau))\\
    &\approx\int(\f-\tau)\MA(\tau)+\int(\p-\tau)\MA(\tau)+2\int(\rho-\tau)\MA(\tau)\\
    &=2\int(\f-\tau)\MA(\tau)+2\int(\p-\tau)\MA(\tau)\\
    &\approx\left(\en(\f)-\en(\tau)+\en(\p)-\en(\tau)\right)=\dd_1(\f,\tau)+\dd_1(\p,\tau).
  \end{align*}
  Here the first inequality is simply the triangle inequality for $\dd_1$, whereas the third and fifth lines follow from~\eqref{equ:Eapprox}. By~\eqref{equ:d1quasi}, the infimum over $\tau$ of the right-hand side equals $\dd_1(\f,\p)$ and we are done.
\end{proof}

\begin{proof}[Proof of Theorem~\ref{thm:d1andI1}] Since $\dd_1(\f,\p)$ and $\ii_1(\f,\p)$ are both continuous along decreasing nets, we may assume wlog $\f,\p\in\cH_\R$. 
  Let us start by proving $\dd_1(\f,\p)\lesssim\ii_1(\f,\p)$. By~\eqref{equ:I1max} and the triangle inequality for $\dd_1$, it suffices to consider the case $\f\ge\p$. But in this case,\eqref{equ:Eapprox} yields
    \begin{equation*}
    \dd_1(\f,\p)
    =\en(\f)-\en(\p)
    \approx\int(\f-\p)\MA(\p)
    \le\ii_1(\f,\p)
  \end{equation*}

  It remains to prove $\dd_1(\f,\p)\gtrsim\ii_1(\f,\p)$. Set $\rho:=\frac12(\f+\p)\in\cH_\R$, so that Lemma~\ref{lem:d1midpt} gives
  $\dd_1(\f,\p)\approx\dd_1(\f,\rho)+\dd_1(\rho,\p)$.
  Pick $\e>0$. By~\eqref{equ:d1quasi}, we can find $\sigma,\tau\in\cH_\R$ such that $\sigma\le\f\wedge\rho$, $\tau\le\rho\wedge\p$, and
  \begin{equation*}
    \dd_1(\f,\rho)\ge\dd_1(\f,\sigma)+\dd_1(\sigma,\rho)-\e
    \quad\text{and}\quad
    \dd_1(\rho,\p)\ge\dd_1(\rho,\tau)+\dd_1(\tau,\p)-\e,
  \end{equation*}
  and hence 
  $$
  \dd_1(\f,\p)\gtrsim\dd_1(\sigma,\rho)+\dd_1(\rho,\tau)-2\e. 
  $$
  As $\sigma\le \rho$, we have
  \begin{equation*}
    \dd_1(\sigma,\rho)
    =\en(\rho)-\en(\sigma)
    \ge\int(\rho-\sigma)\MA(\rho)
    \ge2^{-n}\int(\rho-\sigma)(\MA(\f)+\MA(\p)),
  \end{equation*}
  where the last inequality follows by expanding $\MA(\rho)=\MA(\frac12(\f+\p))$.
  Combining this with the analogous lower bound on $\dd_1(\rho,\tau)$ yields
  \begin{equation*}
    \dd_1(\f,\p)\gtrsim\int(2\rho-\sigma-\tau)(\MA(\f)+\MA(\p))-2\e.
  \end{equation*}
  We conclude by noting that $2\rho-\sigma-\tau\ge\frac12|\f-\p|$ and letting $\e\to0$.
\end{proof}

%
%
\subsection{The Darvas metric on $\cM^1$} 
By~\cite[Proposition~12.7]{trivval}, the Monge--Amp\`ere operator induces a topological embedding with dense image
$$
\MA\colon\cE^1/\R\hto\cM^1,
$$
where $\cE^1$ and $\cM^1$ are both equipped with the strong topology. In particular, the quotient topology of $\cE^1/\R$ is Hausdorff. Since the action of $\R$ on $\cE^1$ by translation preserves $\dd_1$, the topology of $\cE^1/\R$ is defined by the quotient metric 
\begin{equation}\label{equ:quotd1}
  \quotd_1(\f,\p)=\inf_{c\in\R}\dd_1(\f+c,\p).
\end{equation}
Note also that the isometric surjection $\FS\colon(\cT_\R,\dd_1)\twoheadrightarrow(\cH_\R,\dd_1)$, being equivariant with respect to the action of $\R$, induces an isometric surjection
\begin{equation}\label{equ:FSquot}
\FS\colon(\cT_\R/\R,\quotd_1)\twoheadrightarrow(\cH_\R/\R,\quotd_1),
\end{equation}
where $\quotd_1$ respectively denotes the restriction of the quotient metric on $\cN_\R/\R$ and $\cE^1/\R$. 

    As in the proof of Lemma~\ref{lem:distach}, we have: 

\begin{lem}\label{lem:quotdbd} For all $\f,\p\in\cE^1$, there exists $c\in\R$ such that 
$\quotd_1(\f,\p)=\dd_1(\f+c,\p)$ and $|c|\le 2\dd_1(\f,\p)\lesssim\max\{\bii(\f),\bii(\p)\}$. 
\end{lem}
This provides another reason why~\eqref{equ:quotd1} defines a metric on $\cE^1/\R$.

\begin{thm}\label{thm:d1M1} There exists a unique metric $\dd_1$ on $\cM^1$ that defines the strong topology and restricts to the quotient metric~\eqref{equ:quotd1} on $\cE^1/\R\hto\cM^1$. Furthermore, the metric space $(\cM^1,\dd_1)$ is complete.
\end{thm}
Note that completeness this time holds with or without the envelope property, in contrast with Theorem~\ref{thm:Darvas}. As with the latter, the proof is based on a comparison of $\quotd_1$ with the translation invariant functional $\ii\colon\cE^1\times\cE^1\to\R_{\ge 0}$.

\begin{lem}\label{lem:bd1ii} The quasi-metrics $\quotd_1$ and $\ii$ on $\cE^1/\R$ are H\"older comparable, \ie
\begin{align}
 \quotd_1(\f,\p)&\lesssim\ii(\f,\p)^{\a}\max\{\ii(\f),\ii(\p)\}^{1-\a},\label{equ:d1ii}\\
\ii(\f,\p)&\lesssim\quotd_1(\f,\p)^{\a}\max\{\quotd_1(\f),\quotd_1(\p)\}^{1-\a}\label{equ:iid1}
\end{align}
for all $\f,\p\in\cE^1$, with $\a:=1/2^n$. In particular, $\quotd_1(\f)\approx\ii(\f)$. 
\end{lem}

\begin{proof} By translation invariance of $\ii$ and $\quotd_1$, we may assume $\sup\f=\sup\p=0$. Then~\eqref{equ:d1ii} follows directly from~\eqref{equ:d1bii}, since $\quotd_1(\f,\p)\le\dd_1(\f,\p)$. By Lemma~\ref{lem:quotdbd}, we can find $c\in\R$ such that $\quotd_1(\f,\p)=\dd_1(\f+c,\p)$ and $|c|\lesssim\max\{\ii(\f),\ii(\p)\}$. By~\eqref{equ:biid1} we infer 
\begin{equation}\label{equ:dc}
\ii(\f,\p)\le\bii(\f+c,\p)\lesssim \quotd_1(\f,\p)^{\a}\max\{\ii(\f),\ii(\p)\}^{1-\a}. 
\end{equation}
In particular, $\ii(\f)\lesssim\quotd_1(\f)^\a\ii(\f)^{1-\a}$, hence $\ii(\f)\lesssim\quotd_1(\f)$, and~\eqref{equ:iid1} follows. 
\end{proof}

\begin{proof}[Proof of Theorem~\ref{thm:d1M1}] Since $\cE^1/\R\hto\cM^1$ has dense image, uniqueness is clear. By~\cite[Theorem 10.12]{trivval}, the strong topology of $\cM^1$ is defined by a certain quasi-metric $\ii^\vee$, that further satisfies $\ii^\vee(\MA(\f),\MA(\p))\approx\ii(\f,\p)$. Using the estimates of Lemma~\ref{lem:bd1ii} and arguing just as in the proof of Theorem~\ref{thm:Darvas}, we infer the existence of an extension of $\quotd_1$ to a pseudo-metric $\dd_1$ on $\cM^1$ such that 
\begin{align}
\dd_1(\mu,\d_{v_\triv})\approx & \ii^\vee(\mu,\d_{v_\triv})\approx\en^\vee(\mu)\label{equ:normmu} \\
\dd_1(\mu,\mu')\lesssim &\ii^\vee(\mu,\mu')^{\a}\max\{\en^\vee(\mu),\en^\vee(\mu')\}^{1-\a}\label{equ:quotvee}\\
\ii^\vee(\mu,\mu')\lesssim & \dd_1(\mu,\mu')^{\a}\max\{\en^\vee(\mu)),\en^\vee(\mu')\}^{1-\a}\label{equ:veequot}
\end{align}
for all $\mu,\nu\in\cM^1$ (recalling that $\MA(0)=\d_{v_\triv}$). This shows that $\dd_1$ separates points, and hence is a metric on $\cM^1$, which further shares the same convergent and Cauchy sequences as $\ii^\vee$. It thus defines the strong topology of $\cM^1$, and $(\cM^1,\dd_1)$ is complete, because $(\cM^1,\ii^\vee)$ is complete by~\cite[Theorem~10.14]{trivval}.  
\end{proof}

Combining the above estimates with a key estimate for Monge--Amp\`ere integrals from~\cite{trivval}, we get the following H\"older continuity property: 

\begin{thm}\label{thm:BBGZ} There exist $\a_1,\a_2,\a_3\in\R_{>0}$, only depending on $n$, such that $\sum_i \a_i=1$ and 
\begin{equation}\label{equ:BBGZabs}
\left|\int\left|\f-\f'\right|(\mu-\mu')\right|\lesssim\dd_1(\f,\f')^{\a_1}\dd_1(\mu,\mu')^{\a_2} M^{\a_3}
\end{equation}
for all $\f,\f'\in\cE^1$ and $\mu,\mu'\in\cM^1$, where $M:=\max\{\ii(\f),\ii(\f'),\en^\vee(\mu),\en^\vee(\mu')\}$.    Further, there exists $\a\in(0,1)$ only depending on $n$ such that 
\begin{equation}\label{equ:BBGZabs2}
\|\f-\f'\|_{L^1(\mu)}\lesssim\dd_1(\f,\f')^\a\max\{\ii(\f),\ii(\f'),\en^\vee(\mu)\}^{1-\a}.
\end{equation}
\end{thm}
\begin{proof} By~\cite[Theorem~10.3]{trivval}, we have 
$$
\left|\int(\f-\f')(\mu-\mu')\right|\lesssim\ii(\f,\f')^\a\ii^\vee(\mu,\mu')^{\frac 12} M^{\frac 12-\a}. 
$$
Injecting~\eqref{equ:iid1} and~\eqref{equ:veequot} yields
\begin{equation}\label{equ:BBGZ}
\left|\int(\f-\f')(\mu-\mu')\right|\lesssim\dd_1(\f,\f')^{\a_1}\dd_1(\mu,\mu')^{\a_2} M^{\a_3}
\end{equation}
with $\a_1,\a_2,\a_3$ as above. 
Next, write
$$
|\f-\f'|=2(\tau-\f')+(\f'-\f)
$$
with $\tau:=\f\vee\f'\in\cE^1$. On the one hand, we have $\ii(\tau)\lesssim M$. On the other hand, Theorem~\ref{thm:d1andI1} and~\eqref{equ:I1max} yield
$$
\dd_1(\tau,\f')\approx\ii_1(\tau,\f')\le\ii_1(\f,\f')\approx\dd_1(\f,\f'). 
$$
Applying~\eqref{equ:BBGZ} to $\tau,\f'$ and $\f',\f$ now yields~\eqref{equ:BBGZabs}.    Assume now $\mu'=\MA(\f)$. Then $\en^\vee(\mu)\approx\ii(\f)$, and hence 
$$
\left|\int\left|\f-\f'\right|(\mu-\MA(\f))\right|\lesssim\dd_1(\f,\f')^\a\max\{\ii(\f),\ii(\f'),\en^\vee(\mu)\}^{1-\a}.
$$
By Theorem~\ref{thm:d1andI1}, we have on the other hand $\int|\f-\f'|\MA(\f)\lesssim\dd_1(\f,\f')$, and summing up these estimates yields~\eqref{equ:BBGZabs2}.  
\end{proof}

%
%
%
%
  
\section{Divisorial and maximal norms}\label{sec:divmaxnorm}
The restriction of the pseudometric $\dd_1$ to the subspace $\cN^\hom_\R\subset\cN_\R$ of homogeneous norms is still not a metric unless $\dim X=0$, see Example~\ref{exam:subvar2}. Here we study further subspaces on which $\dd_1$ does induce a metric.

One such subspace consists of divisorial norms, defined by finitely many divisorial valuations. These play an important role in the notion of divisorial stability introduced and studied in~\cite{nakstab2}. We then show that, at least in characteristic zero, there is a canonical maximal subspace of $\cN^\hom_\R$ on which $\dd_1$ is a metric. In particular, we prove Theorem~D.
%
%
\subsection{General infimum norms}\label{sec:signorm}
The following construction generalizes the one in~\S\ref{sec:infnorm}.

\begin{defi}
  For any non-pluripolar set $\Sigma\subset X^\an$, and any bounded function $\f\colon\Sigma\to\R$, let $\IN_\Sigma(\f)\in\cN_\R^\hom$ denote the homogeneous norm defined for $s\in R_m$ by
  \begin{equation}\label{equ:INsig}
    \IN_\Sigma(\f)(s)=\inf_{v\in\Sigma}\{v(s)+m\f(v)\}. 
  \end{equation}
\end{defi}
Note that $\exp(-\IN_\Sigma(\f)(s))$ coincides with the more usual supnorm $\sup_\Sigma|s|e^{-m\f}$. The filtration corresponding to $\IN_\Sigma(\f)$ is given by 
\begin{equation*}
\Filt^\la R_m=\left\{s\in R_m\mid v(s)+m\f(v)\ge\la\ \text{for all $v\in\Sigma$}\right\},\quad\la\in\R.
\end{equation*}
The condition that $\Sigma$ is non-pluripolar, which is equivalent to $\tee(\Sigma)<\infty$ (see~\eqref{equ:AT}) and holds as soon as $\Sigma\cap X^\lin\ne\emptyset$, guarantees that $\IN_\Sigma(\f)$ is indeed a (linearly bounded) norm. More precisely: 

\begin{lem}\label{lem:Sigpp} For any subset $\Sigma\subset X^\an$ and any bounded function $\f\colon\Sigma\to\R$, \eqref{equ:INsig} defines a (linearly bounded) norm iff the closure $\overline{\Sigma}\subset X^\an$ is non-pluripolar. Further, we then have $\tee(\overline{\Sigma})=\la_{\max}(\IN_\Sigma(0))$. 
\end{lem}
\begin{proof} Since $\f$ is bounded, it is clear that $\IN_\Sigma(\f)$ is a norm iff 
$$
\tee'(\Sigma):=\sup\left\{m^{-1}\inf_{v\in\Sigma} v(s)\mid s\in R_m\setminus\{0\}\text{ with }m\text{ sufficiently divisible}\right\}
$$
is finite. By continuity of $v\mapsto v(s)$ for any section $s$, we have $\tee'(\Sigma)=\tee'(\overline{\Sigma})$, and we may thus further assume that $\Sigma$ is closed. It will then be enough to show that $\tee(\Sigma)=\tee'(\Sigma)$ (the case of a single point being~\cite[Lemma~4.46]{trivval}). Note that $\tee'(\Sigma)=\sup_\f(-\sup_\Sigma\f)$ where $\f$ runs over $L$-psh functions of the form $\f=m^{-1}\log|s|$ with $s\in R_m\setminus\{0\}$. Since $\sup_{X^\an}\f=\f(v_\triv)=0$, we infer $\tee(\Sigma)\ge\tee'(\Sigma)$. Conversely, pick $\f\in\PSH$. If $\f\in\cH_\R$, then writing $\f$ as in~\eqref{equ:FSalt} yields $\sup_{X^\an}\f\le\sup_\Sigma\f+\tee'(\Sigma)$. In the general case, write $\f$ as the limit of a decreasing net in $(\f_i)$ in $\cH_\R$. Since $\sup_{X^\an}\f_i=\f_i(v_\triv)$ converges to $\sup_{X^\an}\f=\f(v_\triv)$, it suffices to show $\sup_\Sigma\f_i\to\sup_{\Sigma}\f$. Since $X^\an$, and hence $\Sigma$, are compact, we can find $v_i\in\Sigma$ such that $\f_i(v_i)=\sup_\Sigma\f_i$, for each $i$. After passing to a subnet, we may further assume $v_i\to v\in\Sigma$. If $i\le j$ then $\f_i(v_j)\ge\f_j(v_j)=\sup_\Sigma\f_j$, and letting $j\to\infty$ yields $\f_i(v)\ge\lim_j\sup_\Sigma\f_j$. Since $\lim_i\f_i(v)=\f(v)$, we infer $\sup_\Sigma\f\ge\f(v)\ge\lim_j\sup_\Sigma\f_j$, and the result follows. 
\end{proof}

\begin{rmk} Except in the trivial case $\dim X=0$, we can always find a pluripolar subset $\Sigma\subset X^\an$ such that $\overline{\Sigma}$ is non-pluripolar.  Indeed, the trivial valuation $v_\triv$, which is non-pluripolar, lies in the closure of $X(k)\subset X^\an$. By~\cite{Poi}, $v_\triv$ thus lies in the closure of a countable subset $\Sigma\subset X(k)$, which is necessarily pluripolar (see~\cite[Lemma~4.37]{trivval}). 
\end{rmk}

For a fixed non-pluripolar subset $\Sigma\subset X^\an$, we write 
$$
\cN^\Sigma_\R\subset\cN_\R^\hom
$$
for the set of norms $\IN_\Sigma(\f)$, with $\f$ ranging over bounded functions on $\Sigma$.

\begin{exam}\label{exam:Sighom} If $\f\colon X^\an\to\R$ is bounded, then $\IN_{X^\an}(\f)=\IN(\f)$, and Theorem~\ref{thm:INFS} thus yields $\cN^{X^\an}_\R=\cN_\R^\hom$. 
\end{exam} 

A simple check shows that
$$
\IN_\Sigma(\f\wedge\f')=\IN_\Sigma(\f)\wedge\IN_\Sigma(\f'),\quad\IN_\Sigma(\f+c)=\IN_\Sigma(\f)+c,\quad\IN_{t\Sigma}(t\cdot\f)=t\IN_\Sigma(\f)
$$
and
\begin{equation}\label{equ:INsiglip}
\dd_\infty(\IN_\Sigma(\f),\IN_\Sigma(\f'))\le\sup_\Sigma|\f-\f'|
\end{equation}
for all bounded functions $\f,\f'\colon\Sigma\to\R$, $c\in\R$ and $t\in\R_{>0}$. Thus $\cN^\Sigma_\R$ is invariant under the translation action of $\R$ and under minima, and it is invariant under the scaling action of $\R_{>0}$ whenever $\Sigma$ is. 

\begin{prop}\label{prop:signorm}
Pick a non-pluripolar subset $\Sigma\subset X^\an$. Then: 
  \begin{itemize}
  \item[(i)]
 each $\n\in\cN_\R^\Sigma$ satisfies $\n=\IN_\Sigma(\f)$ with $\f:=\FS(\n)|_\Sigma$, and $\f$ is the smallest bounded function on $\Sigma$ with this property; 
   \item[(ii)] if $\Sigma\subset\Sigma'\subset X^\an$ then $\cN^{\Sigma}_\R\subset\cN^{\Sigma'}_\R$.
  \item[(iii)] if $\Sigma$ is further dense in $\Sigma'$, then $\IN_\Sigma(\f)=\IN_{\Sigma'}(\f)$ for each bounded, usc function $\f\colon\Sigma'\to\R$. 
 \end{itemize}
\end{prop}
\begin{proof} Pick a bounded function $\p\colon\Sigma\to\R$ such that $\n=\IN_\Sigma(\p)$. For any $v\in\Sigma$ and any $s\in R_m\setminus\{0\}$ we have $\n(s)\le m^{-1}v(s)+\p(v)$. On the one hand, this implies $\f(v)=\sup_s\{\n(s)-m^{-1}v(s)\}\le\p(v)$ for any $v\in\Sigma$, and hence $\IN_\Sigma(\f)\le\n$. On the other hand, for any $s\in R_m\setminus\{0\}$ and any $v\in\Sigma$, we have $m^{-1}v(s)+\f(v)\ge\n(s)$, so $\IN_\Sigma(\f)\ge\n$. This proves (i). 
  
To see (ii), pick $\n\in\cN_\R^{\Sigma}$, \ie $\n=\IN_\Sigma(\f)$ with $\f\colon\Sigma\to\R$ bounded. Pick $C>0$ such that $\n(s)\le mC$ for $s\in R_m\setminus\{0\}$. We claim that $\n$ coincides with $\n':=\IN_\Sigma(\f')\in\cN_\R^{\Sigma'}$, where $\f'\colon\Sigma'\to\R$ is the extension of $\f$ such that $\f'\equiv C$ on $\Sigma'\setminus\Sigma$. To see this, pick  $s\in R_m\setminus\{0\}$ For each $v'\in\Sigma'\setminus\Sigma$ we have 
$$
v'(s)+m\f(v')\ge mC\ge\n(s)=\inf_{v\in\Sigma}\{v(s)+m\f(v)\},
$$
which yields, as desired, $\n'(s)=\inf_{v\in\Sigma'}\{v(s)+m\f(v)\}=\inf_{v\in\Sigma}\{v(s)+m\f(v)\}=\n(s)$. 

Finally, the inequality $\IN_\Sigma(\f)\ge\IN_{\Sigma'}(\f)$ in~(iii) is trivial. Conversely, pick any $s\in R_m\setminus\{0\}$. Then $m^{-1}v(s)+\f(v)\ge\IN_{\Sigma'}(\f)(s)$ for all $v\in\Sigma$, and this inequality extends to $\Sigma'$ as $v\mapsto m^{-1}v(s)+\f(v)$ is usc on $\Sigma'$ and $\Sigma\subset\Sigma'$ is dense. Thus $\IN_\Sigma(\f)(s)\le\IN_{\Sigma'}(\f)(s)$, which proves (iii). 
\end{proof}

\begin{cor}\label{cor:decrgin1}
  Suppose $\Sigma\subset X^\an$ is non-pluripolar. If $(\n_i)$ is a decreasing net in $\cN^\Sigma_\R$ converging pointwise to $\n\in\cN^\hom_\R$ (see Remark~\ref{rmk:homnet}), then $\n\in\cN^\Sigma_\R$.
\end{cor}
\begin{proof}
  Set $\f_i:=\FS(\n_i)$. Then $\f_i$ is a decreasing net of functions on $X^\an$  bounded below by $\f:=\FS(\n)$. For any $i$, Proposition~\ref{prop:signorm}~(i) and Example~\ref{exam:Sighom} imply
 $$
  \n_i=\IN_\Sigma(\f_i)\ge\IN_\Sigma(\f)\ge\IN(\f)=\n. 
  $$
Taking the infimum over $i$ yields $\n=\IN_\Sigma(\f)\in\cN_\R^\Sigma$.
\end{proof}

Next we generalize the homogenization operator.
\begin{defi}\label{defi:sigproj}
  For any non-pluripolar subset $\Sigma\subset X^\an$, we define a projection operator $\pp_\Sigma\colon\cN_\R\to\cN_\R^\Sigma$ by setting $\pp_\Sigma(\n):=\IN_\Sigma(\FS(\n))$ for $\n\in\cN_\R$. 
\end{defi}
The map $\pp_\Sigma$ is indeed a projection, by Proposition~\ref{prop:signorm}~(i); it is further $1$-Lipschitz with respect to the $\dd_\infty$-pseudometric, by~\eqref{equ:FScontr} and~\eqref{equ:INsiglip}. 

The map $\pp_{X^\an}\colon\cN_\R\to\cN_\R^{X^\an}=\cN_\R^\hom$ coincides with homogenization (see Theorem~\ref{thm:INFS}). Further $\Sigma\subset\Sigma'\Longrightarrow\pp_{\Sigma'}(\n)\le\pp_\Sigma(\n)$. In particular, $\n\le\n^\hom\le\pp_\Sigma(\n)$, and $\pp_\Sigma(\n)$ can be characterized as the smallest norm in $\cN_\R^\Sigma$ such that $\n\le\pp_\Sigma(\n)$. A direct check further yields: 

\begin{lem}\label{lem:decrgin2} Let $(\Sigma_i)$ be an increasing net of non-pluripolar subsets of $X^\an$, and set $\Sigma:=\bigcup_i\Sigma_i$. Then $\pp_{\Sigma_i}$ decreases pointwise to $\pp_\Sigma$ on $\cN_\R$. 
\end{lem}
For later use, we also note: 

\begin{lem}\label{lem:FSenv} For any non-pluripolar subset $\Sigma\subset X^\an$ and $\n\in\cN_\R$, we have $\FS(\pp_\Sigma(\n))=\FS(\n)$ on $\Sigma$. 
\end{lem}
\begin{proof} Set $\f:=\FS(\n)$. Since $\n\le\pp_\Sigma(\n)$, we have $\f\le\FS(\pp_\Sigma(\n))$. Conversely, pick $v\in\Sigma$. For any $s\in R_m\setminus\{0\}$, we then have 
$\pp_\Sigma(\n)(s)=\IN_\Sigma(\f)\le v(s)+m\f(v)$, and hence $\FS(\pp_\Sigma(v))=\sup_s\frac1m(\pp_\Sigma(\n)(s)-v(s))\le\f(v)$, which proves the result. 
\end{proof}

%
%
\subsection{Divisorial norms and PL functions}\label{sec:divnorm}
In the next two subsections we consider two important cases of the construction above.  

\begin{defi}\label{defi:divnorm} We define the set $\cN_\R^\div\subset\cN_\R^\hom$ of \emph{divisorial norms} as the (increasing) union of $\cN^\Sigma_\R$ over all finite subsets $\Sigma\subset X^\div$. The set of \emph{rational divisorial norms} is 
$$
\cN^\div_\Q:=\cN^\div_\R\cap\cN_\Q.
$$
\end{defi}
That the union is increasing follows from Proposition~\ref{prop:signorm}~(ii). Also note that $\cN^\div_\R$ (resp.\ $\cN^\div_\Q$) is invariant under finite minima, under the scaling action by $\Q_{>0}$ and under the translation action by $\R$ (resp. $\Q$). 

Concretely, a norm $\n$ is divisorial iff it can be written as 
\begin{equation}\label{equ:divnorm}
\n=\max_i\{\n_{v_i}+c_i\}
\end{equation}
for a finite set of divisorial valuations $(v_i)$ and $c_i\in\R$, and $\n$ is rational iff the $c_i$ can be chosen in $\Q$. Indeed:

\begin{lem}\label{lem:ratdivnorm} For any finite subset $\Sigma\subset X^\div$, a norm $\n$ lies in $\cN^\Sigma_\Q$ iff it can be written $\n=\IN_\Sigma(\f)$ for some function $\f\colon\Sigma\to\Q$.
\end{lem}
\begin{proof}
  The `if' part is clear. Conversely, assume $\n\in\cN_\Q^\div$, and write $\n=\IN_\Sigma(\f)$ for some function $\f\colon\Sigma\to\R$ on a finite subset $\Sigma\subset X^\div$. Let $\Sigma':=\{v\in\Sigma\mid \f(v)\in\Q\}$ and let $\f'\colon\Sigma\to\Q$ be any function such that $\f'\ge\f$ with equality on $\Sigma'$. Then $\n':=\IN_\Sigma(\f')$ equals $\n$. Indeed, $\n'\ge\n$, and if $s\in R_m\setminus\{0\}$, then $\n(s)=\min_{v\in\Sigma}(m\f(v)-v(s))$. As $\n(s)\in\Q$ and $v(s)\in\Q$ for every $v\in\Sigma$, the minimum cannot be attained on $\Sigma'\setminus\Sigma$, so  $\n(s)=\min_{v\in\Sigma'}\{m\f(v)-v(s)\}=\n'(s)$.
\end{proof}

Recall from~\cite{trivval} that the space $\PL(X)\subset\Cz(X)$ of  \emph{piecewise linear} functions $\f\colon X^\an\to\R$ is defined as the $\Q$-vector space spanned by $\cH_\Q$. It is independent of the choice of $L$, stable under max and min, and is dense in $\Cz(X)$ with respect to uniform convergence. 

As we next show, rational divisorial norms arise precisely as infimum norms of PL functions:
\begin{thm}\label{thm:divPL} A norm $\n\in\cN_\R$ lies in $\cN^\div_\Q$ iff $\n=\IN(\f)$ with $\f\in\PL(X)$. 
\end{thm}

\begin{cor}\label{cor:divtc} Any rational homogeneous norm of finite type is divisorial, \ie $\cT^\hom_\Q\subset\cN^\div_\Q$. In particular, the homogenization of any test configuration $\n\in\cT_\Z$ is a rational divisorial norm. 
\end{cor}
In contrast, $\cT_\R^\hom$ is generally not contained in $\cN_\R^\div$, see Example~\ref{exam:ftnondiv}. We refer to Appendix~A (and especially Theorem~\ref{thm:normint}) for a more detailed discussion of the relation between test configurations and rational divisorial norms. 

\begin{cor}\label{cor:divppty} The envelope property holds for $(X,L)$ iff $\cN_\R^\div\subset\cN_\R^\cont$.
\end{cor}
See~\S\ref{sec:cont} for the space $\cN_\R^\cont$ of continuous norms. 

\begin{exam}\label{exam:nodalcont} If $X$ is a nodal curve, then the envelope property fails, and $\n_v\in\cN_\R^\div$ is indeed not a continuous norm if $v$ is a divisorial valuation with center at the node.
\end{exam}

\begin{proof}[Proof of Theorem~\ref{thm:divPL}] Assume first $\n=\IN(\f)$ with $\f\in\PL(X)$. By~\cite[Lemma~4.26]{trivval}, there exists a finite subset $\Sigma\subset X^\div$ such that $\sup_{X^\an}(\p-\f)=\max_\Sigma(\p-\f)$ for all $\p\in\PSH(L)$. In particular, for any $s\in R_m\setminus\{0\}$ we have 
$$
\sup_{X^\an}(m^{-1}\log|s|-\f)=\max_\Sigma(m^{-1}\log|s|-\f),
$$
\ie $\n(s)=\inf_{v\in X^\an}\{v(s)+m\f(v)\}=\min_{v\in\Sigma}\{v(s)+m\f(v)\}$. This proves $\IN(\f)=\IN_\Sigma(\f)$, which lies in $\cN^\div_\Q$ since PL functions take rational values on $X^\div$.

Conversely, assume $\n\in\cN^\div_\Q$, \ie $\n\in\cN^\Sigma_\Q$ for a finite subset $\Sigma\subset X^\div$. By~\cite[Lemma~2.12]{trivval}, $\Sigma$ is contained in the set $\Sigma_\fa$ of Rees valuations of some flag ideal $\fa$ of $X$; we may thus assume $\Sigma=\Sigma_\fa$, see Proposition~\ref{prop:signorm}~(ii). By Lemma~\ref{lem:ratdivnorm}, we can write $\n=\IN_\Sigma(\tilde\f)$ for some function $\tilde\f\colon\Sigma\to\Q$. By~\cite[Lemma~2.28]{trivval}, there exists $\rho\in\PL^+(X)$ and $r \gg 1$ such that $\f:=r\f_\fa-\rho\in\PL(X)$ satisfies $\f=\tilde\f$ on $\Sigma$, while~\cite[Lemma~2.12]{trivval} shows that
$$
\sup_{X^\an}(\p-\f)=\max_\Sigma(\p-\f)
$$
for all $\p\in\PL^+(X)$, and hence also for all $\p\in\PSH(L)$ (as $\p$ can then be written as a decreasing limit of functions in $\cH_\Q\subset\PL^+(X)$). As above, this implies $\IN(\f)=\n$, which concludes the proof. 
\end{proof}

\begin{proof}[Proof of Corollary~\ref{cor:divtc}] Any norm $\n\in\cT^\hom_\Q$ satisfies $\n=\IN(\f)$ with $\f:=\FS(\n)\in\cH_\Q\subset\PL(X)$ (see Theorem~\ref{thm:INFS} and Proposition~\ref{prop:FSft}). By Theorem~\ref{thm:divPL}, we thus have $\cT^\hom_\Q\subset\cN^\div_\Q$. The last point follows from Lemma~\ref{lem:homogft}. 
\end{proof}

\begin{proof}[Proof of Corollary~\ref{cor:divppty}] By Theorem~\ref{thm:divPL}, $\cN_\R^\div$ is contained in $\cN_\R^\cont$ iff $\IN(\f)$ is continuous for any $\f\in\PL(X)$, \ie $\FS(\IN(\f))\in\Cz(X)$  (see Theorem~\ref{thm:cont}). Since $\f$ is continuous, we have $\FS(\IN(\f))=\qq(\f)=\pp(\f)$ (see Proposition~\ref{prop:FSIN}). Thus $\cN_\R^\div\subset\cN_\R^\cont$ holds iff $\pp(\f)$ is continuous for each $\f\in\PL(X)$. By density of $\PL(X)$ in $\Cz(X)$ and the Lipschitz property of $\pp$ wrt the supnorm, this is also equivalent to the continuity of $\pp(\f)$ for each $\f\in\Cz(X)$, which holds in turn iff $(X,L)$ has the envelope property (see~\cite[Lemma~5.17]{trivval}). 
\end{proof}

%
%
\subsection{Maximal norms and the regularized Fubini--Study operator}\label{sec:maxnorm} 
Specializing now the definitions of~\S\ref{sec:signorm} to $\Sigma:=X^\div$, we introduce:

\begin{defi}\label{defi:Ntilde} We say that a norm $\n\in\cN_\R$ is \emph{maximal} if it lies in $\cN_\R^{\max}:=\cN^{X^\div}_\R$. 
\end{defi}
Explicitly, a norm is maximal iff it can be written as 
$$
\n=\inf_{v\in X^\div}\{\n_v+c_v\}
$$
for a bounded set of constants $(c_v)_{v\in X^\div}$. The (slightly abusive) terminology will be justified by Corollary~\ref{cor:max} below. By Proposition~\ref{prop:signorm}~(ii), we have 
$$
\cN_\R^\div\subset\cN_\R^{\max}\subset\cN_\R^\hom,
$$
both inclusions being strict (except in the trivial case $\dim X=0$), \cf Example~\ref{exam:subvar3} below.  

For each $\n\in\cN_\R$ we set  
$$
\n^{\max}:=\pp_{X^\div}(\n)=\IN_{X^\div}(\FS(\n)). 
$$ 
Then $\n\le\n^\hom\le\n^{\max}$, and $\n^{\max}$ is the smallest norm in $\cN_\R^{\max}$ such that $\n\le\n^{\max}$.

Before going further, recall from~\S\ref{sec:FS} that the Fubini--Study operator associates to any norm $\n\in\cN_\R$ with canonical approximants $\n_d\in\cT_\R$ a bounded, lsc function $\FS(\n)\colon X^\an\to\R$, such that $\FS(\n)=\sup_d\FS(\n_d)$ with $\FS(\n_d)\in\cH_\R$. We denote by $\FSstar(\n):=\FS(\n)^\star$ its usc regularization, which is thus a bounded usc function on $X^\an$. The next result will be instrumental for what follows: 

\begin{lem}\label{lem:FSnegl} For any norm $\n\in\cN_\R$, the following holds:
\begin{itemize}
\item[(i)] $\FSstar(\n)=\FS(\n)$ on $X^\div$; 
\item[(ii)] if $(X,L)$ has the weak envelope property, then $\FSstar(\n)=\FS(\n)$ on $X^\lin$, and $\ten(\FSstar(\n))=\ten(\FS(\n))$; 
\item[(iii)] if $\FSstar(\n)$ is $L$-psh (e.g., if $(X,L)$ has the envelope property), then $\FSstar(\n)$ lies in $\cE^\infty_\uparrow\subset\cE^1$, and $\FS(\n_d)\to\FSstar(\n)$ strongly in $\cE^1$. 
\end{itemize}
\end{lem}
We refer to~\S\ref{sec:env} for the (weak) envelope property and the space $\cE^\infty_\uparrow$ of psh functions approximable from below. Recall that the weak envelope property holds as soon as $\charac k=0$, and that the envelope property then holds if $X$ is further smooth. 

\begin{proof} Since $\FS(\n)=\sup_d\FS(\n_d)$ with $\FS(\n_d)$ $L$-psh, (i) and (ii) respectively follow from Lemmas~\ref{lem:envdiv} and~\ref{lem:envdiv2} (see~\S\ref{sec:env}). If $(X,L)$ has the envelope property, then $\FSstar(\n)$ is $L$-psh, and the rest of (iii) follows from Lemma~\ref{lem:incrdiv}. 
\end{proof}

\begin{prop}\label{prop:maxchar}
For any $\n\in\cN_\R$, we have 
  \begin{itemize}
  \item[(i)] $\n^{\max}=\IN(\FSstar(\n))$; 
  \item[(ii)] $\n$ is maximal iff $\n=\IN(\f)$ for some bounded usc function $\f$ on $X^\an$;
  \end{itemize}
\end{prop}

\begin{cor}\label{cor:maxchar} The space $\cN_\R^{\max}$ is invariant under the scaling action by $\R_{>0}$, the translation action by $\R$, under finite minima, and under decreasing limits. Further, any $\n\in\cN_\R^{\max}$ can be written as the pointwise limit of a decreasing net in $\cN_\R^\div$. 
\end{cor}

\begin{cor}\label{cor:qftmax} Each continuous norm $\n\in\cN_\R^\cont$ satisfies $\n^\hom=\n^{\max}$. In particular, every continuous homogeneous norm is maximal, \ie $\cN_\R^{\cont,\hom}\subset\cN_\R^{\max}$. 
\end{cor}
The first equality fails in general when $\n$ is not continuous (see Example~\ref{exam:subvar3}), and the last inclusion is strict in general (see Corollary~\ref{cor:maxcont} below).  

\begin{cor}\label{cor:linmax} If $(X,L)$ has the weak envelope property, then $\cN_\R^{\max}=\cN_\R^{X^\lin}$; in particular, $\n_v$ is then maximal for any $v\in X^\lin$. 
\end{cor}

\begin{proof}[Proof of Proposition~\ref{prop:maxchar}] Lemma~\ref{lem:FSnegl}~(i) implies that $\n^{\max}=\pp_{X^\div}(\FS\n))$ coincides with $\IN_{X^\div}(\FSstar(\n))$, which is also equal to $\IN(\FSstar(\n))$, since $\FSstar(\n)$ is usc on $X^\an$ and $X^\div$ is dense (see Proposition~\ref{prop:signorm}~(iii)). This proves (i). 

To see (ii), assume $\n$ is maximal. By (i), we then have $\n=\n^{\max}=\IN(\FSstar(\n))$ where $\FSstar(\n)$ is bounded and usc. Conversely, assume $\n=\IN(\f)$ with $\f$ bounded and usc on $X^\an$. By density of $X^\div$, Proposition~\ref{prop:signorm}~(iii) yields $\n=\IN_{X^\div}(\f)$, and hence $\n\in\cN_\R^{\max}$. This proves (ii). 
\end{proof}

\begin{proof}[Proof of Corollary~\ref{cor:maxchar}] For any non-pluripolar $\Sigma\subset X^\an$, the space $\cN^\Sigma_\R$ is invariant under the translation action by $\R$, under finite minima, and under decreasing limits, see Corollary~\ref{cor:decrgin1}. By Proposition~\ref{prop:maxchar}~(ii) and~\eqref{equ:INequiv}, $\cN_\R^{\max}$ is further invariant under the scaling action of $\R_{>0}$ (even though $X^\div$ is only invariant under the scaling action of $\Q_{>0}$). The final assertion is an immediate consequence of Corollary~\ref{cor:decrgin1} and Lemma~\ref{lem:decrgin2}. 
\end{proof}

\begin{proof}[Proof of Corollary~\ref{cor:qftmax}] If $\n\in\cN_\R^{\cont}$ then $\FS(\n)$ is continuous (see Theorem~\ref{thm:cont}), and hence $\n^{\hom}=\IN(\FS(\n))=\n^{\max}$, by Theorem~\ref{thm:INFS} and Proposition~\ref{prop:maxchar}~(i). 
\end{proof}

\begin{proof}[Proof of Corollary~\ref{cor:linmax}] Lemma~\ref{lem:FSnegl}~(ii) implies that $\pp_{X^\lin}(\n)=\IN_{X^\lin}(\FS(\n))$ coincides with $\IN_{X^\lin}(\FSstar(\n))$. By Proposition~\ref{prop:signorm}~(iii), this is also equal to $\IN(\FSstar(\n))$, which is in turn equal to $\n^{\max}$, by Proposition~\ref{prop:maxchar}~(i). Thus $\pp_{X^\lin}(\n)=\n^{\max}$, and hence $\n\in\cN_\R^{X^\lin}\Leftrightarrow\n\in\cN_\R^{\max}$. 
\end{proof}

We can now state the main result of this section: 

\begin{thm}\label{thm:max} For all norms $\n,\n'\in\cN_\R$, we have 
$\n\sim\n'\Longrightarrow\n^{\max}=\n'^{\max}$. If $(X,L)$ has the weak envelope property (e.g., if $\charac k=0$), the converse implication holds. 
\end{thm}

\begin{exam}\label{exam:subvar3} For any subvariety $Z\varsubsetneq X$, the norm $\n=\n_Z\in\cN_\Z^\hom$ of Example~\ref{exam:subvar} is not maximal. Indeed, $\n$ is asymptotically equivalent to the maximal norm $\n_\triv+1$ (see Example~\ref{exam:subvar2}), and hence $\n^{\max}=\n_\triv+1\ne\n$. 
\end{exam}

\begin{cor}\label{cor:d1metr} The restriction of $\dd_1$ to $\cN_\R^{\max}$ a metric. If $(X,L)$ has the weak envelope property, then $\cN_\R^{\max}$ is further maximal for this property. 
\end{cor}

\begin{cor}\label{cor:max} Assume $(X,L)$ has the weak envelope property, and pick any norm $\n\in\cN_\R$. Then $\n$ is maximal iff it is the largest norm in its asymptotic equivalence class. 
\end{cor}

\begin{cor}\label{cor:FSqft2} If $\n,\n'\in\cN_\R^\cont$ are continuous, then $\n\sim\n\Longleftrightarrow\dd_\infty(\n,\n')=0$. 
\end{cor}

As a first step towards Theorem~\ref{thm:max}, we show: 

\begin{lem}\label{lem:max} For all $\n,\n'\in\cN_\R$, the following are equivalent: 
\begin{itemize}
\item[(i)] $\n^{\max}=\n'^{\max}$; 
\item[(ii)] $\FS(\n)=\FS(\n')$ on $X^\div$; 
\item[(iii)] $\FSstar(\n)=\FSstar(\n')$ on $X^\an$.
\end{itemize}
\end{lem}
\begin{proof} Since $\n^{\max}=\pp_{X^\div}(\n)$, Lemma~\ref{lem:FSenv} yields $\FS(\n^{\max})=\FS(\n)$ on $X^\div$, and similarly for $\n'$. This implies (i)$\Rightarrow$(ii), while Lemma~\ref{lem:FSnegl}~(i) yields (ii)$\Leftrightarrow$(iii). Finally, (iii)$\Rightarrow$(i) follows from Proposition~\ref{prop:maxchar}~(i). 
\end{proof}

\begin{lem}\label{lem:asymlin} For all $\n,\n'\in\cN_\R$, we have 
$$
\n\sim\n'\Longleftrightarrow\lim_d\dd_1(\FS(\n_d),\FS(\n'_d))=0\Longrightarrow\FS(\n)=\FS(\n')\text{ on }X^\lin.
$$
\end{lem}
\begin{proof} By construction of the $\dd_1$-metric on $\cH_\R$, $\FS\colon(\cT_\R,\dd_1)\to(\cH_\R,\dd_1)$ is an isometry, and hence $\dd_1(\n_d,\n'_d)=\dd_1(\FS(\n_d),\FS(\n'_d))$ for all $d$ sufficiently divisible. By Theorem~\ref{thm:Fuj}, we have, on the other hand, $\dd_1(\n_d,\n'_d)\to\dd_1(\n,\n')$. This implies the first equivalence. For any $v\in X^\lin$, the measure $\d_v$ lies in $\cM^1$. By~\eqref{equ:BBGZabs2}, we thus have 
$$
\dd_1(\FS(\n_d),\FS(\n'_d))\to 0\Longrightarrow\FS(\n_d)(v)-\FS(\n'_d)(v)\to 0,
$$
which yields the right-hand implication, since $\FS(\n_d)\to\FS(\n)$ and $\FS(\n'_d)\to\FS(\n')$ pointwise on $X^\an$. 
\end{proof}

\begin{proof}[Proof of Theorem~\ref{thm:max}] If $\n\sim\n'$, then $\FS(\n)=\FS(\n')$ on $X^\lin\supset X^\div$, by Lemma~\ref{lem:asymlin}, and hence $\n^{\max}=\n'^{\max}$, by Lemma~\ref{lem:max}. Now assume the weak envelope property. To prove the converse implication, it suffices to show $\n\sim\n^{\max}$ for any $\n\in\cN_\R$. Since $\n\le\n^{\max}$, this amounts to $\vol(\n)=\vol(\n^{\max})$ (see Lemma~\ref{lem:inequiv}). By Theorem~\ref{thm:volen}, we have $\vol(\n)=\ten(\f)$ with $\f:=\FS(\n)$. On the other hand, we have $\n^{\max}=\IN(\f^\star)$ (see Proposition~\ref{prop:maxchar}~(i)), and hence $\vol(\n^{\max})=\ten((\f^\star)_\star)$, by Corollary~\ref{cor:volIN}. Since $\f$ is lsc, we have $\f\le(\f^\star)_\star\le\f^\star$, so by monotonicity of the energy, it suffices to prove that $\ten(\f)=\ten(\f^\star)$, which follows from Lemma~\ref{lem:FSnegl}. 
\end{proof}  

\begin{proof}[Proof of Corollary~\ref{cor:d1metr}] That $\dd_1$ restricts to a metric on $\cN_\R^{\max}$ is a direct consequence of Theorem~\ref{thm:max}. If $\dd_1$ is also a metric on a subset $\cN'\subset\cN_\R$ that contains $\cN_\R^{\max}$, then any $\n\in\cN'$ satisfies $\dd_1(\n,\n^{\max})=0$, by Theorem~\ref{thm:max}, and hence $\n=\n^{\max}\in\cN_\R^{\max}$. Thus $\cN'=\cN_\R^{\max}$. 
\end{proof}

\begin{proof}[Proof of Corollary~\ref{cor:max}] Assume $\n$ is maximal, and pick $\n'\in\cN_\R$ with $\n\sim\n'$. Then $\n'\le\n'^{\max}=\n^{\max}=\n$, by Theorem~\ref{thm:max}, which proves that $\n$ is the largest norm in its equivalence class. Conversely, this last property implies $\n^{\max}\le\n$, since $\n\sim\n^{\max}$ by Theorem~\ref{thm:max}, and hence $\n=\n^{\max}$, \ie $\n\in\cN_\R^{\max}$. 
\end{proof}

 \begin{proof}[Proof of Corollary~\ref{cor:FSqft2}] By Theorem~\ref{thm:cont}, $\FS(\n)$ and $\FS(\n')$ are continuous, and are equal iff $\dd_\infty(\n,\n')=0$. By Lemma~\ref{lem:max}, we thus have $\n\sim\n'\Rightarrow\dd_\infty(\n,\n')=0$, while the converse trivially holds. 
\end{proof}

Assuming now the envelope property (\eg $X$ is smooth and $\charac k=0$), we finally state: 

\begin{thm}\label{thm:FSstar} If $(X,L)$ has the envelope property, then the regularized Fubini--Study operator defines a surjective isometry $\FSstar\colon(\cN_\R,\dd_1)\twoheadrightarrow(\cE^\infty_\uparrow,\dd_1)$, which restricts to an isometric isomorphism $\FSstar\colon(\cN_\R^{\max},\dd_1)\simto(\cE^\infty_\uparrow,\dd_1)$ with inverse $\IN\colon(\cE^\infty_\uparrow,\dd_1)\simto(\cN_\R^{\max},\dd_1)$. 
\end{thm}

\begin{proof} Since $(X,L)$ has the envelope property, $\FSstar(\n)$ lies in $\cE^\infty_\uparrow$ for each $\n\in\cN_\R$ (see Lemma~\ref{lem:FSnegl}~(iii)). Since $\FS\colon(\cT_\R,\dd_1)\to(\cH_\R,\dd_1)$ is an isometry, the canonical approximants $\n_d,\n'_d\in\cT_\R$ satisfy $\dd_1(\n_d,\n'_d)=\dd_1(\FS(\n_d),\FS(\n'_d))$ for all $d$ sufficiently divisible. Now, Theorem~\ref{thm:Fuj} implies on the one hand $\dd_1(\n_d,\n'_d)\to\dd_1(\n,\n')$. On the other hand, Lemma~\ref{lem:FSnegl}~(iii) implies $\dd_1(\FS(\n_d),\FS(\n'_d))\to\dd_1(\FSstar(\n),\FSstar(\n'))$, since $\dd_1$ defines the strong topology of $\cE^1$ (see Theorem~\ref{thm:Darvas}). This proves that $\FSstar\colon(\cN_\R,\dd_1)\to(\cE^\infty_\uparrow,\dd_1)$ is an isometry, whose restriction to $\cN_\R^{\max}$ is necessarily injective, by Theorem~\ref{thm:max}. Conversely, pick $\f\in\cE^\infty_\uparrow$. Then $\IN(\f)\in\cN^{\max}_\R$ (see Proposition~\ref{prop:maxchar}~(ii)). By Proposition~\ref{prop:FSIN} and Lemma~\ref{lem:envreg}~(i), we further have $\FSstar(\IN(\f))=\qq^\star(\f)=\f$. This shows that $\FSstar\colon(\cN_\R^{\max},\dd_1)\simto(\cE^\infty_\uparrow,\dd_1)$ is an isometric isomorphism, and the rest follows. 

\end{proof}

\begin{cor}\label{cor:maxcont} Assume $(X,L)$ has the envelope property. Then 
$$
\cN_\R^{\max}\subset\cN_\R^\cont\Longleftrightarrow\cE^\infty_\uparrow=\CPSH\Longleftrightarrow\dim X\le 1.
$$   
\end{cor}
\begin{proof} The first equivalence follows from Theorem~\ref{thm:contisom} and Theorem~\ref{thm:FSstar}, and the second one from~\cite[Example~13.25]{trivval}). 
\end{proof}

%
%
\section{The Monge--Amp\`ere measure of a norm}\label{sec:MAnorm}
Any ample test configuration for $(X,L)$ defines a measure on $X^\an$ with finite support in $X^\div$. This defines a Monge--Amp\`ere operator $\MA\colon\cT_\Z\to\cM^1$. Here we extend this operator to a map $\cN_\R\to\cM^1$ whose fibers consist of asymptotic equivalence classes modulo translation, thus completing the proof of Theorem~B. The construction, which works even when in the absence of the envelope property, restricts to a homeomorphism between divisorial norms modulo translations and probability measures with finite support in $X^\div$, thus proving Theorem~C.
We also extend Dervan's minimum norm functional from $\cT_\Z$ to $\cN_\R$. 
%
%
\subsection{Monge--Amp\`ere measures of $\R$-test configurations}\label{sec:MARtc}
We define the \emph{Monge--Amp\`ere measure} $\MA(\n)\in\cM^1$ of an $\R$-test configuration $\n\in\cT_\R$ as the Monge--Amp\`ere measure of the associated Fubini--Study function $\FS(\n)\in\cH_\R$ (see Proposition~\ref{prop:FSft}), \ie 
\begin{equation*}
\MA(\n):=\MA(\FS(\n)). 
\end{equation*}
The invariance/equivariance properties of the operators $\MA\colon\cH_\R\to\cM^1$ and $\FS\colon\cT_\R\to\cH_\R$ imply that if $\n\in\cT_\R$, $c\in\R$ and $t\in\R_{>0}$, then 
\begin{equation*}
  \MA(\n+c)=\MA(\n)\quad\text{and}\quad\MA(t\n)=t_\star\MA(\n).
\end{equation*}

When $\n\in\cT_\Z$, the Monge--Amp\`ere measure can be computed geometrically as follows. By the Rees correspondence~\eqref{equ:Rees}, $\n$ is associated to an ample test configuration. Let $(\cX,\cL)$ be its integral closure, with central fiber $\cX_0=\sum_ib_iE_i$. By Lemma~\ref{lem:FStest}, $\f:=\FS(\n)\in\cH_\Q$ satisfies~\eqref{equ:FStc}, so by~\cite[Proposition~7.19~(ii)]{trivval} we have 
\begin{equation}\label{equ:MATZ}
  \MA(\n)=\sum_ib_i(\cL|_{E_i}^n)\d_{v_i},
\end{equation}
where $v_i\in X^\div$ is the divisorial valuation associated to $E_i$.

For general $\n\in\cT_\R$, the support of $\MA(\n)$ is a finite subset of $X^\lin$, see 
Lemma~\ref{lem:suppMA}.
\begin{lem}\label{lem:MATR}
  The Monge--Amp\`ere operator above defines an isometry
  \begin{equation}\label{equ:MATR}
    \MA\colon(\cT_\R/\R,\quotd_1)\to(\cM^1,\dd_1)
  \end{equation}
  with dense image.
\end{lem}
\begin{proof}
  By~\eqref{equ:FSquot}, the Fubini--Study operator defines an isometric surjection $\FS\colon(\cT_\R/\R,\quotd_1)\to(\cH_\R/\R,\quotd_1)$. Now $\cH_\R/\R$ is a dense subspace of $(\cE^1/\R,\quotd_1)$, and by definition, the metric $\dd_1$ on $\cM^1$ has the property that $\MA\colon(\cE^1/\R,\quotd_1)\to(\cM^1,\dd_1)$ is an injective isometry with dense image, see Theorem~\ref{thm:d1M1}. It therefore follows that~\eqref{equ:MATR} is an isometry with dense image.
\end{proof}
%
%
\subsection{Monge--Amp\`ere measures of general norms}\label{sec:MAgennorm}
We now define the Monge--Amp\`ere operator on general norms.
\begin{thm}\label{thm:Hauscomp}
  The Monge--Amp\`ere operator above extends uniquely to an isometry
  \begin{equation}\label{equ:MANR}
    \MA\colon(\cN_\R/\R,\quotd_1)\to(\cM^1,\dd_1),
  \end{equation}
 with dense image.
\end{thm}
As $(\cM^1,\dd_1)$ is complete, the Monge--Amp\`ere operator thus realizes $(\cM^1,\dd_1)$ as the Hausdorff completion of the pseudo-metric spaces $(\cT_\R/\R,\quotd_1)$ and $(\cN_\R/\R,\quotd_1)$.
\begin{proof}
  By Theorem~\ref{thm:Fuj}, $\cT_\R$ is dense in $(\cN_\R,\dd_1)$. As a consequence, $\cT_\R/\R$ sits as a dense subspace of $(\cN_\R/\R,\quotd_1)$, and we conclude using Lemma~\ref{lem:MATR}.
\end{proof}

Combining Theorem~\ref{thm:Hauscomp} with Lemma~\ref{lem:distach}, we get: 
\begin{cor}\label{cor:asymptrans} The induced map $\MA\colon(\cN_\R,\dd_1)\to(\cM^1,\dd_1)$ is $1$-Lipschitz, and its nonempty fibers consist precisely of asymptotic equivalence classes of norms modulo translation.
\end{cor}

The induced map $\MA\colon\cN_\R\to\cM^1$ satisfies the following properties, the first of which gives a more concrete description.
\begin{prop}\label{prop:MAnorm}
  For any norm $\n\in\cN_\R$ we have:
  \begin{itemize}
  \item[(i)]
    the canonical approximants $(\n_d)$ satisfy $\lim_d\MA(\n_d)=\MA(\n)$ strongly in $\cM^1$;
    \item[(ii)]
      for $c\in\R$ and $t\in\R_{>0}$, we have 
      $\MA(\n+c)=\MA(\n)$ and $\MA(t\n)=t_\star\MA(\n)$;
    \item[(iii)]
      $\MA(\n)=\MA(\n^\hom)$,
      where $\n^\hom$ is the homogenization of $\n$;
    \item[(iv)]
      If $(X,L)$ has the weak envelope property, then 
      $\MA(\n)=\MA(\n^{\max})$. 
    \item[(v)] if $\FSstar(\n)$ is $L$-psh, then $\MA(\n)=\MA(\FSstar(\n))$. 
    \end{itemize}
 \end{prop}
Recall that (v) applies if $(X,L)$ has the envelope property, or for any continuous norm $\n\in\cN_\R^\cont$. 
  
  \begin{proof}
    Theorem~\ref{thm:Fuj} shows that $\n_d\in\cT_\R$ satisfy $\lim_d\dd_1(\n_d,\n)=0$, which implies~(i). The equalities in~(ii) follow, since $(\n+c)_d=\n_d+c$ and $(t\n)_d=t\n_d$, whereas~(iii) and~(iv) follow since $\n^\hom\sim\n$ and $\n^{\max}\sim\n$, respectively, see Corollary~\ref{cor:homequiv} and Theorem~\ref{thm:max}. If $\FSstar(\n)$ is $L$-psh, then $\FS(\n_d)\to\FSstar(\n)$ strongly in $\cE^1$ (see Lemma~\ref{lem:FSnegl}), and hence $\MA(\n_d)=\MA(\FS(\n_d))\to\MA(\FSstar(\n))$ in $\cM^1$. This proves (v), in view of (i). 
  \end{proof}

Recall the space $\cN_\R^{\max}=\cN^{X^\div}_\R$ from~\S\ref{sec:maxnorm}. By Corollary~\ref{cor:d1metr}, the pseudometric $\dd_1$ restricts to a metric on $\cN_\R^{\max}$.  Lemma~\ref{lem:quotdbd} thus implies that $\quotd_1$  restricts to a metric on $\cN_\R^{\max}/\R$.
  
  \begin{cor}\label{cor:MAemb}
    The Monge--Amp\`ere operator $\MA\colon\cN_\R\to\cM^1$ induces an isometric embedding 
\begin{equation}\label{equ:MAmax}
\MA\colon(\cN_\R^{\max}/\R,\quotd_1)\hto(\cM^1,\dd_1)
\end{equation}
with dense image. If $(X,L)$ has the weak envelope property, then the image equals $\MA(\cN_\R)$. 
\end{cor}
Recall that the weak envelope property holds when $\charac k=0$ or $\dim X\le 2$.
\begin{proof}
  Everything except for the last statement is clear by what precedes, and that statement is an immediate consequence of Theorem~\ref{thm:max}.
\end{proof}

\begin{rmk}
  Even if $X$ is smooth and $\charac k=0$, $\MA(\cN_\R)$ is a strict subspace of $\cM^1$, which is not so easy to describe. See~\cite{nagreen} for related questions.
\end{rmk}
For later use, we also show the following version of Theorem~\ref{thm:d1andI1}. 

\begin{lem}\label{lem:d1andI1norm}
For all $\n,\n'\in\cN_\R$ we have 
$$
\dd_1(\n,\n')\approx\int\left|\FS(\n)-\FS(\n')\right|\left(\MA(\n)+\MA(\n')\right). 
$$
\end{lem}
\begin{proof} Set $\f_d:=\FS(\n_d)$, $\f'_d:=\FS(\n'_d)$ and 
$$
\mu_d:=\MA(\f_d)=\MA(\n_d),\quad\mu'_d:=\MA(\f'_d)=\MA(\n'_d),
$$
with $(\n_d)$, $(\n'_d)$ the canonical approximants of $\n,\n'$. Since $\FS\colon(\cT_\R,\dd_1)\to(\cH_\R,\dd_1)$ is an isometry, Theorem~\ref{thm:d1andI1} yields
$$
\dd_1(\n_d,\n'_d)=\dd_1(\f_d,\f'_d)\approx\int g_d\,(\mu_d+\mu'_d). 
$$
with $g_d:=|\f_d-\f'_d|$. Since $(\f_d)$, $(\f'_d)$ are uniformly bounded and $\mu_d\to\mu:=\MA(\n)$, $\mu'_d\to\mu':=\MA(\n')$ strongly in $\cM^1$ (see Proposition~\ref{prop:MAnorm}~(i)), \eqref{equ:BBGZabs} yields 
$\int g_d\,(\mu_d+\mu'_d)=\int g_d\,(\mu+\mu')+o(1)$. Since $(g_d)$ is uniformly bounded and converges pointwise to $g:=|\FS(\n)-\FS(\n')|$, dominated convergence applied to the cofinal sequence $(g_{m!})_m$ further yields $\int g_{m!}(\mu+\mu')\to\int g\,(\mu+\mu')$. Combining this with Theorem~\ref{thm:Fuj}, we conclude
$$
\dd_1(\n,\n')=\lim_m \dd_1(\n_{m!},\n'_{m!})\approx\lim_m\int g_{m!}\,(\mu+\mu')=\int g\,(\mu+\mu'),
$$
which proves the result. 
\end{proof}

%
%
\subsection{Variational principle}\label{sec:varprinc}
As we next show, the Monge--Amp\`ere equation $\MA(\n)=\mu$ with $\n\in\cN_\R$ and $\mu\in\cM^1$ admits a variational characterization, that will be deduced from its counterpart for $L$-psh functions. 

\begin{prop}\label{prop:MAvar} For any $\mu\in\cM^1$, we have 
$$
\en^\vee(\mu)=\sup_{\n\in\cN_\R}\left(\vol(\n)-\int\FS(\n)\,\mu\right). 
$$
Further, the supremum is achieved by $\n\in\cN_\R$ iff $\MA(\n)=\mu$. 
\end{prop}
\begin{proof} We have $\en^\vee(\mu)=\sup_{\f\in\cH_\Q}\left(\en(\f)-\int\f\,\mu\right)$, and any $\f\in\cH_\Q$ can be written as $\f=\FS(\n)$ with $\n:=\IN(\f)\in\cN_\R$. Since $\en(\f)=\vol(\n)$, this yields 
$$
\en^\vee(\mu)\le\sup_{\n\in\cN_\R}(\vol(\n)-\int\FS(\n)\,\mu).
$$
For the reverse inequality, pick any $\n\in\cN_\R$, and consider the increasing net $(\f_d)$ in $\cH_\R$ defined by $\f_d:=\FS(\n_d)$, with $(\n_d)$ the canonical approximants of $\n$. Then $\en(\f_d)-\int\f_d\,\mu\le\en^\vee(\mu)$. By Theorem~\ref{thm:Fuj} and Theorem~\ref{thm:volen}, $\en(\f_d)=\vol(\n_d)\to\vol(\n)$, while $\int\f_d\,\mu\to\int\FS(\n)\,\mu$, by monotone convergence (applied to the cofinal sequence $\f_{d!}$). Thus $\vol(\n)-\int\FS(\n)\,\mu\le\en^\vee(\mu)$, and equality holds iff $\en(\f_d)-\int\f_d\,\mu\to\en^\vee(\mu)$, \ie $(\f_d)$ is a maximizing net for $\mu$. By~\cite[Corollary~10.13]{trivval}, the latter is also equivalent to $\MA(\f_d)\to\mu$ strongly in $\cM^1$, and hence to $\MA(\n)=\mu$, since $\MA(\f_d)=\MA(\n_d)\to\MA(\n)$ strongly in $\cM^1$, by Corollary~\ref{cor:asymptrans}.
\end{proof} 

\begin{rmk} With a little bit of extra effort, one can show as in~\cite[Corollary~10.13]{trivval} that a net $(\n_i)$ in $\cN_\R$ computes the supremum, that is $\lim_i\left(\vol(\n_i)-\int\FS(\n_i)\,\mu\right)=\en^\vee(\mu)$, iff $\MA(\n_i)\to\mu$ strongly in $\cM^1$.
\end{rmk}

%
%
\subsection{Divisorial norms and divisorial measures}\label{sec:MAdiv}
The image $\MA(\cN_\R)$ of the Monge--Amp\`ere operator is a strict subset of $\cM^1$ and not so easy to describe, but we now exhibit an important class of measures contained in the image.

Given any compact subset $\Sigma\subset X^\an$, denote by $\cM^\Sigma$ the set of Radon probability measures $\mu$ on $X^\an$ with support in $\Sigma$. 

\begin{exam} When $\Sigma\subset X^\an$ is finite, each $\mu\in\cM^\Sigma$ is of the form $\mu=\sum_{v\in\Sigma}m_v\d_v$, where $m_v:=\mu(\{v\})$, and it is easy to see that $\mu\mapsto m=(m_v)$ defines a homeomorphism of $\cM^\Sigma$ (equipped with the weak topology) with the simplex $\left\{m\in\R_{\ge 0}^\Sigma\mid\sum_v m_v=1\right\}$. \end{exam}

Recall that the strong topology of $X^\lin$ is defined by the metric $\dd_\infty$ (see~\eqref{equ:dinftylin}); the weak topology refers to the subset topology from $X^\an$. For all $v,w\in X^\lin$ we have  
\begin{equation}\label{equ:dinftypsh}
\dd_\infty(v,w)=\sup_{\f\in\PSH}|\f(v)-\f(w)|,
\end{equation}
which shows that $\dd_\infty$ is the smallest metric on $X^\lin$ such that the restriction to $X^\lin$ of any $L$-psh function is $1$-Lipschitz. By~\eqref{equ:dinftypsh}, the weak and strong topologies coincide on a given subset $\Sigma\subset X^\lin$ iff $\PSH|_\Sigma$ is equicontinuous for the weak topology of $\Sigma$. This is in particular the case when $\Sigma$ is strongly compact (as the identity map $(\Sigma,\text{strong})\to(\Sigma,\text{weak})$ is then a homeomorphism, being continuous and bijective on a compact Hausdorff space).

\begin{exam} Every finite subset $\Sigma\subset X^\lin$ is of course strongly compact. If $X$ is smooth and $\charac k=0$, then the dual complex $\D_\cX$ of any snc test configurations $\cX$ also forms a strongly compact subset of $X^\lin$, \cf~\cite[Theorem~A.4]{trivval}. 
\end{exam}

\begin{lem} For any strongly compact subset $\Sigma\subset X^\lin$, we have $\cM^\Sigma\subset\cM^1$, and the induced weak and strong topologies on $\cM^\Sigma$ coincide. 
\end{lem}
\begin{proof} Since $\Sigma$ is strongly compact, $C:=\sup_{v\in\Sigma}\tee(v)$ is finite, and satisfies $\sup\f-\f(v)\le C$ for each $\f\in\PSH$ and $v\in\Sigma$. For each $\mu\in\cM^\Sigma$ we thus have 
$$
\en^\vee(\mu)=\sup_{\f\in\cE^1}\{\en(\f)-\int\f\,\mu\}\le\sup_{\f\in\cE^1}\{\sup\f-\int\f\,\mu\}\le C,
$$
and hence $\mu\in\cM^1$. Now pick a weakly convergent net $\mu_i\to\mu$ in $\cM^\Sigma$. Since $\PSH|_\Sigma$ is equicontinuous, we have $\int\f\,\mu_i\to\int\f\,\mu$ uniformly for $\f\in\PSH$. Thus 
$$
\en^\vee(\mu_i)=\sup_{\f\in\cE^1}\{\en(\f)-\int\f\,\mu\}\to\sup_{\f\in\cE^1}\{\en(\f)-\int\f\,\mu\}=\en^\vee(\mu)
$$
and hence $\mu_i\to\mu$ strongly in $\cM^1$.
\end{proof}

\begin{thm}\label{thm:MAsig} For any strongly compact subset $\Sigma\subset X^\lin$, the Monge--Amp\`ere operator induces a surjective isometry
$$
\MA\colon(\cN^\Sigma_\R/\R,\quotd_1)\twoheadrightarrow(\cM^\Sigma,\dd_1).
$$
If $\Sigma\subset X^\div$, or if the weak envelope property holds (e.g., if $\charac k=0$), then this map is an isometric isomorphism. 
\end{thm}
Recall that $\cN^\Sigma_\R$ denotes the set of norms of the form $\n=\IN_\Sigma(\f)$ for a bounded function $\f\colon\Sigma\to\R$, see~\S\ref{sec:signorm}. We emphasize that Theorem~\ref{thm:MAsig} is true for an \emph{arbitrary} polarized variety, whether or not the envelope property holds. The following important special case illustrates this.

\begin{exam}\label{exam:valnorm}
For each $v\in X^\lin$ we have $\MA(\n_v)=\d_v$. If the envelope property holds for $(X,L)$, then the function $\f_v=\FS(\n_v)$ belongs to $\CPSH$, $\f_v(v)=0$ and $\MA(\f_v)=\d_v$. However, in general the equation $\MA(\f)=\d_v$ may not have any solution in $\cE^1$. This is the case, for example, when $X$ is a nodal curve and $v$ is a divisorial valuation with center at the node (compare Example~\ref{exam:nodalcont}). 
\end{exam}

Another important special case is when $\Sigma\subset X^\div$ is finite. Recall that the set  $\cN^\div_\R$ of divisorial norms is the union of $\cN^\Sigma_\R$ over all nonempty finite subset $\Sigma\subset X^\div$. We similarly introduce: 
\begin{defi}\label{defi:divmeas}
  The set $\cM^\div$ of \emph{divisorial measures} on $X^\an$ is defined by
\begin{equation}\label{equ:Mdivcup}
\cM^\div=\bigcup_{\Sigma\subset X^\div\text{ finite}}\cM^\Sigma.
\end{equation}
\end{defi}
The set $\cM^\div$ is used in~\cite{nakstab2} to define the notion of divisorial stability.

\begin{cor}\label{cor:MAdiv} The Monge--Amp\`ere operator induces an isometric isomorphism  
\begin{equation*}
\MA\colon(\cN^\div_\R/\R,\quotd_1)\simto(\cM^\div,\dd_1). 
\end{equation*}
Further, for any $\n\in\cN_\R^\div$, $\Sigma:=\supp\MA(\n)$ is the smallest finite subset of $X^\div$ such that $\n\in\cN_\R^\Sigma$. 
\end{cor}

\begin{exam}\label{exam:linnondiv} If $v\in X^\lin$, then $\n_v$ is divisorial iff $v$ is divisorial. Indeed, $\n_v\in\cN_\R^\div\Longrightarrow\MA(\n_v)=\d_v\in\cM^\div\Longrightarrow v\in X^\div$. 
\end{exam}

We now turn to the proof of Theorem~\ref{thm:MAsig}. For any $\n\in\cN_\R$ and $\f\in\Cz(X)$, we set
$$
\n[\f]:=\IN(\FS(\n)+\f)\in\cN_\R. 
$$
Thus $\n[0]=\n^\hom$ (see Theorem~\ref{thm:INFS}). The main ingredient in the proof is now the following version of~\cite[Theorem~8.5]{trivval} (itself a consequence of~\cite[Theorem A]{BGM}). 

\begin{lem}\label{lem:BGM} For any $\n\in\cN_\R$ and $\f\in\Cz(X)$, we have 
$$
\frac{d}{dt}\bigg|_{t=0}\vol\left(\n[t\f]\right)=\int \f\,\MA(\n).
$$
\end{lem}

\begin{proof}
  For $d$ sufficiently divisible, set $\p_d:=\FS(\n_d)\in\cH_\R$. By Theorem~\ref{thm:volen} and Corollary~\ref{cor:volIN}, we have $\vol(\n)=\ten(\FS(\n))=\vol(\n[0])$, and 
\begin{equation}\label{equ:voltw} 
\vol\left(\n_d[\f]\right)=\ten(\p_d+\f)\to\ten(\FS(\n)+\f)=\vol(\n[\f]). 
\end{equation}
Assume first $\f\in\PL(X)$. By~\cite[Theorem 8.5]{trivval}, we then have 
$$
\ten(\p_d+t\f)=\en(\p_d)+t\int\f\,\MA(\p_d)+O(t^2)
$$
as $t\to 0$, where the implicit contant in $O$ is uniform with respect to $d$ (but does depend on $\f$). Now $\MA(\p_d)=\MA(\n_d)\to\MA(\n)$ strongly in $\cM^1$ (see Proposition~\ref{prop:MAnorm}~(i)); combined with~\eqref{equ:voltw}, this yields
$$
\vol\left(\n[t\f]\right)=\vol(\n)+t\int\f\,\MA(\n)+O(t^2),
$$
which proves the result for $\f\in\PL(X)$. Consider now an arbitrary $\f\in\Cz(X)$. Since $\PL(X)$ is dense in $\Cz(X)$ with respect to uniform convergence, we can find a sequence $(\f_i)$ in $\PL(X)$ such that $\d_i:=\sup_{X^\an}|\f_i-\f|\to 0$. Then $\f_i-\d_i\le\f\le \f_i+\d_i$, and hence 
$$
\vol\left(\n[t \f_i]\right)-t\d_i\le\vol\left(\n[t \f]\right)\le\vol\left(\n[t\f_i]\right)+t\d_i. 
$$
By the first part of the proof, this yields 
\begin{multline*}
  \int \f_i\,\MA(\n)-\d_i\le\liminf_{t\to 0_+}t^{-1}\left(\vol\left(\n[t \f]\right)-\vol(\n)\right)\\
  \le \limsup_{t\to 0_+}t^{-1}\left(\vol\left(\n[t \f]\right)-\vol(\n)\right)\le\int \f_i\,\MA(\n),
\end{multline*}
and letting $i\to\infty$ yields, as desired, $\lim_{t\to 0}t^{-1}\left(\vol\left(\n[t \f]\right)-\vol(\n)\right)=\int \f\,\MA(\n)$. 
\end{proof}

\begin{proof}[Proof of Theorem~\ref{thm:MAsig}] By Theorem~\ref{thm:Hauscomp}, $\MA\colon(\cN_\R/\R,\quotd_1)\to(\cM^1,\dd_1)$ is an isometry. Let us first show that it maps $\cN_\R^\Sigma/\R$ into $\cM^\Sigma$. Pick $\n\in\cN_\R^\Sigma$. We need to show that $\int\f\MA(\n)=0$ for any $\f\in\Cz$ such that $\f|_\Sigma=0$. Now, for any $t\in\R$, we have
\begin{equation*}
  \n[t\f]
  =\IN(\FS(\n)+t\f)
  \le\IN_\Sigma(\FS(\n)+t\f)
  =\IN_\Sigma(\FS(\n))
  =\n.
\end{equation*}
This implies $\vol(\n[t\f])\le\vol(\n)$ for all $t\in\R$, and hence $\int\f\MA(\n)=0$, thanks to Lemma~\ref{lem:BGM}. 

We next show that $\MA\colon\cN_\R^\Sigma\to\cM^\Sigma$ is onto. Pick $\mu\in\cM^\Sigma$, and choose a maximizing sequence $(\f_i)$ in $\cH_\R$ for $\mu$, \ie $\en(\f_i)-\int\f_i\,\mu\to\en^\vee(\mu)$, normalized by $\sup\f_i=0$. Since $\Sigma$ is strongly compact, the restriction of $\PSH_{\sup}=\{\f\in\PSH\mid\sup\f=0\}$ to $\Sigma$ is equicontinuous and bounded, since $0\le\sup_{v\in\Sigma}(-\f(v))\le\sup_{v\in\Sigma}\tee(v)<\infty$ for $f\in\PSH_{\sup}$. 
By the Arzel\`a--Ascoli theorem, we may assume, after passing to a subsequence, that $\f_i|_\Sigma$ converges uniformly to some $\f\in\Cz(\Sigma)$. We claim that $\n:=\IN_\Sigma(\f)\in\cN_\R^\Sigma$ satisfies $\MA(\n)=\mu$, which will conclude the proof. By Proposition~\ref{prop:MAvar}, it suffices to show $\vol(\n)-\int\FS(\n)\,\mu\ge\en^\vee(\mu)$. 

Set $\n_i:=\IN_\Sigma(\f_i)$. As $\f_i\to\f$ uniformly on $\Sigma$, \eqref{equ:INsiglip} implies $\dd_1(\n_i,\n)\to 0$, and hence $\vol(\n_i)\to\vol(\n)$. Further, $\n_i\ge\IN(\f_i)$, and hence $\vol(\n_i)\ge\vol(\IN(\f_i))=\en(\f_i)$ (see Corollary~\ref{cor:volIN}). By Proposition~\ref{prop:signorm}~(i), we also have $\FS(\n)|_\Sigma\le\f$. This yields, as desired, 
\begin{multline*}
 \vol(\n)-\int\FS(\n)\,\mu\ge\vol(\n)-\int\f\,\mu\\
 =\lim_i\left(\vol(\n_i)-\int\f_i\,\mu\right)\ge \lim_i\left(\en(\f_i)-\int\f_i\,\mu\right)=\en^\vee(\mu). 
\end{multline*}
Finally, if $\Sigma\subset X^\div$, or if the weak envelope property holds, then $\cN^\Sigma_\R$ is contained in $\cN_\R^{\max}$ (see Corollary~\ref{cor:linmax}), and the last point thus follows from Corollary~\ref{cor:MAemb}. 
\end{proof}
%
%
\subsection{Dervan's minimum norm}\label{sec:minnorm}
In~\cite{Der}, Dervan introduced the notion of the \emph{minimum norm} of a test configuration. Here we extend his notion to arbitrary norms.

\begin{defi}\label{defi:minnorm}
  We define the \emph{minimum norm} $\|\n\|$ of $\n\in\cN_\R$ by
  \[
    \|\n\|:=\en^\vee(\MA(\n))\in\R_{\ge0}.
  \]
\end{defi}
By Corollary~\ref{cor:asymptrans}, the minimum norm is a continuous function on  $(\cN_\R,\dd_1)$.
\begin{prop}\label{prop:minnorm}
  For any $\n\in\cN_\R$, we have:  
  \begin{itemize}
  \item[(i)]
    if $\n\in\cT_\Z$ is associated to an ample test configuration, then $\|\n\|$ coincides, up to normalization, with the minimum norm defined in~\cite{Der};
  \item[(ii)]
    the canonical approximants $(\n_d)$ satisfy $\|\n\|=\lim_d\|\n_d\|$;
  \item[(iii)]
    for any $c\in\R$ and $t\in\R_{>0}$ we have
    $\|\n+c\|=\|\n\|,$ and $\|t\n\|=t\|\n\|$;
  \item[(iv)]
    $\|\n\|\approx\quotd_1(\n,\n_\triv)\approx\|\n\|_1$; in particular, 
    $\|\n\|=0$ iff $\n\sim\n_\triv+c$ for some $c\in\R$;
  \item[(v)]
    if $\n\in\cT_\R$, then $\|\n\|=\en(\f)-\int\f\,\MA(\f)=\ii(\f)-\jj(\f)$ with $\f:=\FS(\n)$;
  \item[(vi)]
    if $\n\in\cT_\Q$, then $\|\n\|\in\Q$;
  \item[(vii)]
    $\|\n\|=\|\n^\hom\|$;
  \item[(viii)]
    if $(X,L)$ has the weak envelope property, then $\|\n\|=\|\n^{\max}\|$.
  \end{itemize}
\end{prop}
Here $\|\n\|_1=\dd_1(\n,\n_\triv+\vol(\n))$ is the $L^1$-norm of $\n$, see Definition~\ref{defi:Lpnorm}.
\begin{proof}
  If $\n\in\cT_\R$, then $\MA(\n)=\MA(\f)$ with $\f:=\FS(\n)\in\cH_\R$, so~\eqref{equ:enMA} yields $\|\n\|=\en^\vee(\mu)=\ii(\f)-\jj(\f)$, proving~(v), and also~(vi), since $\f\in\cH_\Q$ when $\n\in\cT_\Q$. If, further, $\n\in\cT_\Z$ is a test configuration, then~\cite[Remark~7.12]{BHJ1} shows that $\ii(\f)-\jj(\f)$ coincides
  (up to normalization by $V_L$) with the minimum norm of $\n$ as defined in~\cite[Definition~2.5]{Der}. Thus~(i) holds. Now~(ii),~(iii),~(vii) and~(viii) are immediate consequence of the corresponding properties in Proposition~\ref{prop:MAnorm}.

  It remains to prove~(iv). That $\|\n\|\approx\quotd_1(\n,\n_\triv)$ follows from~\eqref{equ:normmu} and Theorem~\ref{thm:Hauscomp}, whereas $\|\n\|_1\approx\|\n\|$ follows from (iv) and~\cite[Theorem~7.9]{BHJ1} when $\n\in\cT_\R$, and hence in general, by density. 
 \end{proof}
By $\dd_1$-density of $\cT_\Z$ in $\cN_\R$ (see Corollary~\ref{cor:tcdense}), we infer: 

\begin{cor}\label{cor:minnormext}
  The minimum norm functional $\cN_\R\to\R_{\ge0}$ is the unique $\dd_1$-continuous extension of Dervan's minimum norm from $\cT_\Z$ to $\cN_\R$.
\end{cor}


%
%
%
%
\subsection{Valuations of linear growth}\label{sec:lingrowth}

In this final section, we specialize the above results to prove:
\begin{thm}\label{thm:thmC} For all $v,w\in X^\lin$ with associated norms $\n_v,\n_w\in\cN_\R$ and measures $\d_v,\d_w\in\cM^1$, we have: 
\begin{itemize}
\item[(i)] $\dd_\infty(v,w)=\dd_\infty(\n_v,\n_w)\approx\quotd_1(\n_v,\n_w)=\dd_1(\d_v,\d_w)$; 
\item[(ii)] $\dd_\infty(v,v_\triv)=\tee(v)=\la_{\max}(\n_v)$; 
\item[(iii)] $\ess(v)=\vol(\n_v)=\|\n_v\|=\en^\vee(\d_v)$. 
\end{itemize}
\end{thm}

Since $\quotd_1\le\dd_1\le\dd_p\le\dd_\infty$ on $\cN_\R$ for $1\le p\le\infty$, this implies: 

\begin{cor}\label{cor:thmC} For any $p\in[1,\infty]$, the embeddings 
$$
(X^\lin,\dd_\infty)\hto(\cM^1,\dd_1),\quad (X^\lin,\dd_\infty)\hto(\cN_\R,\dd_p)
$$
respectively defined by $v\mapsto\d_v$ and $v\mapsto\n_v$ are bi-Lipschitz. 
\end{cor}
Note that this implies Corollary~E in the introduction.

\begin{proof}[Proof of Theorem~\ref{thm:thmC}] By Theorem~\ref{thm:MAsig}, we have $\MA(\n_v)=\d_v$, $\MA(\n_w)=\d_v$, and hence
$$
\dd_1(\d_v,\d_v)=\quotd_1(\n_v,\n_w)\le\dd_\infty(\n_v,\n_w)=\dd_\infty(v,w), 
$$
by Theorem~\ref{thm:Hauscomp} and Corollary~\ref{cor:dinftylin}. Next, note that 
\begin{equation}\label{equ:FSnv}
\FS(\n_v)(w)=\sup\{m^{-1}(v(s)-w(s))\},
\end{equation}
where $s$ runs over nonzero sections of $mL$ with $m$ sufficiently divisible. In particular, $\FS(\n_v)\ge 0$, and $\FS(\n_v)(v)=0$. Comparing~\eqref{equ:FSnv} with~\eqref{equ:dinftylin} yields 
\begin{equation}\label{equ:dinftymax}
\dd_\infty(v,w)=\max\{\FS(\n_v)(w),\FS(\n_w)(v)\}.
\end{equation}
On the other hand, for each $c\in\R$, Lemma~\ref{lem:d1andI1norm}  yields 
\begin{multline*}
\dd_1(\n_v+c,\n_w)\approx\int|\FS(\n_v)+c-\FS(\n_w)|(\d_v+\d_w)\\
=|\FS(\n_v)(w)+c|+|c-\FS(\n_w)(v)|\ge \FS(\n_v)(w)+\FS(\n_w)(v)\ge\dd_\infty(v,w). 
\end{multline*}
Thus
$$
\quotd_1(\n_v,\n_w)=\inf_{c\in\R}\dd_1(\n_v,\n_w+c)\gtrsim\dd_\infty(v,w). 
$$
This proves (i), and (ii) follows. Finally, $\MA(\n_v)=\d_v$ implies $\|\n_v\|=\en^\vee(\d_v)$, by definition of the minimum norm. Since $\FS(\n_v)$ vanishes at $v$, Proposition~\ref{prop:MAvar} further yields $\en^\vee(\d_v)=\vol(\n_v)$, which coincides with $\ess(v)$ (see Example~\ref{exam:ess}). This proves (iii). 
\end{proof}

%
%
%
%

\appendix

%
%
%
%
\section{Test configurations, integral closure, and homogenization}\label{sec:tc}
   In this appendix we revisit the correspondence between test configurations and integral norms~\cite{WN12,BHJ1}, and provide a description of homogenization in terms of integral closure. We also provide a geometric description of $\R$-test configurations, following~\cite{HL,Ino}.

%
%

\subsection{The norm associated to a test configuration}\label{sec:tc1}

A    \emph{test configuration} $(\cX,\cL)$ for $(X,L)$ consists of: a flat projective morphism $\pi\colon\cX\to\A^1$; a $\Q$-line bundle $\cL$ on $\cX$; a $\Gm$-action on $(\cX,\cL)$ that makes $\pi$ equivariant; and a $\Gm$-equivariant isomorphism 
\begin{equation}\label{equ:tctriv}
(\cX,\cL)|_{\Gm}\simeq (X,L)\times\Gm. 
\end{equation}
We denote by $z$ the coordinate on $\A^1=\Spec k[z]$ and $\Gm=\Spec k[z^\pm]$. 
\begin{exam} The \emph{trivial test configuration} $(\cX_\triv,\cL_\triv)$ is defined by $\cX_\triv=X\times\A^1$, $\cL_\triv=p_1^\star L$. 
\end{exam}
As originally pointed out in~\cite{WN12}, to any test configuration $(\cX,\cL)$ is associated an integral norm $\n_\cL\in\cN_\Z$, defined on $R_m=\Hnot(X,mL)$ for any $m\in\N$ such that $m\cL$ is a line bundle, as follows. Consider the embedding $\Hnot(\cX,m\cL)\hto R_m\otimes k[z^{\pm}]$ induced by~\eqref{equ:tctriv}. This yields a decomposition
\begin{equation}\label{equ:Reesexp}
\Hnot(\cX,m\cL)=\bigoplus_{\la\in\Z}z^{-\la}\Filt^\la R_m, 
\end{equation}
corresponding to the weight decomposition with respect to the $\Gm$-action, where 
\begin{equation}\label{equ:Reesfilt}
\Filt^\la R_m=\{s\in R_m\mid z^{-\la} s\in\Hnot(\cX,m\cL)\}
\end{equation}
is a $\Z$-filtration of $R_m$, and we define $\n_\cL$ as the associated norm.
  It is clear that $\n_{\cL+c\cX_0}=\n_\cL+c$ for any $c\in\Q$, and using flat base change, one easily checks:  

\begin{lem}\label{lem:Reesbase} If $(\cX_d,\cL_d)$ denotes the base change of $(\cX,\cL)$ with respect to $z\mapsto z^d$, $d\ge 1$, then $\n_{\cL_d}=d\n_\cL$. 
\end{lem}

In order to further analyze the norm $\n_\cL$,    recall from~\cite[\S 1.4]{trivval} that a test configuration $\cX$ is \emph{integrally closed} if $\cX$ is integrally closed in the generic fiber of $\pi$;  when $X$ is normal, this is equivalent to $\cX$ being normal. If $\cX_0$ is reduced, then $\cX$ is integrally closed. 

If $\cX$ is integrally closed, the local ring of $\cX$ at the generic point of any irreducible component $E$ of $\cX_0$ is a DVR, which defines a divisorial valuation $\ord_E$ on $\cX$; we denote by 
\begin{equation}\label{equ:bE}
b_E:=\ord_E(\cX_0)=\ord_E(z)
\end{equation}
the multiplicity of $\cX_0$ along $E$. By~\eqref{equ:tctriv} we have a function field extension $k(X)\hto k(\cX)$, and the restriction of $b_E^{-1}\ord_E$ to $k(X)$ is a divisorial valuation $v_E\in X^\div$, with values in $b_E^{-1}\Z$. Conversely, any divisorial valuation can be geometrically realized in this way. 

Recall also from~\cite[\S 2.7]{trivval} that any test configuration $(\cX,\cL)$ for $(X,L)$ determines a PL function $\f_{\cL}\in\PL(X)$, whose restriction to the dense subset $\Xdiv\subset\Xan$ is given as follows. Pick $v\in\Xdiv$, and choose an integrally closed  test configuration $\cX'$ for $X$ such that $v=v_E$ is associated to an irreducible component $E\subset\cX'_0$ and such that the canonical $\Gm$-equivariant birational maps $\mu\colon\cX'\to\cX$ and $\rho\colon\cX'\to \cX_\triv$ are morphisms. Then $\mu^\star\cL-\rho^\star\cL_\triv=D$ for a $\Q$-Cartier divisor $D$ supported on $\cX'_0$, and
\begin{equation}\label{equ:FStc}
 \f_\cL(v_E)=b_E^{-1}\ord_E(D). 
\end{equation}
  
Conversely, any PL function on $X^\an$ can be realized in this way (see~\cite[Theorem~2.31]{trivval}).

\begin{prop}\label{prop:normtc} Pick an integrally closed test configuration $(\cX,\cL)$, set $\f:=\f_\cL$, and denote by $\Sigma\subset X^\div$ the (finite) set of valuations attached to the irreducible components of $\cX_0$. Then:
\begin{itemize}
\item[(i)] $\n_\cL=\lfloor\IN_\Sigma(\f)\rfloor=\lfloor\IN(\f)\rfloor$; 
\item[(ii)] $\n_\cL^\hom=\IN_\Sigma(\f)=\IN(\f)$;
\item[(iii)] if $\cX_0$ is reduced, then $\n_\cL$ is homogeneous. 
\end{itemize}
\end{prop}
Here $\IN_\Sigma(\f)$ is defined in~\S\ref{sec:signorm}. 

\begin{lem}\label{lem:normtc} Under the above assumptions we have 
$$
\n_\cL=\n_{\mu^\star\cL},\quad\f_\cL=\f_{\mu^\star\cL}
$$
for any morphism of test configurations $\mu\colon\cX'\to\cX$.  
\end{lem}
\begin{proof} By Zariski's main theorem, we have $\mu_\star\cO_{\cX'}=\cO_\cX$ (see~\cite[Lemma~1.12]{trivval}). Combined with the projection formula, this shows $\Hnot(\cX,m\cL)=\Hnot(\cX',m\mu^\star\cL)$ for $m$ sufficiently divisible. The first point follows, while the second one holds by~\eqref{equ:FStc}. 
\end{proof}

\begin{proof}[Proof of Proposition~\ref{prop:normtc}] Pick $s\in R_m$ with $m$ sufficiently divisible and $\la\in\Z$. Then $z^{-\la} s$ determines a rational section $\sigma$ of $m\cL$ which is regular outside $\cX_0$, and hence is regular on $\cX$ iff $\ord_E(\sigma)\ge 0$ for each irreducible component $E$ of $\cX$ (see~\cite[Lemma~1.23]{trivval}). Now~\eqref{equ:FStc} implies 
$$
b_E^{-1}\ord_E(\sigma)=-\la+v_E(s)+\f(v_E),
$$
and we infer
$$
\n_\cL(s)=\max\left\{\la\in\Z\mid\la\le\min_E\{v_E(s)+\f(v_E)\}\right\}=\lfloor\IN_\Sigma(\f)(s)\rfloor.
$$
Next, pick $v\in X^\div$, and choose an integrally closed test configuration $\cX'$ that dominates $\cX$ via $\mu\colon\cX'\to\cX$ and such that $v$ lies in the corresponding set $\Sigma'\subset X^\div$. Lemma~\ref{lem:normtc} and the first step of the proof yield 
$$
\n_\cL(s)=\n_{\mu^\star\cL}(s)=\lfloor\IN_{\Sigma'}(\f)(s)\rfloor\le v(s)+\f(v). 
$$
By density of $X^\div$ and continuity of $\f$, we infer 
$$
\n_\cL(s)\le\inf_{v\in X^\div}\{v(s)+\f(v)\}=\IN(\f)(s)\le\IN_\Sigma(\f)(s). 
$$
Since $\n_\cL=\lfloor\IN_\Sigma(\f)\rfloor$, this proves (i), and (ii) follows, \cf Example~\ref{exam:homfloor}. Finally, if $\cX_0$ is reduced, then each $v\in\Sigma$ is integer valued on $k(X)^\times$. Since $\f(v)$ is rational (see~\eqref{equ:FStc}), we get 
$$
\IN_\Sigma(\f)(s)=\min_{v\in\Sigma}\{v(s)+m\f(v)\}\in\Z
$$
for $s\in R_m$ with $m$ sufficiently divisible, and (iii) now follows from (i) and (ii). 
\end{proof}

\begin{rmk}\label{rmk:Li} Proposition~\ref{prop:normtc}~(ii) implies $\n_\cL^\hom\in\cN_\Q^\Sigma$, and hence $\MA(\n_\cL)=\MA(\n^\hom_\cL)\in\cM^\Sigma$ (see Theorem~\ref{thm:MAsig}). The coefficients of this measure admit an explicit description in terms of positive intersection classes on the canonical compactification $\bar\cX\to\P^1$, see~\cite[Theorem~1.1]{Li23}.
\end{rmk}

Consider now an arbitrary test configuration $(\cX,\cL)$, and denote by $(\tcX,\tcL)$ its integral closure, \ie $\tcX\to\cX$ is the integral closure of $\cX$ in the generic fiber of $\cX\to\A^1$, and $\tcL$ is the pullback of $\cL$.

\begin{thm}\label{thm:homint} For any test configuration $(\cX,\cL)$ for $(X,L)$, we have 
$$
\n_{\cL}^\hom=\n_{\tcL}^\hom=\IN(\f_\cL)\quad\text{and}\quad\n_{\tcL}=\lfloor\n_\cL^\hom\rfloor. 
$$
\end{thm}
In other words, integral closure is the round-down of homogenization. 

\begin{lem}\label{lem:homint} We have $\n_\cL\le\n_{\tcL}\le\n^\hom_\cL$.
\end{lem}
\begin{proof} Pick $s\in R_m\smallsetminus\{0\}$ with $m$ sufficiently divisible. Since $\tcL$ is the pullback of $\cL$ to $\tcX$, $\n_\cL(s)\le\n_{\tcL}(s)=:\mu$ follows directly from~\eqref{equ:Reesfilt}. Since $\sigma:=z^{-\mu} s\in\Hnot(\tcX,m\tcL)$ is integral over $\cO_\cX$, it satisfies $\sigma^d+\sum_{i=1}^d\sigma_i\sigma^{d-i}=0$ for some $d\ge 1$ and $\sigma_i\in\Hnot(\cX,im\cL)$. By~\eqref{equ:Reesexp}, we have a Laurent expansion $\sigma_i=\sum_{\la\in\Z} \sigma_{i,\la}z^{-\la}$ with $\sigma_{i,\la}\in\Hnot(X,mL)$ such that $\n_\cL(\sigma_{i,\la})\ge\la$, and tracing the coefficient of $z^{-d\mu}$ yields $s^d+\sum_{i=1}^d \sigma_{i,i\mu} s^{d-i}=0$. Since $\n^\hom_{\cL}(\sigma_{i,i\mu})\ge\n_{\cL}(\sigma_{i,i\mu})\ge i\mu$, we infer 
$$
d\n^\hom_{\cL}(s)=\n^\hom_{\cL}(s^d)\ge\min_{1\le i\le d}\{i\mu+(d-i)\n^\hom_{\cL}(s)\}; 
$$
hence $\n^\hom_{\cL}(s)\ge\mu=\n_{\tcL}(s)$, and we are done.
\end{proof}

\begin{proof}[Proof of Theorem~\ref{thm:homint}] Lemma~\ref{lem:homint} implies 
$\n_\cL^\hom=\n_{\tcL}^\hom$, which is equal to $\IN(\f_{\tcL})=\IN(\f_\cL)$, by Proposition~\ref{prop:normtc}~(ii) and pullback invariance of $\f_\cL$. The final identity follows from Proposition~\ref{prop:normtc}. 
\end{proof}

As a consequence, we get the following geometric description of homogenization:  

\begin{cor}\label{cor:normtc} For any test configuration $(\cX,\cL)$,    $\n_\cL^\hom$ lies in $\cN_\Q^\div$,   and is equal to $d^{-1}\n_{\tcL_d}$ for any sufficiently divisible $d\in\Z_{\ge 1}$. 
\end{cor}
As above, $(\cX_d,\cL_d)$ denotes the base change of $(\cX,\cL)$ with respect to $z\mapsto z^d$, and $(\tcX_d,\tcL_d)$ is its integral closure. 

\begin{proof} By~\cite[Corollary~2.35]{trivval}, the central fiber of $(\tcX_d,\tcL_d)$ is reduced for $d$  sufficiently divisible. Then $\n^\hom_{\cL_d}=\n^\hom_{\tcL_d}=\n_{\tcL_d}$, by Theorem~\ref{thm:homint} and Proposition~\ref{prop:normtc}~(iii). By Lemma~\ref{lem:Reesbase}, we have, on the other hand, $\n_{\cL_d}=d\n_\cL$, and hence $\n^\hom_{\cL_d}=d\n^\hom_\cL$. Thus $\n_\cL^\hom=d^{-1}\n_{\tcL_d}^\hom$, which lies in $\cN_\Q^\div$, by Proposition~\ref{prop:normtc}~(ii). 
\end{proof}

\begin{rmk} Corollary~\ref{cor:normtc} can be used to provide a more elementary proof of Theorem~\ref{thm:homogft} in the case of rational norms. 
\end{rmk}

Finally, we relate (integrally closed) test configurations and (rational) divisorial norms, as follows: 

\begin{thm}\label{thm:normint} For any $\n\in\cN_\Z$, the following are equivalent:
\begin{itemize}
\item[(i)] $\n=\n_\cL$ is associated to some integrally closed test configuration $(\cX,\cL)$ for $(X,L)$; 
\item[(ii)] $\n=\lfloor\n'\rfloor$ for some $\n'\in\cN_\Q^\div$, which is then uniquely determined as $\n'=\n^\hom$.
\end{itemize}
\end{thm}
\begin{proof} That (i) implies (ii) follows from Proposition~\ref{prop:normtc}. Conversely, pick $\n'\in\cN_\Q^\div$, and set $\n:=\lfloor\n'\rfloor$, so that $\n'=\n^\hom$ (see Example~\ref{exam:homfloor}). By Theorem~\ref{thm:divPL}, we have $\n'=\IN(\f)$ for some $\f\in\PL(X)$, which can in turn be written $\f=\f_\cL$ for some integrally closed test configuration $(\cX,\cL)$. By Proposition~\ref{prop:normtc}, we then have $\n_\cL=\lfloor\n'\rfloor=\n$, which shows (ii)$\Rightarrow$(i). 
\end{proof}

%
%
\subsection{The Rees correspondence}

A test configuration $(\cX,\cL)$ is \emph{ample} if $\cL$ is ample. For $d$ sufficiently divisible, the graded $k[z]$-algebra 
$R(\cX,d\cL)=\bigoplus_{m\in\N}\Hnot(\cX,md\cL)$ is then generated in degree $1$, which shows that $\n=\n_\cL\in\cT_\Z$ is of finite type. Note further that 
\begin{equation}\label{equ:centralproj}
R(\cX_0,d\cL)\simeq\gr_\n R^{(d)}
\end{equation}
for $d$ sufficiently divisible. Thus $(\cX_0,\cL_0)$ can be identified with the central fiber of $\n$ (see~\eqref{equ:centralfiber}). 

Denoting by $\cT$ the set of ample test configurations, $(\cX,\cL)\mapsto\f_\cL$ yields a map $\cT\to\cT_\Z$. A map in the reverse direction is indeed provided by the \emph{Rees construction}. Given $\n\in\cT_\Z$, pick $d\ge \ge1$ such that $\n$ is represented by an integral norm on $R^{(d)}=R(X,dL)$ generated in degree $1$, with associated filtration $(\Filt^\la R^{(d)})_{\la\in\Z}$. The \emph{Rees algebra}
$$ 
\cR:=\bigoplus_{\la\in\Z}z^{-\la}\Filt^{\la}R^{(d)}
$$
is a graded $k[z]$-algebra (with respect to the $\N$-grading inherited from that of $R^{(d)}$), generated in degree $1$, and we set $\cX:=\Proj_{k[z]}\cR$ and $\cL=d^{-1}\cO_{\cX}(1)$. This yields a map $\cT_\Z\to\cT$ which is an inverse of the previous one (see~\cite[Proposition 2.15]{BHJ1}). We shall refer to the 1--1 map 
\begin{equation}\label{equ:Rees}
\cT\simeq\cT_\Z
\end{equation}
so defined as the \emph{Rees correspondence}.  

By~\eqref{equ:centralproj}, the central fiber $\cX_0$ of an ample test configuration $(\cX,\cL)$ is reduced iff $\gr_{\n_\cL} R^{(d)}$ is reduced for $d$ sufficiently divisible, which holds iff $\n_\cL$ is homogeneous    (\ie the converse of Proposition~\ref{prop:normtc}~(iii) holds for ample test configurations).    The Rees correspondence thus induces a bijection between $\cT_\Z^\hom$ and the set of ample test configurations with reduced central fiber. 

Denote by $\cT^\inte\subset\cT$ the set of ample integrally closed test configurations, and by $\cT^\inte_\Z\subset\cT_\Z$ its image under the Rees correspondence. Any test configuration with reduced central fiber is integrally closed, and hence 
$$
\cT_\Z^\hom\subset\cT_\Z^\inte\subset\cT_\Z.
$$

\begin{thm}\label{thm:inthom} Homogenization induces a bijection $\cT_\Z^\inte\simto\cT_\Q^\hom$, with inverse provided by round-down. 
\end{thm}

\begin{proof} Pick $\n\in\cT_\Z^\inte$. Then $\n^\hom\in\cT_\Q^\hom$ (see Lemma~\ref{lem:homogft}, or Corollary~\ref{cor:normtc}), and $\n=\lfloor\n^\hom\rfloor$ (see Proposition~\ref{prop:normtc}). Conversely, pick $\n'\in\cT_\Q^\hom$. By Corollary~\ref{cor:FSbij}, $\n'=\n^\hom$ for some $\n\in\cT_\Z$, \ie $\n=\n_\cL$ for some ample test configuration $(\cX,\cL)$. After passing to the integral closure $(\tcX,\tcL)$ (which remains ample, since $\tcX\to\cX$ is finite), we may further assume that $(\cX,\cL)$ is integrally closed (see Theorem~\ref{thm:homint}), and hence $\n\in\cT_\Z^\inte$. Then $\n=\lfloor\n'\rfloor$, by Proposition~\ref{prop:normtc} again, and $\n'=\n^\hom$, which completes the proof. 
\end{proof}

Finally, we note:

\begin{lem}\label{lem:FStest}
 For any ample test configuration $(\cX,\cL)$ for $(X,L)$ we have 
 $$
 \f_\cL=\FS(\n_\cL)=\FS(\n^\hom_\cL). 
 $$
\end{lem} 

\begin{proof}    The second equality follows from Proposition~\ref{prop:FSH}. Since $\n_\cL^\hom=\IN(\f_\cL)$ (see Theorem~\ref{thm:homint}) and $\f_\cL\in\cH_\Q$, Proposition~\ref{prop:FSIN} further yields $\FS(\n_\cL^\hom)=\qq(\f_\cL)=\f_\cL$, which completes the proof.   
\end{proof}

Combining Theorem~\ref{thm:inthom} with the bijection $\FS\colon\cT_\Q^\hom\simto\cH_\Q$, we thus recover~\cite[Corollary~2.32]{trivval}: 

\begin{cor} The map $(\cX,\cL)\mapsto\f_\cL$ restricts to a bijection $\cT^\inte\simto\cH_\Q$. 
\end{cor}

%
%
\subsection{The case of higher rank}\label{sec:Inoue}
Following~\cite[\S 2.2]{HL} and~\cite[\S2.2]{Ino}, we briefly discuss a version of the Rees correspondence for $\R$-test configurations. 

\begin{defi} For any $r\in\N$, we define a \emph{rank $r$ test configuration $(\cX,\cL,\xi)$ for $(X,L)$} as the following data:
\begin{itemize}
\item a flat projective scheme morphism $\pi\colon\cX\to\A^r$; 
\item a $\Q$-line bundle $\cL$ on $\cX$;
\item a $\Gm^r$-action on $(\cX,\cL)$ that makes $\pi$ equivariant (with respect to the standard action on $\A^r$); 
\item a $\Gm^r$-equivariant isomorphism $(\cX,\cL)|_{\Gm^r}\simeq (X,L)\times\Gm^r$; 
\item a vector $\xi\in\R_+^r$ with $\Q$-linearly independent components.
\end{itemize}
\end{defi}
A usual test configuration as in~\S\ref{sec:tc1} is thus a rank $1$ test configuration, up to the scaling factor $\xi\in\R_{>0}$. 

Denote by $z_1,\dots,z_r$ the coordinates on $\A^r=\Spec k[z_1,\dots,z_r]$ and $\Gm^r=\Spec k[z_1^\pm,\dots,z_r^\pm]$. The above data yields an embedding 
$$
\Hnot(\cX,m\cL)\hto R_m[z_1^\pm,\dots,z_r^\pm]
$$
for $m$ sufficiently divisible, and we define a norm $\n_{\cL,\xi}\in\cN_\R$ by setting 
$$
\n_{\cL,\xi}(s)=\max\left\{\langle\xi,\a\rangle\mid \a\in\Z^r,\,z^{-\a} s\in\Hnot(\cX,m\cL)\right\}
$$
for $s\in R_m$, where $z^\a:=\prod_i z_i^{\a_i}$. Note that $\n_{\cL,\xi}\in\cN_\La$ with $\La:=\sum_i\Z\xi_i\simeq\Z^r$. 

\begin{prop} For any rank $r$ test configuration $(\cX,\cL,\xi)$ with $\cL$ relatively ample, the associated norm $\n_{\cL,\xi}$ is of finite type; this norm is further of rank $r$, and its central fiber can be identified with the fiber $(\cX_0,\cL_0)$ of $\pi$ over $0\in \A^r$. Conversely, any $\R$-test configuration $\n\in\cT_\R$ arises in this way. 
\end{prop}
\begin{proof} We sketch the argument, and refer to~\cite[Proposition~2.20]{Ino} for details. Assume $\cL$ is relatively ample, and set $\n=\n_{\cL,\xi}$. The restriction map $R(\cX,d\cL)\to R(\cX_0,d\cL_0)$ is surjective for $d$ sufficiently divisible, and one checks that it induces an isomorphism $\gr_\n R^{(d)}\simeq R(\cX_0,d\cL_0)$ as $\N\times\Z^r$-graded algebras. The rest easily follows. 

Conversely, pick $\n\in\cT_\R$, of rank $r$. As in Example~\ref{exam:HL}, one can find an embedding $X\hto\P^N$ such that $\cO(1)|_X=dL$, an action of a torus $T\simeq\Gm^r$ on $(\P^N,\cO(1))$ and $\xi\in N_\R\simeq\R^r$, such that the induced norm on $R(\P^N,\cO(1))$ restricts to $\n$. Acting on $X$ defines a $T$-equivariant morphism $T\to\mathrm{Hilb}$ to the Hilbert scheme of $\P^N$. Pick a regular top-dimensional cone $\sigma\subset N_\R$ that contains $\xi$ in its interior, and denote by $B\simeq\A^r$ the corresponding toric affine variety. After passing to a finer cone, one may assume, by toric resolution of singularities, that the corresponding $T$-equivariant rational map $B\dashrightarrow\mathrm{Hilb}$ is a morphism, and pulling back the universal family yields the desired polarized family $(\cX,\cL)$. 
\end{proof}

%
%
%
\section{The toric case}\label{sec:toric}

We give a brief account of how some of the main results in the paper specialize to the toric setting~\cite{FultonToric,BPS}. See also Appendix~B in~\cite{trivval}.

Consider an algebraic torus $T\simeq\Gm^n$, with associated dual lattices $M:=\Hom(T,\Gm)$ and $N:=\Hom(\Gm,T)$. We have a canonical embedding $M\hookrightarrow k(T)^\times$ given by $\a\mapsto z^\a$ onto the set $T$-invariant functions, and a dual canonical embedding $N_\R\hookrightarrow T^\val$ given by $\xi\mapsto v_\xi$ onto the set of $T(k)$-invariant valuations, such that  $v_\xi(z^\a)=\langle \xi,\a\rangle$ for all $\xi\in N_\R$ and $\a\in M\hookrightarrow k(T)^\times$.

A polarized toric variety $(X,L)$ is determined by a rational polytope $P\subset M_\R$, such that, for each $m$ sufficiently divisible, the set of weights $\a\in M$ of the $T(k)$-module $R_m=\Hnot(X,mL)$ coincides with $mP\cap M$, each with multiplicity $1$. Denoting by $\hat P:=\R_{\ge 0}(\{1\}\times P)\subset\R\times M_\R$ the (rational polyhedral) cone over $P$, this yields, for $d$ sufficiently divisible, a 1--1 correspondence between:
\begin{itemize}
\item[(a)] the set of toric (\ie $T(k)$-invariant) norms $\n$ on $R^{(d)}=R(X,dL)$ and superadditive functions $h\colon\Ga^{(d)}\to\R$ on the semigroup $\Ga^{(d)}:=(d\N\times M)\cap\hat P$ such that $h(m,\a)=O(m)$; 
\item[(b)] the subset of toric homogeneous norms $\n$ and concave, bounded functions $g\colon P\to\R$, the corresponding superadditive function on $\Ga^{(d)}$ being $h(m,\a)= m g(m^{-1}\a)$.
\end{itemize}
\begin{exam} Each $\xi\in N_\R$ determines a toric homogeneous norm $\n_{v_\xi}$, with associated function $g(\a)=\langle\xi,\a\rangle+c$ for $c\in\R$ such that $\inf_P g=0$, \ie $c=-\inf_{\a\in P}\langle\xi,\a\rangle$. 
\end{exam}
A function $g$ as in (b) above is automatically lsc on $P$ (see~\cite{GKR}), but might be discontinuous at some boundary points. Denote by $g^\vee\colon N_\R\to\R$ its (convex) Legendre transform, defined by 
$$
g^\vee(\xi)=\sup_{\a\in P}\{\langle\a,\xi\rangle+g(\a)\},
$$
and let also $\la_P$ be the Lebesgue measure of $P$, normalized to mass $1$. Then: 
\begin{itemize}
\item[(i)] $\FS(\n)|_{N_\R}=g^\vee-0^\vee$, where $0^\vee$ coincides with the support function of $P$; 
\item[(ii)] $\vol(\n)=\int g\,\la_P$; 
\item[(iii)] $\dd_\infty(\n,\n')=\sup_P|g-g'|$, and $\dd_p(\n,\n')=\|g-g'\|_{L^p(\la_P)}$ for $p\in [1,\infty)$; 
\item[(iv)] $\n\in\cT_\R$ (resp.~$\cT_\Q$) iff $g^\vee(\xi)=\max_i\{\langle\xi,\a_i\rangle+\la_i\}$ for a finite set $\a_i\in P\cap M_\Q$ and $\la_i\in\R$ (resp.~$\Q$); 
\item[(v)] $\n\in\cN_\R^\cont\Longleftrightarrow g\in\Cz(P)\Longleftrightarrow g\text{ usc }\Longleftrightarrow\n\in\cN_\R^{\max}$; 
\item[(vi)] $\n\in\cN_\R^\div$ (resp.~$\cN_\Q^\div$) iff $g(\a)=\min_j\{\langle \xi_j,\a\rangle+c_j\}$ for a finite set $\xi_j\in N_\Q$ and $c_j\in\R$ (resp.~$\Q$); 
\item[(vii)] $\MA(\n)=\MA_\R(g^\vee)=(\nabla g)_\star\la_P$, where $\MA_\R$ is the real Monge--Amp\`ere operator and $\nabla g$ is the ($\la_P$-a.e.~defined) gradient of $g$. 
\end{itemize}
By (iv), (vi) and basic convex geometry, it follows that any toric homogeneous norm $\n$ satisfies 
$$
\n\in\cN_\Q^\div\Longleftrightarrow\n\in\cT_\Q.
$$
However, both implications fail when $\Q$ is replaced with $\R$. 
\begin{exam}\label{exam:divnonft} Assume $(X,L)=(\P^1,\cO(1))$, and consider the toric divisorial norm 
$$
\n:=\min\{\n_v,\n_\triv+c\}\in\cN_\R^\div, 
$$
where $v=\ord_0$ with $0$ is the origin in $\A^1\subset\P^1$, and $c\in [0,1]$. The corresponding concave function is $g(\a)=\min\{\a,c\}$ with $\a\in P=[0,1]\subset M_\R=\R$, and a simple computation yields 
$$
g^\vee(\xi)=\max\{\xi+c,c\xi+c,0\}
$$
for $v\in N_\R=\R$. Using (iv), this shows $\n\in\cT_\R\Longleftrightarrow c\in\Q$. 
\end{exam}

\begin{exam}\label{exam:ftnondiv} For any $\xi\in N_\R\subset X^\lin$, $\n_\xi$ is of finite type (see Example~\ref{exam:torusaction}). However, $\n_\xi$ is divisorial iff $\xi\in N_\Q$ (see Example~\ref{exam:linnondiv}). 
\end{exam}

%
%
%
%

%
%
%
%


\begin{thebibliography}{BGJKM20}

\bibitem[Berk90]{BerkBook}
V.~G.~Berkovich.
\newblock \emph{Spectral theory and analytic
geometry over non-Archi\-medean fields.}
\newblock  Mathematical Surveys and Monographs, 33. 
\newblock American Mathematical Society, Providence, RI, 1990.

 \bibitem[Berk93]{BerkIHES}
 V.~G.~Berkovich.
 \newblock \emph{\'Etale cohomology for non-Archimedean analytic spaces}.
 \newblock Publ. Math. Inst. Hautes \'Etudes Sci. \textbf{78} (1993), 5--161. 

 \bibitem[BBGZ13]{BBGZ}
 R.~J.~Berman, S.~Boucksom, V.~Guedj and A.~Zeriahi.
 \newblock \emph{A variational approach to complex Monge--Amp\`ere equations}.
 \newblock Publ. Math. Inst. Hautes {\'E}tudes Sci. 117 (2013), 179--245.

 \bibitem[BBJ21]{YTD}
 R.~J.~Berman, S.~Boucksom and M.~Jonsson.
 \newblock \emph{A variational approach to the Yau--Tian--Donaldson conjecture}.
 \newblock J. Amer. Math. Soc. \textbf{34} (2021),  605--652.

\bibitem[BlJ20]{BlJ}
H.~Blum, M.~Jonsson. 
\newblock \emph{Thresholds, valuations, and K-stability}.
\newblock Adv. Math. \textbf{365} (2020), 107062.

\bibitem[BLQ22]{BLQ} 
H.~Blum, Y.~Liu, L.~Qi. 
\newblock \emph{Convexity of multiplicities of filtrations on local rings}. 
\newblock \texttt{arXiv:2208.04902}. 

\bibitem[BLX22]{BLX22} 
H.~Blum, Y.~Liu, C.~Xu. 
\newblock\emph{Openness of K-semistability for Fano varieties}. 
\newblock Duke Math. J. \textbf{171} (2022), 2753--2797.

\bibitem[BLXZ23]{BLXZ} 
H.~Blum, Y.~Liu, C.~Xu, Z.~Zhuang. 
\newblock\emph{The existence of the K\"ahler--Ricci soliton degeneration}.
\newblock Forum Math. Pi \textbf{11} (2023), Paper No.\ e9, 28 pp.


\bibitem[BLZ22]{BLZ} 
H.~Blum, Y.~Liu, C.~Zhou 
\newblock\emph{Optimal destabilization of K-unstable Fano varieties via stability thresholds}. 
\newblock Geom. Topol. \textbf{26} (2022), 2507--2564.


\bibitem[BX20]{BX20} 
H.~Blum, C.~Xu.  
\newblock\emph{Uniqueness of K-polystable degenerations of Fano varieties}. 
\newblock Ann.\ of Math. (2) {\bf 190} (2020), 609--656.

\bibitem[BGR]{BGR}
S.~Bosch, U.~G\"untzer and R.~Remmert.
\newblock\emph{Non-Archimedean Analysis}.
\newblock Springer-Verlag, Berlin, Heidelberg, 1994.

\bibitem[Bou14]{Bou14} 
S.~Boucksom.
\newblock\emph{Corps d'Okounkov}.
\newblock  Exp. No. 1059.
\newblock Ast{\'e}risque \textbf{361} (2014), 1--41.

\bibitem[BC11]{BC} 
S.~Boucksom and H.~Chen.
\newblock\emph{Okounkov bodies of filtered linear series}.
\newblock Compos. Math. \textbf{147}  (2011), 1205--1229.

\bibitem[BE21]{BE} 
S.~Boucksom, D.~Eriksson.
\newblock \emph{Spaces of norms, determinant of cohomology and Fekete points in 
  non-Archimedean geometry}.
\newblock Adv. Math. \textbf{378} (2021), 107501.

\bibitem[BFJ16]{siminag}
S.~Boucksom, C.~Favre and M.~Jonsson.
\newblock \emph{Singular semipositive metrics in non-Archimedean geometry}.
\newblock J. Algebraic Geom. \textbf{25} (2016), 77--139. 

\bibitem[BGM22]{BGM}
S.~Boucksom, W.~Gubler, F.~Martin.
\newblock\emph{Differentiability of relative volumes over an arbitrary non-Archimedean field}.
\newblock Int. Math. Res. Not., Volume 2022, Issue 8, 6214--6242.

\bibitem[BHJ17]{BHJ1}
S. Boucksom, T.~Hisamoto and M. Jonsson.
\newblock\emph{Uniform K-stability, Duistermaat--Heckman measures 
  and singularities of pairs}.
\newblock Ann. Inst. Fourier \textbf{67} (2017), 743--841.

\bibitem[BoJ22a]{trivval}
S.~Boucksom, M.~Jonsson. 
\newblock \emph{Global pluripotential theory over a trivially valued field}.
\newblock Ann. Fac. Sci. Toulouse. Math. \textbf{31} (2022), 647--836.

\bibitem[BoJ22b]{trivadd}
S.~Boucksom and M.~Jonsson. 
\newblock \emph{Addendum to: Global pluripotential theory over a trivially valued field}.
\newblock To appear in Ann. Fac. Sci. Toulouse. Math.
\newblock \texttt{arXiv:2206.07183}. 

\bibitem[BoJ22c]{nakstab2}
S.~Boucksom and M.~Jonsson. 
\newblock\emph{A non-Archimedean approach to K-stability, II: divisorial stability and openness}. 
\newblock To appear in J. Reine Angew. Math.
\newblock \texttt{arXiv:2206.09492}.

\bibitem[BoJ23]{nagreen}
S.~Boucksom and M.~Jonsson. 
\newblock\emph{Non-Archimedean Green's functions and Zariski
  decompositions}.
\newblock \texttt{arXiv:2304.00144}.

\bibitem[BKMS15]{BKMS}
S.~Boucksom, A.~K\"uronya, C.~Maclean and T.~Sz\'emberg.
\newblock\emph{Vanishing sequences and Okounkov bodies}.
\newblock Math. Ann. \textbf{361} (2015), 811--834.

\bibitem[BT72]{BT72}
F.~Bruhat and J.~Tits.
\newblock\emph{Groupes r\'eductifs sur un corps local}.
\newblock Inst. Hautes \'Etudes Sci. Publ. Math. \textbf{41} (1972), 5--251. 

\bibitem[BPS11]{BPS} 
J.~I.~Burgos Gil, P.~Philippon and M.~Sombra.
\newblock \emph{Arithmetic geometry of toric varieties. Metrics, measures, and heights}.
\newblock  Ast\'erisque No. 360 (2014).

\bibitem[Cha06]{CL06} 
A.~Chambert-Loir.
\newblock \emph{Mesures et {\'e}quidistribution
 sur les espaces de Berkovich}.
\newblock  J. Reine Angew. Math. \textbf{595} (2006), 215--235.

\bibitem[CD12]{CLD} 
 A.~Chambert-Loir and A.~Ducros.
\newblock \emph{Formes diff{\'e}rentielles r{\'e}elles et courants 
  sur les espaces de Berkovich}.
\newblock \texttt{arXiv:1204.6277}.

\bibitem[CM15]{CM15} 
H.~Chen and C.~Maclean.
\newblock\emph{Distribution of logarithmic spectra of the equilibrium energy}.
\newblock Manuscripta Math. \textbf{146}  (2015), 365--394.

\bibitem[CM18]{CM18} 
H.~Chen and A.~Moriwaki.
\newblock\emph{Extension property of semipositive invertible sheaves over a non-archimedean field}.
\newblock  Ann. Sc. Norm. Super. Pisa Cl. Sci. (5) \textbf{18} (2018), 241--282. 

\bibitem[CSW18]{CSW}
X.X.~Chen, S.~Sun and B.~Wang.
\newblock\emph{K\"ahler-Ricci flow, K\"ahler--Einstein metric, and K-stability}.
\newblock  Geom. Topol. \textbf{22} (2018), 3145--3173.

\bibitem[Cod19]{Cod19}
G.~Codogni.
\newblock\emph{Tits buildings and K-stability}.
\newblock  Proc. Edinb. Math. Soc., \textbf{62} (2019), 799--815.

\bibitem[Dar15]{Dar15} 
T.~Darvas.
\newblock\emph{The Mabuchi completion of the space of K\"ahler potentials}.
\newblock Adv. Math. \textbf{285} (2015), 182--219.

\bibitem[Dar17]{Dar17} 
T.~Darvas.
\newblock\emph{The Mabuchi geometry of finite energy classes}.
\newblock Amer. J. Math. \textbf{139} (2017), 1275--1313.

\bibitem[DDL18]{DDL18}  
T.~Darvas, E.~Di Nezza and C.~H.~Lu.
\newblock \emph{$L^1$ metric geometry of big cohomology classes}.
\newblock Ann. Inst. Fourier \textbf{68} (2018), 3053--3086.

\bibitem[DLR20]{DLR20}  
T.~Darvas, C.~H.~Lu and Y.~A.~Rubinstein.
\newblock \emph{Quantization in geometric pluripotential theory}.
\newblock Comm. Pure Appl. Math. \textbf{73} (2020), 1100--1138.

\bibitem[DX22]{DX}
T.~Darvas, M.~Xia. 
\newblock \emph{The closures of test configurations and algebraic singularity types}. 
\newblock Adv. Math. \textbf{397} (2022), Paper No. 108198, 56 pp.

\bibitem[Der16]{Der} 
R.~Dervan. 
\newblock\emph{Uniform stability of twisted constant scalar curvature K\"ahler functions}. 
\newblock Int. Math. Res. Not. IMRN (2016) \textbf{15}, 4728--4783.

\bibitem[DeL23]{DerLeg}
R.~Dervan, E.~Legendre.
\newblock\emph{Valuative stability of polarised varieties}.
\newblock Math. Ann. \textbf{385} (2023), 357--391.

\bibitem[DeSz20]{DS}
R.~Dervan, G.~Sz\'ekelyhidi. 
\newblock\emph{The K\"ahler-Ricci flow and optimal degenerations}.
\newblock J. Differential Geom. {\bf 116} (2020), 187--203. 

\bibitem[Don02]{Dontoric}
S. K. Donaldson. 
\newblock \emph{Scalar curvature and stability of toric varieties. }
\newblock  J. Differential Geom. \textbf{62} (2002), no. 2, 289--349.  

\bibitem[Duc09]{Duc09}
A.~Ducros. 
\newblock\emph{Les espaces de Berkovich sont excellents}. 
\newblock Ann. Inst. Fourier {\bf 59} (2009), 1443--1552.

 \bibitem[ELS03]{ELS}
 L.~Ein, R.~Lazarsfeld and K.~Smith.
 \newblock \emph{Uniform approximation of Abhyankar valuation ideals in smooth
 function fields}.
 \newblock  Amer. J. Math. \textbf{125} (2003), 409--440.

\bibitem[Fan19]{Fan19}
Y.~Fang.
\newblock\emph{Study of positively metrized line bundles over a non-Archimedean field via holomorphic convexity}.
\newblock Ph.D Thesis, University of Paris, 2019.

\bibitem[Fuj19a]{Fujval}
K.~Fujita.
\newblock\emph{A valuative criterion for uniform K-stability of $\Q$-Fano varieties}.
\newblock J. Reine Angew. Math. \textbf{751} (2019), 309--338.

 \bibitem[Fuj19b]{Fujplt}
 K.~Fujita.
 \newblock\emph{Uniform K-stability and plt blowups of log Fano pairs}.
 \newblock Kyoto J. Math. \textbf{59} (2019), 399--418.

\bibitem[Ful93]{FultonToric}
W.~Fulton.
\newblock \emph{Introduction to toric varieties}.
\newblock Annals of Mathematics Studies, 131. 
\newblock Princeton University Press, Princeton, NJ, 1993.

\bibitem[GKR68]{GKR} 
D.~Gale,V.~Klee, R.T.~Rockafellar. 
\newblock\emph{Convex functions on convex polytopes}. 
\newblock Proc. Amer. Math. Soc. \textbf{19} (1968), 867--873.

\bibitem[Gub98]{GublerLocal}
W.~Gubler.
\newblock \emph{Local heights of subvarieties over non-Archimedean fields}.
\newblock J. Reine Angew. Math. \textbf{498} (1998), 61--113.


\bibitem[GJKM19]{GJKM}
W.~Gubler, P.~Jell, K.~K\"unnemann and M. Sombra. 
\newblock \emph{Continuity of plurisubharmonic envelopes in non-Archimedean geometry and test ideals (with an appendix by Jos\'e Ignacio Burgos Gil and Martin Sombra).}
\newblock Ann. Inst. Fourier \textbf{69} (2019), 2331--2376.

\bibitem[HL20]{HL} 
J.~Han and C.~Li. 
\newblock{Algebraic uniqueness of K\"ahler--Ricci flow limits and optimal degenerations of Fano varieties}.
\texttt{arXiv:2009.01010}.

\bibitem[Ino22]{Ino}
E.~Inoue. 
\newblock{Entropies in $\mu$-framework of canonical metrics and K-stability, II -- Non-archimedean aspect: non-archimedean $\mu$-entropy and $\mu$K-semistability}. 
\newblock \texttt{arXiv:2202.12168}. 

\bibitem[KK12]{KK} K.~Kaveh, A.G.~Khovanskii.
\newblock\emph{Newton--Okounkov bodies, semigroups of integral points, graded algebras and intersection theory}.
\newblock Ann.\ of Math. (2) {\bf 176} (2012), no. 2, 925--978.

 \bibitem[K\"ur03]{Kur03}
 A.~K\"uronya.
 \newblock\emph{A divisorial valuation with irrational volume}.
 \newblock J.~Algebra \textbf{262} (2003), 413--423.

\bibitem[LM09]{LM09}
R.~Lazarsfeld and M.~Musta\c{t}\v{a}. 
\newblock\emph{Convex bodies associated to linear series}. 
\newblock Ann. Sci. {\'E}c. Norm. Sup{\'e}r. (4), \textbf{42} (2009), 783--835.

\bibitem[Li17]{LiEquivariant}
C.~Li.
\newblock\emph{K-semistability is equivariant volume minimization}.
\newblock Duke Math. J. \textbf{166} (2017), 3147--3218.

\bibitem[Li22a]{Li22a} 
 C.~Li. 
 \newblock\emph{G-uniform stability and K\"ahler-Einstein metrics on singular Fano varieties}.
\newblock Invent. Math. \textbf{227} (2022), 661--744.

 \bibitem[Li22b]{Li22b} 
 C.~Li. 
 \newblock\emph{Geodesic rays and stability in the cscK problem}.
 \newblock \newblock Ann. Sci. {\'E}c. Norm. Sup{\'e}r. \textbf{55} (2022), 1529--1574.
 
 \bibitem[Li23]{Li23} 
 C.~Li. 
 \newblock\emph{K-stability and Fujita approximation}.
 \newblock In \emph{Birational geometry, K\"ahler--Einstein metrics and degenerations}, 545--566.
\newblock Springer Proc. Math. Stat., \textbf{409}. Springer, 2023.
\newblock \texttt{arXiv:2102.09457}.
 
\bibitem[LX20]{LX20}
C.~Li and C.~Xu.
\newblock\emph{Stability of valuations and Koll\'ar components}.
\newblock J. Eur. Math. Soc. \textbf{22} (2020), 2573--2627.

\bibitem[Liu23]{LiuYa}
Yaxiong Liu. 
\newblock\emph{Openness of uniformly valuative stability on the K\"ahler cone of projective manifolds}. 
\newblock Math. Z. \textbf{303}(2023), no.2, Paper No. 52, 26 pp.




\bibitem[MR15]{MR15} 
D.~McKinnon and M.~Roth. 
\newblock \emph{Seshadri constants, diophantine approximation, and Roth's theorem for arbitrary varieties}.
\newblock Invent. Math. \textbf{200} (2015), 513--583.

\bibitem[MFK]{GIT}
D.~Mumford, J.~Fogarty and F.~Kiwan. 
\newblock \emph{Geometric invariant theory}. 
\newblock Third edition. Ergebnisse der Mathematik und ihrer Grenzgebiete (2), vol 34.
\newblock Springer-Verlag, Berlin, 1994.

\bibitem[Oda15]{Oda15} 
Y.~Odaka.
\newblock\emph{On parametrization, optimization and triviality of test configurations}. 
\newblock Proc. Amer. Math. Soc. \textbf{143}, 25--33. 

 \bibitem[Poi13]{Poi}
 J.~Poineau.
 \newblock\emph{Les espaces de Berkovich sont ang\'eliques}.
 \newblock Bull. Soc.\ Math. France {\bf 141}, 2013, 267--297.

\bibitem[Reb21]{Reb21}
R.~Reboulet.
\newblock\emph{The asymptotic Fubini--Study operator over general non-Archimedean fields}.
\newblock Math. Z. \textbf{299} (2021), 2341--2378.

\bibitem[Reb22]{Reb22}
R.~Reboulet.
\newblock\emph{Plurisubharmonic geodesics in spaces of non-Archimedean metrics of finite energy}.
\newblock J. Reine Angew. Math. \textbf{793} (2022), 59--103.

\bibitem[Sz\'e15]{Sze}
G.~Sz\'ekelyhidi. 
\newblock \emph{Filtrations and test-configurations}. 
\newblock With an appendix by S.~Boucksom.
\newblock Math. Ann. \textbf{362} (2015), 451--484. 

 \bibitem[Tem04]{TemkinLocalII}
 M.~Temkin.
 \newblock\emph{On local properties of non-Archimedean analytic spaces II}.
 \newblock Israel J. Math. \textbf{140} (2004), 1--27.

 \bibitem[Tem15]{TemkinSurvey}
 M.~Temkin.
 \newblock \emph{Introduction to Berkovich analytic spaces}.
 \newblock Berkovich spaces and applications, 3--66.
 \newblock Lecture Notes in Mathematics, 2119. 
 \newblock Springer, 2015.

\bibitem[WN12]{WN12}
D.~Witt Nystr\"om.
\newblock\emph{Test configurations and Okounkov bodies}.
\newblock Compos. Math. \textbf{148} (2012), 1736--1756.

\bibitem[Zha22]{Zha22}
K. Zhang.
\newblock\emph{Valuative invariants with higher moments}.
\newblock J. Geom. Anal. \textbf{32} (2022), no.1, Paper No. 10, 33 pp.

\bibitem[Zha95]{Zha95}
S.-W Zhang.
\newblock\emph{Positive line bundles on arithmetic varieties}.
\newblock  J. Amer. Math. Soc. \textbf{8} (1995), 187--221.

\end{thebibliography}
\end{document}